\documentclass[11pt]{amsart}

\usepackage{amsmath}
\usepackage{amssymb}
\usepackage{extpfeil}
\usepackage[normalem]{ulem}
\usepackage[all]{xy}
\usepackage{longtable}
\usepackage{titletoc}
\usepackage{mathrsfs}
\usepackage{mathtools}
\usepackage{amscd}
\usepackage{appendix}
\usepackage{amsthm}
\usepackage{amsfonts}
\usepackage{tikz-cd}
\usepackage{array}
\usepackage{graphicx}
\usepackage{hyperref}
\usepackage{longtable}
\usepackage{xtab}
\usepackage{soul,xcolor}
\usepackage{upgreek}
\usepackage[total={15.875truecm,22.86truecm},centering]{geometry}
\usepackage{stackrel}
\usepackage{eucal}
\usepackage{bbm}
\usepackage[outline]{contour}
\usepackage{scalerel}
\usepackage{color}

\newtheorem{theorem}[equation]{Theorem}
\newtheorem{lemma}[equation]{Lemma}
\newtheorem{proposition}[equation]{Proposition}
\newtheorem{corollary}[equation]{Corollary}

\theoremstyle{definition}
\newtheorem{definition}[equation]{Definition}
\newtheorem{example}[equation]{Example}

\theoremstyle{remark}
\newtheorem{remark}[equation]{Remark}
\newtheorem{Subsubsec}[equation]{}

\numberwithin{equation}{subsection}

\DeclareMathAlphabet{\mathpzc}{OT1}{pzc}{m}{n}
\DeclareMathAlphabet{\matheur}{U}{eur}{m}{n}

\definecolor{lava}{RGB}{207,16,32}
\definecolor{purple}{RGB}{148,0,211}

\newcommand{\FF}{\mathbb{F}}
\newcommand{\ZZ}{\mathbb{Z}}
\newcommand{\QQ}{\mathbb{Q}}
\newcommand{\RR}{\mathbb{R}}

\newcommand{\TT}{\mathbb{T}}
\newcommand{\GG}{\mathbb{G}}

\newcommand{\CC}{\mathbb{C}}
\newcommand{\OO}{\mathbb{O}}
\newcommand{\PP}{\mathbb{P}}
\newcommand{\KK}{\mathbb{K}}

\newcommand{\NN}{\mathbb{N}}
\newcommand{\MM}{\mathbb{M}}

\newcommand{\ba}{\mathbf{a}}

\newcommand{\bK}{\mathbf{K}}

\newcommand{\bX}{\mathbf{X}}

\newcommand{\bU}{\mathbf{U}}

\newcommand{\bochi}{\boldsymbol{\chi}}
\newcommand{\boy}{\boldsymbol{y}}
\newcommand{\boxx}{\boldsymbol{x}}

\newcommand{\afk}{\mathfrak{a}}

\newcommand{\cfk}{\mathfrak{c}}
\newcommand{\dfk}{\mathfrak{d}}
\newcommand{\mfk}{\mathfrak{m}}
\newcommand{\nfk}{\mathfrak{n}}
\newcommand{\pfk}{\mathfrak{p}}
\newcommand{\qfk}{\mathfrak{q}}

\newcommand{\Afk}{\mathfrak{A}}

\newcommand{\Ifk}{\mathfrak{I}}

\newcommand{\Pfk}{\mathfrak{P}}

\newcommand{\Fcal}{\CMcal{F}}

\newcommand{\Hcal}{\CMcal{H}}

\newcommand{\Ocal}{\CMcal{O}}
\newcommand{\Pcal}{\CMcal{P}}
\newcommand{\Rcal}{\CMcal{R}}

\newcommand{\Ucal}{\CMcal{U}}

\newcommand{\Ccal}{\CMcal{C}}

\newcommand{\Acal}{\CMcal{A}}

\newcommand{\Mscr}{\mathscr{M}}

\newcommand{\Pscr}{\mathscr{P}}

\newcommand{\Sscr}{\mathscr{S}}

\newcommand{\eA}{\matheur{A}}

\newcommand{\eG}{\matheur{G}}
\newcommand{\eH}{\matheur{H}}
\newcommand{\eK}{\matheur{K}}

\newcommand{\eM}{\matheur{M}}

\newcommand{\eF}{\matheur{F}}

\newcommand{\ek}{\matheur{k}}

\DeclareMathOperator{\Aut}{Aut}
\DeclareMathOperator{\Lie}{Lie}

\DeclareMathOperator{\GL}{GL}
\DeclareMathOperator{\Mat}{Mat}

\DeclareMathOperator{\End}{End}

\DeclareMathOperator{\Gal}{Gal}
\DeclareMathOperator{\Hom}{Hom}
\DeclareMathOperator{\Res}{Res}

\DeclareMathOperator{\trdeg}{tr.deg}
\DeclareMathOperator{\wt}{wt}

\DeclareMathOperator{\Emb}{Emb}
\DeclareMathOperator{\Inf}{Inf}
\DeclareMathOperator{\Spec}{Spec}
\DeclareMathOperator{\Div}{Div}

\DeclareMathOperator{\rank}{rank}
\DeclareMathOperator{\ord}{ord}

\DeclareMathOperator{\Gr}{Gr}

\newcommand{\ok}{\bar{k}}

\newcommand{\sep}{\mathrm{sep}}

\newcommand{\laurent}[2]{{#1 (\!( #2 )\!)}}

\def\rank{\operatorname{rank}}

\newcommand{\ari}{\operatorname{ari}}
\newcommand{\geo}{\operatorname{geo}}
\newcommand{\cyc}{\operatorname{cyc}}

\newcommand{\St}{\operatorname{St}}

\newcommand{\gal}{\operatorname{Gal}}

\newcommand{\divv}{\operatorname{div}}

\newcommand{\Pic}{\operatorname{Pic}}

\newcommand{\Frob}{\operatorname{Frob}}

\newcommand{\re}{\operatorname{Re}}

\newcommand\subfrac[2]{\genfrac{}{}{0pt}{}{#1}{#2}}

\newcommand{\assign}{\mathrel{\vcenter{\baselineskip0.5ex \lineskiplimit0pt
                     \hbox{\scriptsize.}\hbox{\scriptsize.}}}%
                     =}
\newcommand{\rassign}{=%
                     \mathrel{\vcenter{\baselineskip0.5ex \lineskiplimit0pt
                     \hbox{\scriptsize.}\hbox{\scriptsize.}}}%
                     }

\DeclareRobustCommand{\longleftmapsto}{\text{\reflectbox{$\longmapsto$}}}

\begin{document}

\title[Special gamma values and Chowla--Selberg phenomenon over function fields]{Algebraic Relations among Special Gamma Values and the Chowla--Selberg Phenomenon over Function Fields}

\author[F.-T.\ Wei]{Fu-Tsun Wei}
\address{Department of Mathematics, National Tsing Hua University, No.~101, Section 2, Kuang-Fu Road, Hsinchu City 300044, Taiwan R.O.C.}
\email{ftwei@math.nthu.edu.tw}

\thanks{
The author is supported by NSTC grant 109-2115-M-007-017-MY5
and the National Center for Theoretical Sciences.
}

\subjclass[2000]{Primary 11J93; Secondary 11G09}

\date{\today}

\begin{abstract}
The aim of this paper is to determine all algebraic relations among various special gamma values over function fields, and prove a Chowla--Selberg-type formula for quasi-periods of CM abelian $t$-modules. Our results are based on the intrinsic relations between gamma values in question and periods of CM dual $t$-motives, which are interpreted in terms of their ``distributions''. This also enables us to derive an analogue of the Deligne--Gross period conjecture for CM Hodge--Pink structures.
\end{abstract}

\keywords{Function field, CM dual $t$-motive, Period distribution, Stickelberger distribution, Gamma value, Chowla--Selberg formula, Deligne--Gross period conjecture}

\maketitle
\tableofcontents

\section{Introduction}
The main goal of this paper is to determine the algebraic relations among special gamma values over function fields, and derive an analogue of the Chowla--Selberg formula for quasi-periods of CM abelian $t$-modules.
In the classical case, the gamma function was first introduced by Euler in order to solve the interpolation problem of factorials.
This transcendental function shows up in many areas and plays an essential role since then.
The celebrated Chowla--Selberg formula (see \cite{CS67}) expresses the periods of CM elliptic curves, up to algebraic multiples, in terms of particular twisted products of gamma values at fractions.
This led to extensive studies on the connection between special gamma values in question and periods of CM abelian varieties
(see \cite{Gr78}, \cite{Anderson82}, \cite{Co93}, \cite{MR04}, \cite{Yang10}, \cite{BM16}, and \cite{Fr17}, etc.).
We refer the reader to \cite{Gr21} for a comprehensive historical review.
On the other hand, gamma functions on the function field side similarly arose from \emph{Carlitz factorials} \cite{Car35}, and then grew into three ``interpolations'' in the work of Goss \cite{Goss88} and Thakur \cite{Thakur91}.
After establishing natural functional equations of these gamma functions, Thakur proposed a recipe/conjecture on the \emph{Chowla--Selberg phenomenon} (see \cite[7.12]{Thakur91}) for the periods of CM Drinfeld modules and examined several cases.
In this paper, we shall prove Thakur’s conjecture, and clarify the intrinsic relations between special values of various gamma functions and ``abelian CM periods'' in positive characteristic.

\subsection{Special gamma values and algebraic independence}

Let $p$ be a prime number and $q$ be a power of $p$.
Let $A\assign \FF_q[\theta]$, the polynomial ring with one variable $\theta$ over a finite field $\FF_q$ with $q$ elements, and $k \assign \FF_q(\theta)$, the fraction field of $A$.
Denote by $k_\infty$ the completion of $k$ with respect to the absolute value $|\cdot|_\infty$ associated to the infinite place of $k$ normalized so that $|\theta|_\infty = q$, and $\CC_\infty$ the completion of a chosen algebraic closure of $k_\infty$.
Let $\ok$ be the algebraic closure of $k$ in $\CC_\infty$.
Put $A_+ \assign \{a \in A\mid a \text{ is monic}\}$.

There are three analogues of the gamma function in the function field case:
\begin{itemize}
    \item
    \emph{Arithmetic gamma function} (see Goss \cite[App.]{Go80}): Let $\ZZ_p$ be the ring of $p$-adic integers. For $y \in \ZZ_p$,
    \[
    \Gamma_{\ari}(y) \assign \Pi_{\ari}(y-1),
    \quad \text{ where } \quad 
    \Pi_{\ari}(y)\assign \prod_{i=1}^\infty \left(\prod_{j=0}^{i-1}(1-\frac{\theta^{q^j}}{\theta^{q^i}})\right)^{y_i}.
    \]
    Here $y_i$ for $i \in \ZZ_{\geq 0}$ are the unique integers with $0\leq y_i<q$ such that
    $y = \sum_{i=0}^\infty y_i q^i$.
    \item
    \emph{Geometric gamma function} (see Thakur \cite{Thakur91}): For $x \in \CC_\infty \setminus (-A_+ \cup \{0\})$,
    \[
    \Gamma_{\geo}(x) \assign \frac{1}{x} \cdot \Pi_{\geo}(x), \quad \text{ where } \quad 
    \Pi_{\geo}(x) \assign \prod_{a \in A_{\scaleto{+}{4pt}}}\left(1+\frac{x}{a}\right)^{-1}.
    \]
    \item \emph{Two-variable gamma function} (see Goss \cite{Goss88}): For $x \in \CC_\infty \setminus (-A_+ \cup \{0\})$ and $y \in \ZZ_p$,
    \[
    \Gamma(x,y) \assign \frac{1}{x} \cdot \frac{\Pi_{\geo}(x,y-1)}{\Pi_{\ari}(y-1)} \quad \text{ where } \quad 
    \Pi_{\geo}(x,y) \assign \prod_{i=0}^\infty\prod_{\substack{a \in A_{\scaleto{+}{4pt}} \\ \deg a = i}}\Big(1+\frac{x}{a}\Big)^{-y_i}.
    \]
    Here $y_i$ for $i \in \ZZ_{\geq 0}$ are chosen as in the arithmetic case.
\end{itemize}

The two-variable gamma function defined above actually follows Thakur's modification in \cite[the $\Gamma_2$ in Definition~8.1.1]{Thakur91}.
We refer the reader to
Remark~\ref{rem: Goss-gamma} for its comparison with Goss' original definition.
The two-variable gamma function can be seen as an extension of both the arithmetic and geometric ones.
In particular, one has that
\[
\Gamma(x,1-\frac{1}{q-1})\cdot \Gamma_{\ari}(1-\frac{1}{q-1}) = \Gamma_{\geo}(x), \quad \forall x \in \CC_\infty \setminus (-A_+\cup \{0\}).
\]

The natural functional equations of these gamma functions, e.g.\ the corresponding reflection and multiplication formulas, are well-established in \cite{Goss88} and \cite{Thakur91}.
We recall in Proposition~\ref{prop: AR-Gamma} the needed monomial relations, up to $\ok^\times$-multiples, among the gamma values at $x \in k\setminus(-A_+\cup \{0\})$ and $y \in \ZZ_{(p)} \assign \QQ \cap \ZZ_p$ induced by the functional equations.
The first theorem of this paper is the following analogue of the Lang--Rohrlich conjecture:

\begin{theorem}[Theorem~\ref{thm: Lang-Rohrlich-conj} and Corollary~\ref{cor: AR-Gamma}]\label{thm: L-R-conj}
 Given $\nfk \in A_+$ and $\ell \in \NN$, we have that
\begin{align*}
&\ \trdeg_{\ok}\ok\Big(\Gamma_{\geo}(x), \Gamma_{\ari}(y), \Gamma(x,y)\ \Big|\ x \in \frac{1}{\nfk}A\setminus(-A_+\cup\{0\}),\ y \in \frac{1}{q^\ell-1}\ZZ\Big) \\
=&\  1+(\ell - \frac{1}{(q-1)^{\epsilon_\nfk}})\cdot \#(A/\nfk)^\times,
\end{align*}
where $\epsilon_\nfk = 1$ if $\deg \nfk >0$ and $0$ otherwise.
Consequently, all algebraic relations among gamma values $\Gamma_{\geo}(x)$, $\Gamma_{\ari}(y)$, $\Gamma(x,y)$ for $x \in k\setminus (-A_+\cup\{0\})$ and $y \in \ZZ_{(p)}$ are explained by the monomial relations listed in Proposition~\ref{prop: AR-Gamma}.
\end{theorem}

\begin{remark}
(1) For $c_1,c_2 \in \CC_\infty^\times$, we denote by $c_1\sim c_2$ if $c_1/c_2 \in \ok^\times$. 
From the relation in \cite[Theorem~1.4]{Thakur91}:
\[
\Gamma_{\ari}(1-\frac{a}{q-1}) \sim \tilde{\pi}^{\frac{a}{q-1}}, \quad \forall a \in \ZZ \text{ with } 0\leq a<q-1,
\]
where $\tilde{\pi}$ is the \emph{Carlitz fundamental period} (see Remark~\ref{rem: Carlitz period}),
the monomial relations in Proposition~\ref{prop: AR-Gamma} imply that Theorem~\ref{thm: L-R-conj} coincides with the geometric case in \cite[Corollary~1.2.2]{ABP} (resp.\ the arithmetic case in \cite[Corollary~3.3.3]{CPTY10}) when taking $\ell =1$ (resp.\ $\nfk = 1$), see Remark~\ref{rem: comparison}~(1) and (2).

(2) The transcendence theory of arithmetic and geometric gamma values has been fully developed separately
(see \cite{Thakur91}, \cite{Thakur96}, \cite{All96}, \cite{MY97}, and \cite{CPTY10} for arithmetic gamma values, and see \cite{Sinha97}, \cite{BP02}, \cite{ABP} for the geometric ones).
In contrast, there seems to have been no further progress on two-variable gamma values after Goss and Thakur.
Theorem~\ref{thm: L-R-conj} not only fills in this missing part, but also clarifies the algebraic independence of these three different types of gamma values.
In particular, the algebraic relations among arithmetic and geometric gamma values come solely from their relations with $\tilde{\pi}$, see Remark~\ref{rem: comparison}~(3).
\end{remark}

The proof of Theorem~\ref{thm: L-R-conj} is based upon the ``period interpretation'' of the gamma values in question, which is illustrated via the ``period distribution'' introduced in the next subsection.

\subsection{The period distribution and gamma distributions}\label{subsec: PD-GD}
Let $\eA \assign \FF_q[t]$ be the polynomial ring over $\FF_q$ with another variable $t$ (transcendental over $\CC_\infty$), and $\ek = \FF_q(t)$ the fraction field of $\eA$.
Let $v_\theta: \ek \cong k$ be the $\FF_q$-algebra isomorphism sending $t$ to $\theta$.
We regard $\ek$ as the function field of the projective $t$-line $\PP^1$ over $\FF_q$, and $v_\theta$ corresponds to the evaluation map at the point $\theta \in \PP^1(\CC_\infty)$.
For $f \in \ek$, we denote $f(\theta)\assign v_\theta(f) \in k$.

Let $\ek^{\sep}$ be a separable closure of $\ek$, and put $\eG\assign \gal(\ek^{\sep}/\ek)$.
Fix an embedding $\nu: \ek^{\sep} \hookrightarrow \ok \subset \CC_\infty$ extending from $v_\theta$.
With respect to $\nu$, the Galois group $\gal(k^{\sep}_\infty/k_\infty)$, where $k^{\sep}_\infty$ is the separable closure of $k_\infty$ in $\CC_\infty$, corresponds to a subgroup of $\eG$ denoted by $\eG_\infty$.
The space $\Sscr(\eG)$ of \emph{Stickelberger functions} consists of all locally constant $\QQ$-valued functions $\varphi$ on $\eG$ satisfying (see Definition~\ref{defn: S(G)}):
\begin{itemize}
    \item[(1)] For $\varrho_1,\varrho_2 \in \eG$ and $\varrho_\infty \in \eG_\infty$, we have
    \[
    \varphi\big(\varrho_1 \cdot (\varrho_2 \varrho_\infty \varrho_2^{-1}\varrho_\infty^{-1})\big) = \varphi(\varrho_1);
    \]
    \item[(2)]
    Let $d\varrho_\infty$ be the Haar measure on $\eG_\infty$ normalized so that $\text{vol}(\eG_\infty,d\varrho_\infty) = 1$.
    Then
    \[
    \int_{\eG_\infty}\varphi(\varrho_1\varrho_\infty) d\varrho_\infty = \int_{\eG_\infty}\varphi(\varrho_2\varrho_\infty) d\varrho_\infty \quad \forall \varrho_1,\varrho_2 \in \eG.
    \]
\end{itemize}
The connections between particular Stickelberger functions and Dirichlet $L$-values are discussed in Lemma~\ref{lem: evaluator}.
In particular, 
let $\eK\subset \ek^{\sep}$ be a \emph{CM field over $\ek$} (recalled in Section~\ref{sec: CM}) which is fixed by the commutator subgroup $[\eG,\eG_\infty]$ and $\Xi$ a \emph{generalized CM type of $\eK$} (see Definition~\ref{defn: A.CM type}).
We define the Stickelberger function $\varphi_{\eK,\Xi}$ associated to $(\eK,\Xi)$ in \eqref{eqn: phi-K-Phi}, and show that $\Sscr(\eG)$ is spanned by $\varphi_{\eK,\Xi}$ for all CM fields $\eK$ over $\ek$ fixed by $[\eG,\eG_\infty]$ and generalized CM types $\Xi$ of $\eK$ (see Lemma~\ref{lem: S(G)}).

In \cite[Sec.~5]{BCPW22}, an analogue of Shimura's period symbols was introduced arising from the ``CM periods'' through the theory of Anderson's dual $t$-motives.
The fundamental properties of these symbols, which were established in \cite[Thm.~5.3.2]{BCPW22},
enable us to compare CM periods associated to different CM fields.
The needed results are recalled in Section~\ref{sec: defn PS}.
We then introduce an analogue of Anderson's \emph{period distribution} in the following.

\begin{theorem}[Theorem~\ref{thm: PD}]\label{thm: pd}
 Fix an embedding $\nu: \ek^{\sep} \hookrightarrow \ok \subset \CC_\infty$ extending from the evaluation map $v_\theta$.
There exists a unique $\QQ$-linear map $\Pscr_\nu: \Sscr(\eG)\rightarrow \CC_\infty^\times/\ok^\times$ satisfying that for every CM field $\eK\subset \ek^{\sep}$ fixed by $[\eG,\eG_\infty]$ and every generalized CM type $\Xi$ of $\eK$,
\[
\Pscr_\nu(\varphi_{\eK,\Xi}) = \Pcal_\eK(\xi_\nu,\Xi),
\]
where $\Pcal_\eK(\xi_\nu,\Xi)$ is the period symbol given in Definition~\ref{defn: period quantity}, and $\xi_\nu$ corresponds to the embedding $\nu\big|_{\eK}:\eK\hookrightarrow \CC_\infty$ under the bijection in Remark~\ref{rem: JK-Emb}. 
\end{theorem}

On the other hand, the gamma distributions are defined by taking values at ``fractions'': for $x \in k$ and $y\in \ZZ_{(p)}$, put $\{x\} \in k$ (resp.\ $\langle y\rangle_{\ari} \in \ZZ_{(p)}$) with $|\{x\}|_\infty <1$ (resp. $0\leq \langle y \rangle_{\ari} <1$) and $x-\{x\}\in A$ (resp.\ $y-\langle y\rangle_{\ari} \in \ZZ$). Then $\{\cdot\}$ and $\langle \cdot \rangle_{\ari}$ can be viewed as functions on $k/A$ and $\ZZ_{(p)}/\ZZ$, respectively.
For $x \in k/A$ and $y \in \ZZ_{(p)}/\ZZ$, we set
\[
\tilde{\Gamma}(x,y)\assign
\begin{cases}
\Gamma(\{x\},1-\langle -y\rangle_{\ari}), & \text{ if $0\neq x \in k/A$,}\\
\Gamma_{\ari}(1-\langle -y\rangle_{\ari})^{-1}, & \text{ otherwise.}
\end{cases}
\]
Then $\tilde{\Gamma}(x,y) \in \CC_\infty^\times$, and the \emph{two-variable gamma distribution} $\hat{\Gamma}: k/A\times \ZZ_{(p)}/\ZZ \rightarrow \CC_\infty^\times/\ok^\times$ is defined by
\[
\hat{\Gamma}(x,y) \assign \tilde{\Gamma}(x,y)\cdot \ok^\times \quad \in \CC_\infty^\times/\ok^\times, 
\quad \forall (x,y) \in k/A\times \ZZ_{(p)}/\ZZ.
\]
Moreover, the \emph{geometric (resp.\ arithmetic) gamma distribution} $\hat{\Gamma}_{\geo}: k/A\rightarrow \CC_\infty^\times/\ok^\times$ (resp.\ $\hat{\Gamma}_{\ari}: \ZZ_{(p)}/\ZZ\rightarrow \CC_\infty^\times/\ok^\times$) is
\[
\hat{\Gamma}_{\geo}(x) \assign \hat{\Gamma}(x,\frac{1}{1-q}), \quad \forall x \in k/A \quad (\text{resp.} \quad \hat{\Gamma}_{\ari}(y) \assign \hat{\Gamma}(0,y)^{-1}, \quad \forall y \in \ZZ_{(p)}/\ZZ).
\]

The main bridge between the above gamma distributions and the period distribution in Theorem~\ref{thm: pd} is built from the so-called \emph{Stickelberger distributions} (in the sense of \cite[Chap.~1, \S 3]{K-L}) associated to the corresponding ``diamond brackets'' introduced by Thakur~\cite{Thakur91}. 
Intuitively,
we identify $k/A$ (resp.\ $\ZZ_{(p)}/\ZZ$) with the ``Carlitz torsions'' (resp.\ roots of unity) in $\ok$, respectively, and equip with a $\eG$-action $\star$ via the Artin map (see \eqref{eqn: 2var-artin} and \eqref{eqn: G-action}).
For each $x \in k/A$ and $y \in \ZZ_{(p)}/\ZZ$, we define $\St^{\geo}(x), \St^{\ari}(y), \St(x,y): \eG\rightarrow \QQ$ respectively by (see Definition~\ref{defn: ST}):
\[
\St^{\geo}(x)(\varrho)\assign \langle \varrho\star x\rangle_{\geo}-\frac{1}{q-1}, 
\quad 
\St^{\ari}(y)(\varrho)\assign \langle -\varrho\star y\rangle_{\ari}
\]
\[
\St(x,y)(\varrho)\assign \langle \varrho \star x, -\varrho \star y\rangle - \langle -\varrho \star y\rangle_{\ari}, \quad 
\forall \varrho \in \eG,
\]
where $\langle \cdot \rangle_{\geo}$, $\langle \cdot \rangle_{\ari}$, and $\langle \cdot,\cdot\rangle$ are the \emph{geometric, arithmetic, and two-variable diamond brackets}, respectively (recalled in Section~\ref{sec: diamond bracket}).
It is seen that $\St^{\geo}(x)$, $\St^{\ari}(y)$, and $\St(x,y)$ all belong to $\Sscr(\eG^{\cyc})$,
where $\eG^{\cyc}$ is the Galois group of the compositum $\eK^{\cyc}$ of all \emph{Carlitz cyclotomic function fields} and constant field extensions over $\ek$ (see Section~\ref{sec: cyclotomic extn}), and $\Sscr(\eG^{\cyc})$ consists of all the functions in $\Sscr(\eG)$ factoring through $\eG^{\cyc}$.
Moreover, the diamond bracket relations (recalled in Section~\ref{sec: diamond bracket}) assure the ``distribution'' properties of $\St^{\geo}: k/A \rightarrow \Sscr(\eG^{\cyc})$, $\St^{\ari}: \ZZ_{(p)}/\ZZ\rightarrow \Sscr(\eG^{\cyc})$, and $\St: (k/A)\times(\ZZ_{(p)}/\ZZ)\rightarrow \Sscr(\eG^{\cyc})$ (see Remark~\ref{rem: St-relations}).
Generalizing the construction of the ``soliton dual $t$-motives'' in \cite[6.4.2]{ABP}, we derive the following analogue of Deligne's theorem (stated in \cite[Theorem~4.7]{Anderson82}):

\begin{theorem}[Theorem~\ref{thm: Gamma dist}]\label{thm: D-thm}
Let $\nu_1: \eK^{\cyc}\hookrightarrow \CC_\infty$ be the specific embedding defined in \eqref{eqn: nu_1}, and $\Pscr_{\nu_1}^{\cyc}$ be the restriction of the period distribution to $\Sscr(\eG^{\cyc})$ with respect to $\nu_1$.
Then
\[
\hat{\Gamma}_{\geo} = \Pscr_{\nu_1}^{\cyc}\circ \St^{\geo}, \quad 
\hat{\Gamma}_{\ari} = \Pscr_{\nu_1}^{\cyc}\circ \St^{\ari},
\quad \text{ and } \quad 
\hat{\Gamma} = \Pscr_{\nu_1}^{\cyc}\circ \St.
\]
\end{theorem}

The strategy of the proof of Theorem~\ref{thm: L-R-conj} is briefly outlined as follows:
Given $\nfk \in \eA_+$ and $\ell \in \NN$, let $\eK_{\nfk,\ell}$ be the \emph{$(\nfk,\ell)$-th cyclotomic function field} (given above \eqref{eqn: cyc-Artin}) and $\eG_{\nfk,\ell}$ the Galois group of $\eK_{\nfk,\ell}$ over $\ek$.
Denote by $\Sscr(\eG_{\nfk,\ell})$ the space consisting of the functions in $\Sscr(\eG^{\cyc})$ factoring through $\eG_{\nfk,\ell}$.
Let $\ok\big(\Pscr_{\nu_1}^{\cyc}(\Sscr(\eG_{\nfk,\ell}))\big)$ be the field generated by arbitrary representatives of $\Pscr_{\nu}^{\cyc}(\varphi) \in \CC_\infty^\times/\ok^\times$ for all $\varphi \in \Sscr(\eG_{\nfk,\ell})$.
The main theorem of \cite{BCPW22} on Shimura's conjecture shows that (see Theorem~\ref{thm: trdeg})
\[
\trdeg_{\ok} \ok \big(\Pscr_{\nu_1}^{\cyc}(\Sscr(\eG_{\nfk,\ell}))\big) = 1 + (\ell - \frac{1}{(q-1)^{\epsilon_\nfk}}) \cdot \#(\eA/\nfk)^\times.
\]
On the other hand, we prove below that the Stickelberger distributions $\St^{\geo}$, $\St^{\ari}$, and $\St$ are actually \emph{universal} (see Corollary~\ref{cor: Uari-iso}, \ref{cor: Ugeo-iso}, and \ref{cor: Ucyc-iso}).
Combining with Theorem~\ref{thm: D-thm}, we obtain that the field $\ok \big(\Pscr_{\nu_1}^{\cyc}(\Sscr(\eG_{\nfk,\ell}))\big)$ is algebraic over 
\begin{align*}
&\ \ok \Big(\Pscr_{\nu_1}^{\cyc}\big(\St(x,y)\big)\ \Big|\ x \in \frac{1}{\nfk(\theta)}A/A,\ y \in \frac{1}{q^{\ell}-1}\ZZ/\ZZ\Big) \\
= &\ \ok\Big(\tilde{\Gamma}(x,y)\ \Big|\ x \in \frac{1}{\nfk(\theta)}A/A,\ y \in \frac{1}{q^\ell-1}\ZZ/\ZZ \Big) \\
= &\ \ok\Big(\Gamma_{\geo}(x), \Gamma_{\ari}(y), \Gamma(x,y)\ \Big|\ x \in \frac{1}{\nfk(\theta)}A\setminus(-A_+\cup\{0\}),\ y \in \frac{1}{q^\ell-1}\ZZ\Big).
\end{align*}
The final assertion of Theorem~\ref{thm: L-R-conj} thereby will follow from the compatibility between the monomial relations in Proposition~\ref{prop: AR-Gamma}  and the diamond bracket relations in Lemma~\ref{lem: ari-relation},~\ref{lem: AL-geoG}, Proposition~\ref{prop: db-relation}, and Remark~\ref{rem: db-relation}.\\

Theorem~\ref{thm: D-thm} shows on one hand that all types of gamma values at fractions come from periods of CM dual $t$-motives, and, on the other hand, enables us to write ``abelian CM periods'' in terms of products of the gamma values in question.
Indeed, the universality of the Stickelberger distributions indicates that every Stickelberger function $\varphi \in \Sscr(\eG^{\cyc})$ can be written as a $\QQ$-linear combination of $\St(x,y)$ for $x \in k/A$ and $y \in \ZZ_{(p)}/\ZZ$.
In particular, let $\eK \subset \eK^{\cyc}$ be an \emph{imaginary} field over $\ek$ (i.e.\ the infinite place of $\ek$ does not split in $\eK$) and $\varphi_\eK \in \Sscr(\eG^{\cyc})$ be the characteristic function of the Galois group $\gal(\eK^{\cyc}/\eK)$.
Writing
\begin{equation}\label{eqn: T-r/c}
\varphi_\eK = \sum_{x \in k/A} \sum_{y \in \ZZ_{(p)}/\ZZ} m_{x,y} \St(x,y), \quad \text{ where }  m_{x,y} \in \QQ \text{ with } m_{x,y}=0 \text{ for almost all }x,y,
\end{equation}
we prove the following recipe/conjecture of Thakur from Theorem~\ref{thm: D-thm}:

\begin{theorem}[Theorem~\ref{thm: CSP}]
Keep the notation from above.
For every nonzero period $\lambda$ 
of a Drinfield $\eA$-module of rank $[\eK:\ek]$ over $\ok$ with CM by $O_{\eK}$, where $O_{\eK}$ is the integral closure of $\eA$ in $\eK$, we have
\[
\lambda \sim \prod_{x \in k/A}\prod_{y \in \ZZ_{(p)}/\ZZ}\tilde{\Gamma}(x,y)^{m_{x,y}}.
\]
\end{theorem}

\begin{remark}
${}$
\begin{itemize}
\item[(1)] Let $\eK\subset \eK^{\cyc}$ be an imaginary field over $\ek$.
Since any Drinfeld $\eA$-module $E_{\rho}$ of rank $[\eK:\ek]$ over $\bar{k}$ with CM by an arbitrary $\eA$-order in $\eK$ is isogenous (over $\bar{k}$) to one with CM by $O_{\eK}$, the above statement also applies to every nonzero period $\lambda$ of such $E_{\rho}$.
\item[(2)] The above recipe is actually generalized to quasi-periods of CM abelian $t$-modules with arbitrary CM type $(\eK,\Xi)$ over $\ok$, where $\eK \subset \eK^{\cyc}$, see Theorem~\ref{thm: qp-CM} and \ref{thm: HB-period}.
\end{itemize}
\end{remark}

Our next task is to carry out the decomposition \eqref{eqn: T-r/c} in an explicit form. 
This is achieved in Proposition~\ref{prop: phi-ST} from the arithmetic property of the Stickelberger functions $\St(x,y)$ for $x \in k/A$ and $y \in \ZZ_{(p)}/\ZZ$ regarded as ``evaluators'' of Dirichlet $L$-values in Lemma~\ref{lem: evaluator}.
Therefore Theorem~\ref{thm: D-thm} provides a definite expression for every ``abelian CM period'' in terms of a monomial product of special gamma values, 
which results in an analogue of the Chowla--Selberg formula presented in the next subsection.

\subsection{\texorpdfstring{The Chowla--Selberg formula for abelian $t$-modules}{The Chowla--Selberg formula for abelian t-modules}}

In the theory of function field arithmetic,
\emph{abelian $t$-modules}, introduced by Anderson \cite{Anderson86}, play the analogous rule of commutative algebraic groups in the classical settings. 
The \emph{quasi-periods} of abelian $t$-modules (introduced by Gekeler~\cite{Ge89} and Yu~\cite{Yu90} for Drinfeld modules, and generalized to arbitrary abelian $t$-modules by Brownawell-Papanikolas \cite{BP02}) 
often show up as ``special values'' in various circumstances (see \cite{Car35}, \cite{Yu91}, \cite{CY07}, \cite{NP21}, \cite{GN21}, \cite{Ge88}, \cite{C12}, \cite{C12-2}, \cite{CG22}, etc.).
From the comparisons between the ``de Rham pairings''  established in \cite[Prop.~8.3.4]{BCPW22} (recalled in \eqref{eqn: dR-P comp}), we may naturally transform the study of quasi-periods of a given abelian $t$-module $E_\rho$ into the analysis of periods of the associated \emph{Hartl--Juschka dual $t$-motive $\eM(\rho)$} (see \eqref{eqn: M(rho)}).
In particular, when $E_\rho$ is a
\emph{CM abelian $t$-module} over $\ok$ (see after \eqref{eqn: dR-P comp}), its quasi-periods can be expressed in terms of our period symbols (cf.~\cite[Thm.~8.4.1 and 8.5.2]{BCPW22}).

Now, let $E_\rho$ be a CM abelian $t$-module with generalized CM type $(\eK,\Xi)$ over $\ok$.
Suppose further that $\eK \subset \eK_{\nfk,\ell}$ for $\nfk \in \eA_+$ and $\ell \in \NN$, where $\eK_{\nfk,\ell}$ is the $(\nfk,\ell)$-th cyclotomic function field (given above \eqref{eqn: cyc-Artin}) .
Let $\varphi_{\eK,\Xi} \in \Sscr(\eG^{\cyc})$ be the Stickelberger function associated to $(\eK,\Xi)$ given in \eqref{eqn: phi-K-Phi}.
We may view $\varphi_{\eK,\Xi}$ as a $\ZZ$-valued function on $\eG_\eK\assign \gal(\eK/\ek)$.
For each $\varrho_0 \in \eG_\eK$, put $(\varrho_0\cdot \varphi_{\eK,\Xi})(\varrho) \assign \varphi_{\eK,\Xi}(\varrho\varrho_0)$ for every $\varrho \in \eG_\eK$.
Then the space of quasi-periods of $E_\rho$ is written as (see \eqref{eqn: CM-qp})
\[
\sum_{\varrho_0 \in \eG_\eK} \ok\cdot \wp_{\nu_1}^{\cyc}(\varrho_0 \cdot \varphi_{\eK,\Xi}),
\]
where $\wp_{\nu_1}^{\cyc}(\varrho_0 \cdot \varphi_{\eK,\Xi}) \in \CC_\infty^\times$ is an arbitrary representative of $\Pscr_{\nu_1}^{\cyc}(\varrho_0 \cdot \varphi_{\eK,\Xi}) \in \CC_\infty^\times/\ok^\times$.
For each $\varrho \in \eG_\eK$, put $m_\varrho \assign \varphi_{\eK,\Xi}(\varrho)$.
From the explicit expression of $\varrho_0\cdot\varphi_{\eK,\Xi}$ in Proposition~\ref{prop: phi-ST} for every $\varrho_0 \in \eG_\eK$, we derive the following Chowla--Selberg formula for quasi-periods of $E_\rho$:

\begin{theorem}[Theorem~\ref{thm: ex-qp-CM}~(1)]\label{thm: CSF-1}
Keep the notation as above.
The space of quasi-periods of $E_\rho$ is spanned over $\bar{k}$ by
\[
\tilde{\pi}^{\frac{\wt(\Xi)}{[\eK:\eK^+]}}
\cdot  \prod_{\varrho \in \eG_\eK} \left(\prod_{\cfk\mid \nfk}\prod_{\substack{a \in (\eA/\cfk)^\times \\ i \in \ZZ/\ell \ZZ}} \tilde{\Gamma}(\frac{a(\theta)}{\cfk(\theta)},\frac{q^i}{1-q^\ell})^{n_{\cfk}(\varrho\varrho_0,a,i)}\right)^{\frac{m_{\varrho}}{[\eK:\ek]}}, \quad \text{ for }  \varrho_0 \in \eG_\eK,
\]
where $\wt(\Xi)$ is the weight of $\Xi$ given in \eqref{eqn: weight}, $\eK^+$ is the maximal totally real subfield of $\eK$, and $n_\cfk(\varrho,a,i)$ is a particular rational number defined in \eqref{eqn: defn-nc}.
\end{theorem}

In particular, let $\eK^+$ be the maximal totally real subfield of $\eK$. When $\wt(\Xi) = 1$, i.e.\ $\Xi$ is a CM type of $\eK$,
there exist exactly $\varrho_1,...,\varrho_d \in \eG_\eK$, where $d = [\eK^+:\ek]$, such that 
$\varphi_{\eK,\Xi}(\varrho) = 1$ if $\varrho = \varrho_i$ for some $i$ with $1\leq i \leq d$ and $0$ otherwise.
We then obtain the Chowla--Selberg formula for $\lambda \in \Lambda_\rho$, where $\Lambda_\rho$ is the \emph{period lattice of $E_\rho$} (see Section~\ref{sec: pv-CM}):

\begin{theorem}[Theorem~\ref{thm: ex-qp-CM}~(2)] \label{thm: CSF-2}
Let $E_\rho$ be a CM abelian $t$-module with CM type $(\eK,\Xi)$ over $\ok$, where $\eK \subset \eK_{\nfk,\ell}$ for $\nfk \in \eA_+$ and $\ell \in \NN$.
Let $\eK^+$ be the maximal totally real subfield of $\eK$, and put $d = [\eK^+:\ek]$.
Take $\varrho_1,...,\varrho_d \in \eG_\eK$ as above, and let $\nu_{\varrho_j}\assign \nu_1 \circ \varrho_j$ for $1\leq j \leq d$.
There exist an ideal $\Ifk_\rho$ of $O_\eK$ (the integral closure of $\eA$ in $\eK$) and a suitable $\bar{k}$-linear isomorphism  $\Lie(E_\rho) \cong \bar{k}^d$ so that the
image of every period vector $\lambda \in \Lambda_\rho \subset \Lie(E_\rho)(\CC_\infty)$ in $\CC_\infty^d$ has the form $(\nu_{\varrho_1}(\alpha)\lambda_1,...,\nu_{\varrho_d}(\alpha)\lambda_d) \in \CC_\infty^d$,
where $\alpha \in \Ifk_\rho$ and for $1\leq j \leq d$,
\[
\lambda_j \assign \tilde{\pi}^{\frac{1}{[\eK:\eK^+]}}
\cdot  \prod_{j'=1}^d \prod_{\cfk\mid \nfk}\prod_{\substack{a \in (\eA/\cfk)^\times \\ i \in \ZZ/\ell\ZZ}} \tilde{\Gamma}(\frac{a(\theta)}{\cfk(\theta)},\frac{q^i}{1-q^\ell})^{\frac{n_{\cfk}(\varrho_{j'}\varrho_{j}^{-1},a,i)}{[\eK:\ek]}}.
\]
\end{theorem}

Furthermore, applying Theorem~\ref{thm: CSF-1} and \ref{thm: CSF-2} to the case of Drinfeld modules, we get:

\begin{theorem}[Theorem~\ref{thm: ex-qp-CM}~(3)] \label{thm: CSF-3}
Let $E_\rho$ be a Drinfeld $\eA$-module of rank $r$ over $\ok$ which has CM by $O_\eK$, where $O_\eK$ is the integral closure of $\eA$ in an imaginary field $\eK$ over $\ek$ with $[\eK:\ek] = r$.
Suppose $\eK \subset \eK_{\nfk,\ell}$ for $\nfk \in \eA_+$ and $\ell \in \NN$.
Then
\[
\lambda^{[\eK:\ek]} \sim  
\tilde{\pi}
\cdot \prod_{\cfk\mid \nfk}\prod_{\substack{a \in (\eA/\cfk)^\times \\ i\in \ZZ/\ell\ZZ}} \tilde{\Gamma}(\frac{a(\theta)}{\cfk(\theta)},\frac{q^i}{1-q^\ell})^{n_{\cfk}(\text{\rm id}_{\eK},a,i)}
\quad \text{ for every nonzero period } \lambda \in \Lambda_\rho \subset \CC_\infty.
\]
Moreover,
\[
\varpi_\eK^{\varrho}\assign \tilde{\pi}^{\frac{1}{[\eK:\ek]}}
\cdot \prod_{\cfk\mid \nfk}\prod_{\substack{a \in (\eA/\cfk)^\times \\ i\in \ZZ/\ell\ZZ}} \tilde{\Gamma}(\frac{a(\theta)}{\cfk(\theta)},\frac{q^i}{1-q^\ell})^{\frac{n_{\cfk}(\varrho,a,i)}{[\eK:\ek]}},
\quad \text{ for } \varrho \in \eG_\eK,
\]
are algebraically independent over $\bar{k}$ and
form a $\bar{k}$-basis of the space of quasi-periods of $E_\rho$.
\end{theorem}

\begin{remark}
(1) When the imaginary field $\eK$ in Theorem~\ref{thm: CSF-3} is either a constant field extension of $\ek$ or the $t$-th Carlitz cyclotomic function field, we verify in Example~\ref{ex: CSF-1} and~\ref{ex: CSF-geo} that the first statement of Theorem~\ref{thm: CSF-3} agrees with Thakur's formula in \cite[Theorem~1.6]{Thakur91} and \cite[Theorem~4.11.2]{Thakur04}, respectively,
and the second statement matches with \cite[Theorem~3.3.2]{CPTY10} and \cite[Theorem~6.4.7~(2)]{Thakur04}, respectively.

(2) Suppose the imaginary field $\eK$ in Theorem~\ref{thm: CSF-3} is quadratic over $\ek$.
We may take $\ell = 2$ and $\nfk = \dfk \in \eA_+$ to be the generator of the discriminant ideal of $O_\eK/\eA$.
Let
$\bochi:\eG_\eK \rightarrow \{\pm 1\}$ be the quadratic character associated to $\eK/\ek$. 
Regarding $\bochi$ as a quadratic character on $(\eA/\dfk)^\times \times (\ZZ/2\ZZ)$ via the Artin map \eqref{eqn: cyc-Artin}, we calculate in Example~\ref{ex: CSF-2} that 
\[
n_\dfk(\varrho,a,i) = \frac{w_\eK \cdot \bochi(a,i+\deg \dfk)}{\#\Pic(O_\eK)} \cdot 
\begin{cases}
1, & \text{ if $\varrho = \text{id}_\eK$,}\\
-1, & \text{ otherwise,}
\end{cases}
\]
where $\Pic(O_\eK)$ is the ideal class group of $O_\eK$ and $w_\eK \assign \#(O_\eK)^\times/\#(\FF_q^\times)$.
Thus, if we put
\[
\omega_{\eK}^{\pm} \assign \sqrt{\tilde{\pi}}
\cdot \prod_{\substack{a \in (\eA/\dfk)^\times \\ i\in \ZZ/2\ZZ}} \tilde{\Gamma}(\frac{a(\theta)}{\dfk(\theta)},\frac{q^i}{1-q^2})^{\pm  \frac{w_\eK \bochi(a,i+\deg \dfk)}{2\#\Pic(O_\eK)}},
\]
then Theorem~\ref{thm: CSF-3} says that
\[
\lambda \sim \varpi_\eK^+, \quad \text{ for every nonzero period $\lambda \in \Lambda_\rho \subset \CC_\infty$.}
\]
Moreover, $\{\varpi_\eK^+,\varpi_\eK^-\}$ are algebraically independent over $\ok$ and form a $\ok$-basis of the space of the quasi-periods of $E_\rho$.

(3) After suitable ``normalization'' of the gamma values in \eqref{eqn: gamma-normalization}, we derive a Lerch-type formula for the derivative of the Dirichlet $L$-function at $s=0$ (see Theorem~\ref{thm: Lerch-type formula}).
Also, applying the Kronecker limit formula established in \cite{Wei}, we then obtain an ``analytic version'' of the Chowla--Selberg formula in \eqref{eqn: A-CSF}.
\end{remark}

\subsection{The Deligne--Gross conjecture for CM Hodge--Pink structures}

Motivated by his geometric proof of the Chowla--Selberg formula, Gross \cite[p.~205]{Gr78} proposed a conjecture (with an explicit formulation suggested by Deligne) which asserts that the periods of a ``geometric'' Hodge structure with multiplication by an abelian number field can be expressed, up to $\overline{\QQ}^\times$-multiples, as products of gamma values at fractions, with exponents determined by the Hodge decomposition.
This conjecture still remains wide open, except for certain verified cases (cf.\ \cite[p.~205--207]{Gr78}, \cite{MR04}, and \cite{Fr17}).
Having sufficient tools at hands on the function field side, we are now able to prove an analogue of the Deligne--Gross period conjecture in the context of dual $t$-motives.

Let $\eM$ be a pure uniformizable dual $t$-motive over $\ok$, and $\eH(\eM) = (H^\eM, W_\bullet H^\eM, \qfk^\eM)$ the \emph{$\ek$-Hodge--Pink structure of $\eM$} (see Example~\ref{ex: H(M)}).
Given a $\ek$-Hodge--Pink sub-structure $\eH' = (H',W_\bullet H', \qfk')$ of $\eH(\eM')$, suppose it has \emph{full-CM} by a CM field $\eK$ over $\ek$ (see \eqref{eqn: HCM}).
Assume further that $\eK \subset \eK_{\nfk,\ell}$ for $\nfk \in \eA_+$ and $\ell \in \NN$.
Let $\Phi_{\eH'}$ be the \emph{Hodge--Pink type of $\eH'$} (see Definition~\ref{defn: HP-type}).
We verify in Lemma~\ref{lem: HP-type in IK0} that $\Phi_{\eH'}$ is a generalized CM type of $\eK$.
In particular, the associated Stickelberger function $\varphi_{\eK,\Phi_{\eH'}}$ lies in $\Sscr(\eG_{\nfk,\ell})$.
As mentioned in the end of Section~\ref{subsec: PD-GD}, the universality of the Stickelberger distribution ensures that there exists a function $\varepsilon^{\eH'}: (\frac{1}{\nfk(\theta)}A/A)\times (\frac{1}{q^\ell-1}\ZZ/\ZZ) \rightarrow \QQ$ such that
\[
\varphi_{\eK,\Phi_{\eH'}} = 
\sum_{x \in \frac{1}{\nfk(\theta)}A/A} \ 
\sum_{y \in \frac{1}{q^\ell-1}\ZZ/\ZZ} \varepsilon^{\eH'}(x,y)\cdot \St(x,y).
\]
On the other hand, for each $\varrho \in \eG_\eK$, let $\nu_\varrho \assign \nu_1 \circ \varrho: \eK \hookrightarrow \CC_\infty$.
As $\eH$ has full-CM by $\eK$, there exists a unique one dimensional eigen-subspace of $H'_{\CC_\infty} \assign \CC_\infty \otimes_{\ek} H'$ whose eigenvalue for each $\alpha \in \eK$ is $\nu_{\varrho}(\alpha)$.
Identifying $H'_{\CC_\infty}$ with a subspace of the de Rham module $H_{\mathrm{dR}}(\eM,\CC_\infty)$ of $\eM$ over $\CC_\infty$ via the \emph{period isomorphism} (see \eqref{eqn: period-iso}), take a nonzero differential $\omega_{\varrho}^{\eH'}$ of $\eM$ over $\CC_\infty$ which lies in the eigen-subspace of $H'_{\CC_\infty}$ with respect to $\nu_{\varrho}$.
Our analogue of the Deligne--Gross period conjecture is presented in the following.

\begin{theorem}[Theorem~\ref{thm: DG-conj}]\label{thm: DG-C}
Keep the notation from above.
If $\omega_{\varrho}^{\eH'}$ is algebraic (i.e.\ $\omega_{\varrho}^{\eH'} \in H_{\mathrm{dR}}(\eM,\ok)$,
then for every $\gamma$  in $H_{\mathrm{Betti}}(\eM)$ so that the ``period integral'' $\int_{\gamma} \omega_{\varrho}^{\eH'}$ is nonzero, we have that
\[
\int_{\gamma} \omega_{\varrho}^{\eH'} \sim \prod_{x \in \frac{1}{\nfk(\theta)}A/A} \ 
\prod_{y \in \frac{1}{q^\ell-1}\ZZ/\ZZ} \tilde{\Gamma}(x,y)^{\varepsilon^{\eH'}(\varrho^{-1}\star x, \varrho^{-1} \star y)}.
\]
\end{theorem}

The main step in the proof of Theorem~\ref{thm: DG-C} is to appeal to  the function field analogue of the Hodge conjecture established by Pink, Hartl, and Juschka in \cite{Pink}, \cite{HP04}, \cite{Jus10}, and \cite{HJ20}, to obtain a dual $t$-motive $\eM'$ together with an \emph{essentially surjective} morphism $f: \eM \rightarrow \eM'$ (i.e.\ $\eM'$ is isogenous to $f(\eM)$) so that $\eH(\eM')$ is isomorphic to $\eH'$ via $f$ (see Theorem~\ref{thm: HP-conj}).
As $\eH'$ has full-CM by $\eK$, we may take $\eM'$ to be a CM dual $t$-motive over $\ok$ with generalized CM type $(\eK,\Phi_{\eH'})$ without loss of generality (see Proposition~\ref{prop: CM-HP} and Remark~\ref{rem: field-of-defn}), and the morphism $f$ becomes defined over $\ok$ by Lemma~\ref{lem: morphism-defining-field}.
This enables us to see that
\[
\int_{\gamma} \omega_{\varrho}^{\eH'}
\sim \wp_{\nu_1}^{\cyc}\big(\varphi_{\eK,\Phi_{\eH'}^{\varrho^{-1}}}\big),
\]
and so the result follows from Theorem~\ref{thm: D-thm}.

\begin{remark}
(1) The proof of Theorem~\ref{thm: DG-C} ensures in particular that a nonzero algebraic differential $\omega_\varrho^{\eH'}$ exists for each $\varrho \in \eG_\eK$ (see Remark~\ref{rem: alg-omega}).

(2) Although the choice of the function $\varepsilon^{\eH'}$ may not be unique, we can follow the approach of our Chowla--Selberg formula to take an explicit one (see Remark~\ref{rem: ex-epsilon}).
As a result, analogues of the Gross formulas \cite[(4), (5) in p.~194]{Gr78} are shown in Example~\ref{ex: Gross-formula}.
\end{remark}

\subsection*{Acknowledgements}

This work is a subsequent project of the author's  joint paper \cite{BCPW22} with W.~Dale Brownawell, Chieh-Yu Chang, and Matthew A. Papanikolas.
He sincerely thanks them for so many helpful discussions and suggestions.
The author is also grateful to Jing Yu for providing very thoughtful and useful comments during the preparation of this paper. 
He particularly thanks National Center for Theoretical Sciences for financial support and hospitality over the past years.

\section{Preliminaries}

\subsection{Basic settings}

Given a finite field $\FF_q$ with $q$ elements, 
let $A \assign \FF_q[\theta]$, the polynomial ring with one variable $\theta$ over $\FF_q$, and $k\assign \FF_q(\theta)$, the field of fractions of $A$.
Let $|\cdot|_\infty$ be the absolute value on $k$ with respect to the infinite place $\infty$ of $k$ which is normalized so that $|\theta|_\infty \assign q$.
We may identify the completion of $k$ with respect to $|\cdot|_\infty$ with $k_\infty \assign  \laurent{\FF_{q}}{1/\theta}$.
Let $\CC_\infty$ be the completion of a fixed algebraic closure $\bar{k}_\infty$ of $k_\infty$.
We still denote by $|\cdot|_\infty$ the extended absolute value on $\CC_\infty$.
Let $\bar{k}$ be the algebraic closure of $k$ in $\CC_\infty$,
and let
$k^{\text{sep}}$ (resp.\ $k_\infty^{\text{sep}}$) be the separable closure of $k$ (resp.\ $k_\infty$) in $\bar{k}$ (resp.\ $\bar{k}_\infty$).
In this paper, elements in $\CC_\infty$ are regarded as scalars.

Let $\eA \assign \FF_q[t]$, the polynomial ring with another variable $t$ over $\FF_q$ (where $t$ is transcendental over $\CC_\infty$), and $\ek \assign \FF_q(t)$, the field of fractions of $\eA$.
Here we regard $\ek$ (resp.\ $\eA$) as the function field (resp.\ affine coordinate ring) of the projective (resp.\ affine) $t$-line $\PP^1$ (resp.\ $\mathbb{A}^1$) over $\FF_q$.
In particular, the evaluation map at the point $t=\theta$ on the base change of $\PP^1$ from $\FF_q$ to $k$ induces an $\FF_q$-algebra isomorphism $v_\theta: \ek \cong k$ sending $t$ to $\theta$.
For each $f \in \ek$, we put $f(\theta)\assign v_\theta(f) \in k$.
The absolute value on $\ek$ induced by the one on $k$ via $v_\theta$ is denoted by $|\cdot|_\infty$ as well if no confusion arises.
Finally, let $\eA_+$ be the set of monic polynomials in $\eA$.
By abuse of notation, the ideal generated by $\nfk \in \eA_+$ is also denoted by $\nfk$.

\subsection{\texorpdfstring{Dual $t$-motives and Betti modules}{Dual t-motives and Betti modules}}

Let $\KK$ be an algebraically closed subfield of $\CC_\infty$ containing $\bar{k}$.
Let $\KK[t,\sigma]$ be the non-commutative polynomial ring over $\KK$ subject to the following relations:
\[
tc = ct, \quad t\sigma = \sigma t, 
\quad\text{and}\quad 
\sigma c = c^{1/q} \sigma, \quad \forall c \in \KK.
\]
We may view the (resp.\ twisted) polynomial ring $\KK[t]$ (resp.\ $\KK[\sigma]$) as a subring of $\KK[t,\sigma]$.

\begin{definition}
A \emph{dual $t$-motive over $\KK$} is a left $\KK[t,\sigma]$-module $\eM$ satisfying that:
\begin{itemize}
    \item[(1)] $\eM$ is a free $\KK[t]$-module of finite rank;
    \item[(2)] $\eM$ is a free left $\KK[\sigma]$-module of finite rank;
    \item[(3)] $(t-\theta)^N\cdot \eM \subset \sigma \eM$ for a sufficiently large integer $N$.
\end{itemize}
Given two dual $t$-motives $\eM$ and $\eM'$ over $\KK$, a \emph{morphism $f:\eM \rightarrow \eM'$ over $\KK$} is a $\KK[t,\sigma]$-module homomorphism from $\eM$ to $\eM'$.
\end{definition}

\begin{remark}
This definition of dual $t$-motive is what other sources refer to as a `$t$-finite' or `$A$-finite' dual $t$-motive (see \cite{HJ20} and \cite{NP21}). 
\end{remark}

\begin{remark}\label{rem: isogeny}
A morphism $f: \eM\rightarrow \eM'$ over $\KK$ is called an \emph{isogeny} if $f$ is injective with finite dimensional cokernel over $\KK$. In this case, we call $\eM$ \emph{isogenous} to $\eM'$ over $\KK$.
Moreover, it is known that there exists $a \in \eA$ such that $a\eM' \subset f(\eM)$ (see \cite[Theorem~4.4.7]{ABP}), whence there exists another isogeny $g: \eM' \rightarrow \eM$ satisfying that
\[
f\circ g(m') = a\cdot m',\quad \forall m' \in \eM',
\quad \text{ and } \quad 
g\circ f (m) = a \cdot m, \quad \forall m \in \eM.
\]
Thus $\eM'$ is isogenous to $\eM$ over $\KK$ as well.
\end{remark}

\begin{Subsubsec}[\emph{Uniformizability}]\label{sec: uniformizibility}
Let $\TT$ be the following Tate algebra:
\[
\TT \assign \left\{ \sum_{i=0}^\infty c_i t^i \in \CC_\infty[\![t]\!]\ \Bigg|\ \lim_{i\rightarrow \infty} |c_i|_\infty = 0\right\}.
\]
Let $n \in \ZZ$.
For each $f = \sum_{i=0}^\infty c_i t^i \in \TT$,
the \emph{$q^n$-power Frobenius twist of $f$} is
\[
f^{(n)}\assign \sum_{i=0}^\infty c_i^{q^n} t^i.
\]
In particular, we may associate a $\CC_\infty[t,\sigma]$-module structure on $\TT$ extending $\CC_\infty[t]$-multiplication so that
\[
\sigma \cdot f = f^{(-1)}, \quad \forall f \in \TT.
\]
Let $\TT[\sigma]$ be the twisted polynomial ring over $\TT$ subject to the relation 
$\sigma f = f^{(-1)}\sigma$ for every $f \in \TT$.
Set $\MM \assign \TT \otimes_{\KK[t]} \eM$,
which is equipped with a module structure over $\TT[\sigma]$ so that $\sigma$ acts diagonally on both $\TT$ and $\eM$.
The {\it Betti module of $\eM$} is
\begin{equation}\label{eqn: Betti}
H_{\rm Betti}(\eM) \assign \MM^\sigma = \{ m \in \MM\mid \sigma \cdot m = m\}.
\end{equation}
It is known that $H_{\rm Betti}(\eM)$ is a free $\eA$-module of rank less than or equal to $\rank_{\KK[t]}\eM$,
and the natural map
$\TT \otimes_{\eA} H_{\rm Betti}(\eM) \rightarrow \MM$
is injective
(cf.\ \cite[(2.4.1)]{BCPW22}).
We call $\eM$ {\it uniformizable} if the map
$\TT \otimes_{\eA} H_{\rm Betti}(\eM) \rightarrow \MM$
is also surjective, whence becomes a $\TT[\sigma]$-module isomorphism.
In this case, one has that $\rank_{\eA} H_{\rm Betti}(\eM) = \rank_{\KK[t]}(\eM)$,
and \cite[Proposition 2.3.30]{HJ20} says that 
\[
H_{\rm Betti}(\eM)
\subset 
\TT^\dagger \otimes_{\KK[t]} \eM \rassign \MM^\dagger,
\]
where $\TT^\dagger$ is the subring of $\TT$ consisting of all rigid analytic functions on $\CC_\infty \setminus\{\theta^{q^i}\mid i \in \NN\}$.
\end{Subsubsec}

\begin{Subsubsec}[\emph{Push-forward}]\label{subsec: push-forward}
Let $\eM$ and $\eM'$ be two uniformizable dual $t$-motives over $\KK$. 
Given a morphism $f:\eM\rightarrow \eM'$ over $\KK$,
we may extend $f$ to a $\TT[\sigma]$-module homomorphism from $\MM$ to $\MM'$ which induces an $\eA$-module homomorphism $f_*: H_{\mathrm{Betti}}(\eM)\rightarrow H_{\mathrm{Betti}}(\eM')$, called the \emph{push-forward of $f$}.
In particular, when $f$ is an isogeny, by Remark~\ref{rem: isogeny} one knows that
$f_*$ must be injective with
\[
\dim_{\FF_q}\left(\frac{H_{\mathrm{Betti}}(\eM')}{f_*\big(H_{\mathrm{Betti}}(\eM)\big)}\right) = \dim_{\KK}\left(\frac{\eM'}{f(\eM)}\right) < \infty.
\]
Thus 
the $\ek$-linear map 
$\text{id}_\ek \otimes f_* : \ek\otimes_{\eA} H_{\mathrm{Betti}}(\eM) \rightarrow \ek\otimes_{\eA} H_{\mathrm{Betti}}(\eM')$ is an isomorphism.
\end{Subsubsec}

A dual $t$-motive $\eM$ over $\CC_\infty$ is said to be \emph{defined over $\KK$} if there exists a dual $t$-motive $\eM_0$ over $\KK$ such that 
$\eM \cong \CC_\infty \otimes_{\KK} \eM_0$.
We have the following.

\begin{lemma}\label{lem: morphism-defining-field}
Let $\eM$ and $\eM'$ be two uniformizable dual $t$-motive over $\KK$.
We put $\eM_{\CC_\infty} \assign \CC_\infty \otimes_{\KK} \eM$ and $\eM_{\CC_\infty}' \assign \CC_\infty \otimes_{\KK} \eM'$, which become uniformizable dual $t$-motives over $\CC_\infty$. 
Then every morphism $\tilde{f}:\eM_{\CC_\infty} \rightarrow \eM'_{\CC_\infty}$ over $\CC_\infty$ is actually defined over $\KK$, i.e.\ $\tilde{f}(\eM)\subseteq \eM'$.
\end{lemma}

\begin{proof}
Let $n = \rank_{\KK[t]}\eM$ and $n' = \rank_{\KK[t]}(\eM')$.
Take $\mathbf{m} \in \Mat_{1,n}(\eM)$ (resp.\ $\mathbf{m'} \in \Mat_{1,n'}(\eM')$) whose entries form a $\KK[t]$-base of $\eM$ (resp. $\eM'$). 
By \cite[Proposition~3.1.3 and Lemma~4.4.12]{ABP} there exist $\Psi \in \GL_n(\TT\cap \KK[\![t]\!])$
and $\Psi' \in \GL_{n'}(\TT\cap \KK[\![t]\!])$ so that
the entries of $\Psi^{-1}\mathbf{m}$ (resp.\ $(\Psi')^{-1}\mathbf{m'}$) form an $\eA$-base of $H_{\mathrm{Betti}}(\eM)$ (resp.\ $H_{\mathrm{Betti}}(\eM')$).
Then we can find $U \in \Mat_{n,n'}(\eA)$ so that
\[
\tilde{f}_*\big(\Psi^{-1}\mathbf{m}\big) = U \cdot (\Psi')^{-1} \mathbf{m'}.
\]
On the other hand, the push-forward $\tilde{f}_*$ of $\tilde{f}$ satisfies that
\[
\Psi^{-1} \cdot \tilde{f}(\mathbf{m}) = \tilde{f}_*\big(\Psi^{-1}\mathbf{m}\big).
\]
Combining these two equalities, we obtain that
\[
\tilde{f}(\mathbf{m}) = \big(\Psi \cdot U \cdot (\Psi')^{-1}\big) \mathbf{m'},
\]
whence $\Psi \cdot U \cdot (\Psi')^{-1} \in \Mat_{n,n'}(\KK[\![t]\!] \cap \CC_\infty[t]) = \Mat_{n,n'}(\KK[t])$.
Therefore $\tilde{f}(\eM)\subseteq \eM'$.
\end{proof}

\begin{Subsubsec}[\emph{Purity}]
Let $\eM$ be a dual $t$-motive over $\KK$.
We may extend the $\sigma$-action on $\eM$ to $\eM(\!(1/t)\!)\assign \KK(\!(1/t)\!)\otimes_{\KK[t]} \eM$ by:
\[
\sigma \cdot (f\otimes m) \assign f^{(-1)} \otimes (\sigma m), \quad \forall f \in \KK(\!(1/t)\!) \quad \text{ and } \quad m \in \eM,
\]
where for each Laurent series $f = \sum_{i=m}^\infty c_i t^{-i} \in \KK(\!(1/t)\!)$,
$f^{(-1)}$ is the $q^{-1}$-power Frobenius twist of $f$:
\[
f^{(-1)}\assign \sum_{i=m}^\infty c_i^{q^{-1}} t^{-i}.
\]
We say that $\eM$ is \emph{pure} if there exist a $\KK[\![1/t]\!]$-lattice $L$ of $\eM(\!(1/t)\!)$ and $u,v \in \NN$ such that
\[
t^u \cdot L = \sigma^v \cdot L.
\]
In this case, it is known that (see \cite[Proposition~2.4.10~(e)]{HJ20}) 
\[
-\frac{u}{v} = -\frac{\rank_{\KK[\sigma]}\eM}{\rank_{\KK[t]}\eM} \rassign w(\eM),
\]
called the \emph{weight of $\eM$}.
Moreover, any subquotient of a pure $t$-motive is still pure of the same weight (see \cite[Proposition~2.4.10~(c)]{HJ20}).
\end{Subsubsec}

\subsection{CM fields and CM types}\label{sec: CM}

A finite  extension $\eF$ over $\ek$ is called \emph{totally real} if the infinite place $\infty$ of $\ek$ splits completely in $\eF$. We note that $\eF$ is then necessarily separable over $\ek$, and every finite extension $\eK/\ek$ contains a maximal totally real intermediate field. 
A \emph{CM field over $\ek$} is a finite separable extension $\eK$ over $\ek$
such that every place $\infty^{+}$ of its maximal real subfield $\eK^{+}$ lying over the infinite place $\infty$ of $\ek$ is non-split in $\eK$.

Let $\eK$ be a CM field over $\ek$.
As in the classical case, one may consider CM types as special collections of $\FF_q$-algebra embeddings from $\eK$ into $\CC_\infty$ sending $t$ to $\theta$.
For our purpose, we shall follow \cite[Definition 3.3.2]{BCPW22} to define generalized CM types using the geometric model of the CM function field $\eK$.
In concrete terms, suppose the constant field of $\eK$ is $\FF_{q^\ell}$ with $\ell\geq 1$.
Then $\eK$ is geometric over $\FF_{q^\ell}(t)$.
Let $X$ be the smooth, projective, geometrically connected algebraic curve defined over $\FF_{q^\ell}$ with function field $\eK$.
Let $\eK^+$ be the maximal totally real subfield of $\eK$ (which is always geometric over $\ek$), and let $X^+$ be the curve over $\FF_q$ with function field $\eK^+$.
The embeddings 
$\ek\hookrightarrow \eK^+\hookrightarrow \eK$ induce the $\FF_q$-morphisms (viewing $X$ as an $\FF_q$-scheme) $\pi_{X/X^+}:X\rightarrow X^+$, $\pi_{X^+/\PP^1}:X^+\rightarrow \PP^1$, and $\pi_{X/\PP^1}:X\rightarrow \PP^1$
so that
$$
\pi_{X/\PP^1} = \pi_{X^+/\PP^1}\circ \pi_{X/X^+}.
$$

Let $\bX \assign \bar{k} \times_{\FF_q} X$ and $\bX^+\assign \ok \times_{\FF_q}X^+$ be the base changes of $X$ and $X^+$ from $\FF_q$ to $\bar{k}$, respectively, 
and denote by $\pi_{\bX/\PP^1}$, $\pi_{\bX^+/\PP^1}$, and $\pi_{\bX/\bX^+}$ the base changes of respective morphisms to $\bar{k}$.
Set
\[J_\eK\assign  \pi_{\bX/\PP^{1}}^{-1}(\theta) = \{\xi \in \bX(\CC_\infty)\mid \pi_{\bX/\PP^{1}}(\xi) = \theta \},\]
the set of $\CC_\infty$-valued points of $\bX$ lying above $\theta \in \PP^1$ (which is actually a subset of $\bX(\bar{k})$).
The generalized CM types of a given CM field are defined as follows:

\begin{definition}\label{defn: A.CM type}
Let $\eK$ be a CM field over $\ek$ with the maximal totally real subfield $\eK^+$.
Let $I_\eK$ be the free abelian group generated by elements in $J_\eK$.
Let $I_\eK^0$ be the subgroup of $I_\eK$ consisting of elements
$$ \Phi= \sum_{\xi \in J_\eK} m_\xi \xi, \quad m_\xi \in \ZZ,$$
satisfying
\[
 \sum_{\xi\in \pi_{\bX/\bX^+}^{-1}(\xi_1^+)} m_\xi= \sum_{\xi'\in \pi_{\bX/\bX^+}^{-1}(\xi_2^+)} m_{\xi'},
\quad \forall \xi_1^+, \xi_2^+ \in J_{\eK^+}.
\]
We put 
\begin{equation}\label{eqn: weight}
\wt(\Phi)\assign \sum_{\xi\in \pi_{\bX/\bX^+}^{-1}(\xi^+)} m_\xi \quad \text{ for a } \xi^+ \in J_{\eK^+}.
\end{equation}
A  \textit{generalized CM type} of $\eK$ is a nonzero effective divisor in $I_\eK^0$. A \textit{CM type} of $\eK$ is a generalized CM type $\Xi$ of $\eK$ for which $\text{wt}(\Xi) = 1$. 
\end{definition}

\begin{remark}\label{rem: JK-Emb}
Let $\Emb(\eK,\CC_\infty)$ be the set of embeddings $\nu : \eK \hookrightarrow \CC_\infty$ satisfying that $\nu\big|_{\ek} = \nu_\theta$.
We recall the bijection between $J_\eK$ and
$\Emb(\eK,\CC_\infty)$
given in \cite[after Prop.~3.3.4]{BCPW22} as follows:
Take $\xi \in J_\eK$, corresponding to an $\FF_q$-morphism
\[
\Spec(\CC_\infty) \rightarrow X.
\]
As $\xi$ is not $\overline{\FF}_q$-valued (where $\overline{\FF}_q$ is the algebraic closure of $\FF_q$ in $\CC_\infty$),
this morphism must factor through the generic point
$\Spec(\eK)\hookrightarrow X$.
Thus we obtain an $\FF_q$-morphism $\Spec(\CC_\infty)\rightarrow \Spec(\eK)$, which corresponds to an $\FF_q$-algebra embedding $\nu_\xi:\eK\hookrightarrow \CC_\infty$ sending $t$ to $\theta$ (as $\pi_{\bX/\PP^1}(\xi) = \theta$).
Conversely, for $\nu \in \Emb(\eK,\CC_\infty)$, the induced $\FF_q$-morphism
$$
\Spec(\CC_\infty)\stackrel{\nu^*}{\longrightarrow}
\Spec(\eK)\hookrightarrow X
$$
gives a $\CC_\infty$-valued point $\xi_\nu$ of $\bX$ so that $\pi_{\bX/\PP^1}(\xi_\nu) = \theta$ (as $\nu(t) = \theta$).
It is clear that
$$
\xi_{\nu_\xi} = \xi, \quad \forall \xi \in J_\eK, \quad \text{ and } \quad \nu_{\xi_\nu} = \nu, \quad \forall \nu \in \Emb(\eK,\CC_\infty).
$$

In particular, let $\Xi = \sum_{i=1}^d \xi_i$ be a CM type of $\eK$, where $d = [\eK^+:\ek]$.
one has 
\[
J_{\eK^+} = \{ \pi_{\bX/\bX^+}(\xi_i)\mid i=1,...,d\}
\]
whence $\Emb(\eK^+,\CC_\infty) = \{\nu_{\xi_i}\big|_{\eK^+}\mid i = 1,...,d\}$.
This explains the analogy between our CM types and the classical ones.
\end{remark}

\begin{Subsubsec}\label{sub:Res and Inf}
{\it Inflation and restriction.}
Let $\eK_1$ and $\eK_2$ be two CM fields over $\ek$ with $\eK_1 \subset \eK_2$. Let $X_1$ (resp.\ $X_2$) be the $\FF_q$-scheme associated to $\eK_1$ (resp.\ $\eK_2$), and let $\pi:X_2 \rightarrow X_1$ be the $\FF_q$-morphism corresponding to the embedding $\eK_1\hookrightarrow \eK_2$.
The inflation $\Inf_{\eK_2/\eK_1}:I_{\eK_1}\rightarrow I_{\eK_2}$ and restriction $\Res_{\eK_2/\eK_1}:  I_{\eK_2}\rightarrow I_{\eK_1}$ are defined via the pull-back and push-forward, respectively:
\begin{align*}
\Inf_{\eK_2/\eK_1}(\Phi_1) &\assign  \pi^*(\Phi_1), \quad \forall \Phi_1 \in I_{\eK_1},\\
\Res_{\eK_2/\eK_1}(\Phi_2) &\assign  \pi_*(\Phi_2), \quad \forall \Phi_2 \in I_{\eK_2}.
\end{align*}
One checks that
\[
\Inf_{\eK_2/\eK_1}\bigl(I_{\eK_1}^0\bigr) \subset I_{\eK_2}^0 \quad \textup{and} \quad
\Res_{\eK_2/\eK_1}\bigl(I_{\eK_2}^0\bigr) \subset I_{\eK_1}^0,
\]
which in particular implies that for each generalized CM type $\Xi_1$ of $\eK_1$ \textup{(resp.\ $\Xi_2$ of $\eK_2$)}, the divisor $\Inf_{\eK_2/\eK_1}(\Xi_1)$ \textup{(resp.\ $\Res_{\eK_2/\eK_1}(\Xi_2)$)} is a generalized CM type of $\eK_2$ \textup{(resp.\ $\eK_1$)}.
\end{Subsubsec}

\subsection{\texorpdfstring{CM dual $t$-motives}{CM dual t-motives}}

We retain the notations of the last subsection with a CM field $\eK$ over $\ek$ whose constant field is $\FF_{q^\ell}$, and fix an algebraically closed subfield $\KK$ of $\CC_\infty$ containing $\ok$.
Let $O_\eK$ be the integral closure of $\eA = \FF_q[t]$ in $\eK$.
For each $\xi \in J_\eK$, the corresponding embedding $\nu_\xi$ induces a ring homomorphism
\[
\bar{\nu}_\xi : O_\bK \assign \KK \otimes_{\FF_q}O_\eK \twoheadrightarrow \KK\subset \CC_\infty.
\]
Set $\Pfk_\xi \assign \ker \bar{\nu}_\xi$, which is a maximal ideal of $O_\bK$ so that $\Pfk_\xi \cap \KK[t] = (t-\theta)\KK[t]$.
The map $(\xi\mapsto \Pfk_\xi)$ gives a one-to-one correspondence between $J_\eK$ and the set of maximal ideals of $O_\bK$ lying over $(t-\theta)\KK[t]$, and
\[ 
(t-\theta) \cdot O_\bK = \prod_{\xi \in J_\eK} \Pfk_\xi.
\]
\begin{definition}\label{defn: CM dual-t-motive}
A \emph{CM dual $t$-motive over $\KK$} is a dual $t$-motive $\eM$ over $\KK$ satisfying:
\begin{itemize}
    \item[(1)] There exists a CM field $\eK$ over $\ek$ with $[\eK:\ek] = \rank_{\KK[t]}(\eM)$ and an $\eA$-algebra homomorphism $O_\eK\rightarrow \End_{\KK[t,\tau]}(\eM)$, which induces an $O_\bK$-module structure on $\eM$ so that $\eM$ is projective of rank one over $O_\bK$ (see \cite[Lemma C.1.1]{BCPW22}).
    \item[(2)] There exists a generalized CM type $\Xi = \sum_{\xi \in J_\eK} m_\xi \xi$ of $\eK$
    such that 
    \[
    \sigma \eM = \Ifk_\Xi \cdot \eM, \quad \text{ where } \Ifk_\Xi \assign \prod_{\xi \in J_\eK} \Pfk_\xi^{m_\xi}.
    \]
    \item[(3)] $\eM$ is pure. 
\end{itemize}
In this case, we say that $\eM$ has generalized CM type $(\eK,\Xi)$.
\end{definition}

\begin{remark}
In \cite{BCPW22}, the definition of CM dual $t$-motives comes from geometric constructions following \cite{Sinha97}, \cite{BP02} and \cite{ABP}.
It is shown in \cite[Proposition~C.1.3]{BCPW22} that these two definitions are equivalent.
\end{remark}

Recall the following essential properties of CM dual $t$-motives (see \cite[Theorem 1.2.5 and Remark 1.2.6]{BCPW22}).

\begin{theorem}\label{thm: CM prop}
${}$
\begin{itemize}
    \item[(1)] Given a CM field $\eK$ and a generalized CM type $\Xi$ of $\eK$, there exists a CM dual $t$-motive with generalized CM type $(\eK,\Xi)$ over $\bar{k}$.
    \item[(2)] Every CM dual $t$-motive is uniformizable (and pure).
    \item[(3)] Two CM dual $t$-motives over $\KK$ with the same generalized CM type must be isogenous over $\KK$.
\end{itemize}
\end{theorem}

\section{Period symbols and the period distribution}

\subsection{Period symbols}\label{sec: defn PS}

Let $\eM$ be a dual $t$-motive over $\bar{k}$.
The \emph{de Rham module of $\eM$ over $\ok$} is (see \cite[Section 4.5]{HJ20}):
\[
H_{\mathrm{dR}}(\eM,\ok)\assign 
\Hom_{\bar{k}}(\eM/(t-\theta)\eM, \bar{k}).
\]
In general, for an algebraically closed subfield $\KK$ of $\CC_\infty$ containing $\ok$,
we set 
\[
H_{\mathrm{dR}}(\eM,\KK)\assign \KK \otimes_{\ok} H_{\mathrm{dR}}(\eM,\ok) \cong \Hom_{\KK}(\eM_\KK/(t-\theta)\eM_\KK, \KK), \quad 
\text{where} \quad \eM_\KK \assign \KK\otimes_{\ok} \eM.
\]
We also  identify $H_{\mathrm{dR}}(\eM,\KK)$ with the space of $\KK$-linear functionals on $\eM_\KK$ which factor through $\eM_\KK/(t-\theta) \eM_\KK$.
Elements of $H_{\mathrm{dR}}(\eM,\KK)$ are called \emph{differentials of $\eM$ over $\KK$}.

Suppose $\eM$ is uniformizable.
Given $\omega \in H_{\mathrm{dR}}(\eM,\bar{k})$,
we naturally extend $\omega$ to a $\CC_{\infty}$-linear functional on $\MM^\dagger = \TT^\dagger \otimes_{\bar{k}[t]}\eM$ by setting
\[
\omega(f\otimes m)\assign  f(\theta)\cdot \omega(m), \quad \forall f \in \TT^\dagger \text{ and }  m \in \eM,
\]
and this induces
a $\CC_\infty$-linear functional $\omega : \MM^\dagger/(t-\theta)\MM^\dagger \rightarrow \CC_\infty$.
Suppose $\eM$ is uniformizable.
The \textit{de Rham pairing} of $\eM$ is defined as follows:
\begin{equation}\label{E:deRham pairing}
H_{\mathrm{dR}}(\eM,\bar{k}) \times H_{\mathrm{Betti}}(\eM)
\longrightarrow \CC_\infty, \quad (\omega,\gamma) \longmapsto \int_{\gamma} \omega \assign  \omega(\gamma),
\end{equation}
This pairing is non-degenerate
(see \cite[Lemma~5.1.2]{BCPW22}).
Every period integral $\int_{\gamma} \omega$ is called a \textit{period of $\eM$}.
In particular, let
\[
H^\eM \assign \Hom_{\eA}(H_{\mathrm{Betti}}(\eM),\ek)
\quad \text{ and } \quad 
H^\eM_{\CC_\infty} \assign \CC_\infty \underset{v_\theta,\ek}{\otimes} H^\eM \cong \Hom_{\eA}(H_{\mathrm{Betti}}(\eM),\CC_\infty),
\]
where the $\eA$-module structure on $\CC_\infty$ comes from the evaluation map $v_\theta: \ek \cong k\subset \CC_\infty$.
The de~Rham pairing induces the so-called \emph{period isomorphism}
\begin{equation}\label{eqn: period-iso}
H_{\mathrm{dR}}(\eM,\CC_\infty) \cong H^\eM_{\CC_\infty}.
\end{equation}

Given two uniformizable dual $t$-motives $\eM$ and $\eM'$ over $\KK$,
let $f:\eM \rightarrow \eM'$ be a morphism over $\KK$. The induced $\KK$-linear homomorphism $f^*: H_{\mathrm{dR}}(\eM',\KK) \rightarrow H_{\mathrm{dR}}(\eM,\KK)$ defined by
\[
f^*\omega' (m)\assign \omega'(f(m)), \quad \forall m \in \eM \text{ and } \omega' \in H_{\mathrm{dR}}(\eM',\KK),
\]
is called the \emph{pull-back of $f$}.
Recall that $f_*: H_{\mathrm{Betti}}(\eM)\rightarrow H_{\mathrm{Betti}}(\eM')$ is the push-forward of $f$ (see Section~\ref{subsec: push-forward}).
The following equality holds:
\begin{equation}\label{eqn: pf-pb}
\int_{f_*\gamma}\omega' = \int_{\gamma}f^*\omega',\quad \forall \gamma \in H_{\mathrm{Betti}}(\eM),\ \omega' \in H_{\mathrm{dR}}(\eM',\KK).
\end{equation}
Consequently, let $f_*^\vee: H^{\eM'}_{\CC_\infty} \rightarrow H^{\eM}_{\CC_\infty}$ be the $\CC_\infty$-linear map induced by $f_*$.
The following diagram commutes:
\begin{equation}\label{eqn: comp-period-iso}
\SelectTips{cm}{}
\xymatrix{
H_{\mathrm{dR}}(\eM',\CC_\infty) \ar[r]^>>>>>{\sim} \ar[d]_{f^*} & H^{\eM'}_{\CC_\infty}\ar[d]^{f_*^{\vee}} \\
H_{\mathrm{dR}}(\eM,\CC_\infty) \ar[r]^>>>>>{\sim} & H^{\eM}_{\CC_\infty},
}
\end{equation}

\begin{remark}\label{rem: ess-surj}
Let $\eM$ and $\eM'$ be two dual $t$-motives over $\KK$.
A morphism $f:\eM\rightarrow \eM'$ (over $\KK$) is \emph{essentially surjective} if $f(\eM)$ is isogenous to $\eM'$ via the inclusion.
In this case, we get that $f_*^\vee : H^{\eM'}_{\CC_\infty} \rightarrow H^{\eM}_{\CC_\infty}$ is injective, whence so is $f^*: H_{\mathrm{dR}}(\eM',\KK)\rightarrow H_{\mathrm{dR}}(\eM,\KK)$ by the above commutative diagram.
\end{remark}

Now, let $\eM$ be a CM dual $t$-motive over $\bar{k}$ with generalized CM type $(\eK,\Xi)$.
For each $\xi \in J_\eK$, there exists a 
nonzero differential $\omega_{\eM,\xi} \in H_{\mathrm{dR}}(\eM,\bar{k})$, unique up to a $\bar{k}^{\times}$-multiple, satisfying (see \cite[Proposition~5.2.1]{BCPW22})
\begin{equation}\label{eqn: omegaM-xi}
\omega_{\eM,\xi} (\alpha \cdot m) = \nu_\xi(\alpha) \cdot \omega_{\eM,\xi}(m), \quad \text{$\forall\, m \in \eM$ and $\alpha \in O_\eK$.}
\end{equation}
We call $\omega_{\eM,\xi}$ the \emph{differential of $\eM$ associated with $\xi$}.
Since $H_{\rm Betti}(\eM)$ is projective of rank one over $O_\eK$ and any other CM dual $t$-motive with the same CM type $(\eK,\Xi)$ must be isogenous to $\eM$, the period $\int_{\gamma} \omega_{\eM,\xi}$ for a nonzero cycle $\gamma \in H_{\rm Betti}(\eM)$ is, up to a $\bar{k}^\times$-multiple, uniquely determined by $\eK$, $\Xi$, and $\xi$ alone.

\begin{definition}[See \emph{\cite[Def.~5.2.2]{BCPW22}}] \label{defn: period quantity}
Let $\eK$ be a CM field over $\ek$ and $\Xi$ be a generalized CM type of $\eK$. 
Given $\xi \in J_\eK$, the {\it period symbol} $\Pcal_\eK(\xi,\Xi)$ is the coset in $\CC_\infty^\times/\bar{k}^\times$ represented by the period
$\int_{\gamma} \omega_{\eM,\xi}$
for any CM dual $t$-motive $\eM$ with generalized CM type $(\eK,\Xi)$ over $\bar{k}$ and a nonzero cycle $\gamma \in H_{\rm Betti}(\eM)$.

We also use the notation $p_\eK(\xi,\Xi) \in \CC_\infty^\times$ for an arbitrary representative of $\Pcal_\eK(\xi,\Xi)$, and call it a period symbol as well if there is no risk of confusion.
\end{definition}

\begin{remark}\label{rem: Carlitz period}
We denote by $x\sim y$ for $x,y \in \CC_\infty^\times$ that $x/y \in \bar{k}^\times$.
When $\eK = \ek$ and $\Xi = n\xi_\theta$, where $\xi_\theta$ is the point $t=\theta \in \PP^1_{/\bar{k}}$,
we have that (see \cite[Proposition~5.2.3]{BCPW22})
\[
p_{\ek}(\xi_\theta, n\xi_\theta) \sim \tilde{\pi}^n, 
\]
where $\tilde{\pi}$ is a fundamental period of the Carlitz $\FF_{q}[t]$-module: 
\[
\tilde{\pi} \assign (-\theta)^{\frac{q}{q-1}} \cdot \prod_{i=1}^\infty \left(1-\frac{\theta}{\theta^{q^i}}\right)^{-1}.
\]
\end{remark}

The needed properties of period symbols are stated in the following (see \cite[Theorem~1.3.1]{BCPW22}):

\begin{theorem}\label{thm: PS}
Let $\eK$ be a CM field over $\ek$.
For each $\xi \in J_\eK$,
we may extend $\Pcal_\eK(\xi,\cdot)$ to a unique group homomorphism 
from $I_\eK^0$ to $\CC_\infty^\times/\bar{k}^\times$ satisfying the following:
\begin{itemize}
\item[(1)] Let $\eK'$ be a CM field containing $\eK$. For $\xi' \in J_{\eK'}$ and $\Phi^0 \in I_\eK^0$, 
\[
    \Pcal_{\eK}\big({\Res}_{\eK'/\eK}(\xi'),\Phi^0\big) = \Pcal_{\eK'}\big(\xi',{\Inf}_{\eK'/\eK}(\Phi^0)\big),
\]
where ${\Res}_{\eK'/\eK}$ and ${\Inf}_{\eK'/\eK}$ are restriction and inflation, respectively (see Section~\ref{sub:Res and Inf}).
\item[(2)] For each $\varrho \in \Aut(\eK/\ek)$, 
\[
    \Pcal_{\eK}\big(\xi^{\varrho},\Phi^0\big) = \Pcal_\eK\big(\xi,(\Phi^0)^{\varrho^{-1}}\big), \quad \forall \xi \in J_\eK,\ \Phi^0 \in I_\eK^0,
\]
where $\xi^\varrho \in J_\eK$ corresponds to the embedding $\nu_\xi \circ \varrho$ and $(\Phi^0)^{\varrho} \assign \sum_{\xi} m_\xi \xi^\varrho$.
\item[(3)] Let $\eK^+$ be the maximal totally real subfield of $\eK$. Given $\xi^+ \in J_{\eK^+}$ and $\Phi^0 \in I_\eK^0$, we have the Legendre relation
\[
\prod_{\substack{\xi \in J_\eK \\ \pi_{\bX/\bX^+}(\xi) = \xi^+}}p_\eK(\xi,\Phi^0)
\sim \tilde{\pi}^{\text{\rm wt}(\Phi^0)}.
\]
\end{itemize}
\end{theorem}

Next, we shall use these properties to introduce an analogue of Anderson's period distribution on the space of ``Stickelberger functions''.

\subsection{Stickelberger functions}
Let $\eG \assign \gal(\ek^{\text{sep}}/\ek)$ where  $\ek^{\text{sep}}$ is a separable closure  of $\ek$.
Fix an embedding $\nu: \ek^{\text{sep}} \hookrightarrow \bar{k} \subset \CC_\infty$ which extends the evaluation map $v_\theta:\ek\cong k$ sending $t$ to $\theta$.
The Galois group $\gal(k^{\rm sep}_\infty/k_\infty)$ is identified with a subgroup $\eG_\infty$ of $\eG$ via $\nu$, which is the decomposition group of a place of $\ek^{\text{sep}}$ lying over $\infty$.

\begin{definition}\label{defn: S(G)}
Let $\Sscr(\eG)$ be the space consisting of all locally constant $\QQ$-valued functions $\varphi$ on $\eG$ satisfying the following two conditions:
\begin{itemize}
    \item[(1)] For $\varrho_1,\varrho_2 \in \eG$ and $\varrho_\infty \in \eG_\infty$, we have
    $$\varphi\big(\varrho_1 \cdot (\varrho_2 \varrho_\infty \varrho_2^{-1}\varrho_\infty^{-1})\big) = \varphi(\varrho_1);$$
    \item[(2)]
    Let $d\varrho_\infty$ be the Haar measure on $\eG_\infty$ normalized so that $\text{vol}(\eG_\infty,d\varrho_\infty) = 1$.
    Then
    $$\int_{\eG_\infty}\varphi(\varrho_1\varrho_\infty) d\varrho_\infty = \int_{\eG_\infty}\varphi(\varrho_2\varrho_\infty) d\varrho_\infty \quad \forall \varrho_1,\varrho_2 \in \eG.$$ 
\end{itemize}
\end{definition}

\begin{remark}\label{rem: S(G)}
${}$
\begin{itemize}
    \item[(1)]
    The space $\Sscr(\eG)$ is endowed with a left action of $\eG$ under right translations:
    $$
    (\varrho_0 \cdot \varphi) (\varrho)\assign \varphi(\varrho\varrho_0), \quad \forall \varrho_0, \varrho \in \eG \text{ and } \varphi \in \Sscr(\eG).
    $$
Given $\varphi \in \Sscr(\eG)$, let $\eH_\varphi$ be the stabilizer of $\varphi$ in $\eG$, which is of finite index in $\eG$.
The condition (1) of Definition~\ref{defn: S(G)} says that the commutator subgroup $[\eG,\eG_\infty]$ is contained in $\eH_\varphi$,
which implies in particular that $\eG_\infty \eH_\varphi$ becomes a subgroup of $\eG$ with finite index.
\item[(2)]
Given a subgroup $\eH$ of $\eG$ with $\eH \supset [\eG,\eG_\infty]\ (\supset [\eG_\infty,\eG_\infty])$,
one observes that 

\begin{tabular}{cl}
\\[-1em]
(a) & $\varrho \eG_\infty \varrho^{-1} \subset \eG_\infty \eH$ for every $\varrho \in \eG$; \\[0.15cm]
(b) & 
$\eH$ is normal in $\eG_\infty \eH$ and 
$\displaystyle\frac{\eG_\infty \eH}{\eH} \cong \frac{\eG_\infty}{\eG_\infty \cap \eH}$ is an abelian group.\\[-1em]
& 
\end{tabular}

\noindent When $\eH = \eH_\varphi$ for a  function $\varphi \in \Sscr(\eG)$,
let 
$\eK_\varphi$ and $\eK^+_\varphi$ be the fixed field of $\eH_\varphi$ and $\eG_\infty\eH_\varphi$ in $\ek^{\text{sep}}$, respectively. 
Then (a) implies that $\eK_\varphi^+$ is a totally real field over $\ek$, and (b) shows that
$\eK_\varphi$ is a CM field over $\ek$ with maximal totally real subfield $\eK_\varphi^+$ satisfying that the extension $\eK_\varphi/\eK_\varphi^+$ is abelian with $\gal(\eK_\varphi/\eK_\varphi^+) \cong \eG_\infty \eH_\varphi / \eH_\varphi$.
\item[(3)] Given two CM fields $\eK_1$ and $\eK_2$, the compositum $\eK_1\cdot \eK_2$ may not be a CM field in general (cf.\ \cite[Remark~3.1.2~(4)]{BCPW22}).
However, for $\varphi_1, \varphi_2 \in \Sscr(\eG)$, 
the condition (1) of Definition~\ref{defn: S(G)} ensures that
$$
[\eG,\eG_\infty] \subset \eH_{\varphi_1}\cap \eH_{\varphi_2}.
$$
This implies that the compositum $\eK= \eK_{\varphi_1}\cdot \eK_{\varphi_2}$ is still a CM field with 
\[
\gal(\ek^{\text{sep}}/\eK) = \eH_{\varphi_1}\cap \eH_{\varphi_2},
\]
and its maximal totally real subfield $\eK^+$ is the fixed field of $\eG_\infty(\eH_{\varphi_1}\cap \eH_{\varphi_2})$. 
    \item[(4)] The condition (1) of Definition~\ref{defn: S(G)} says in particular that the space $\Sscr(\eG)$ is independent of the chosen embedding $\nu$.
    Indeed, let $\nu':\ek^{\rm sep}\hookrightarrow \bar{k}$ be another embedding and $\eG_\infty'$ be the corresponding decomposition group.
    There exists $\varrho_0 \in \eG$ so that $\nu' = \nu \circ \varrho_0$, whence
    $\eG_\infty' = \varrho_0^{-1} \eG_\infty \varrho_0$.
    Notice that 
    $$[\eG,\eG_\infty] = [\eG,\varrho_0^{-1}\eG_\infty\varrho_0] = [\eG,\eG_\infty'].$$
    Thus for a locally constant function $\varphi$ on $\eG$, the condition that $[\eG,\eG_\infty]$ is contained in $\eH_\varphi$ is equivalent to saying that $[\eG,\eG_\infty']$ is contained in $\eH_\varphi$, and for every $\varrho \in \eG$ and $\varrho_\infty \in \eG_\infty$,
    $$
    \varphi(\varrho \varrho_\infty) = \varphi\big((\varrho \varrho_\infty) (\varrho_\infty^{-1}\varrho_0^{-1} \varrho_\infty \varrho_0)\big) = \varphi(\varrho (\varrho_0^{-1} \varrho_\infty \varrho_0)). 
    $$
    Therefore the assertion holds.
    \item[(5)]
    Given an open subgroup $\eH$ of $\eG$ with $[\eG,\eG_\infty] \subset \eH$,
    set
\[
    \Sscr(\eG/\eH) \assign \{\varphi \in \Sscr(\eG)\mid \varrho \cdot \varphi = \varphi, \quad \forall \varrho \in \eH\}.
\]
Every $\varphi \in \Sscr(\eG/\eH)$ can be written in the following unique form:
\[
\varphi = \sum_{\varrho \eH \in \eG/\eH}m_{\varrho \eH}(\varphi) \cdot \mathbf{1}_{\varrho \eH},
\]
where $m_{\varrho \eH}(\varphi) = \varphi(\varrho)$ for each coset $\varrho \eH \in \eG/\eH$ and $\mathbf{1}_{\varrho \eH}$ is the characteristic function of $\varrho \eH$.
Moreover, put $\eH^+\assign \eG_\infty \eH$ and $\eH_\infty \assign \eG_\infty \cap \eH$.
For each $\varphi \in \Sscr(\eG/\eH)$, the isomorphism $\eG_\infty/\eH_\infty \cong \eH^+/\eH$ enables us to translate the condition~(2) in Definition~\ref{defn: S(G)} into
\[
\sum_{\varrho_\infty \eH_\infty \in \eG_\infty /\eH_\infty}m_{\varrho_1\varrho_\infty \eH}(\varphi) = \sum_{\varrho_\infty \eH_\infty \in \eG_\infty /\eH_\infty}m_{\varrho_2\varrho_\infty \eH}(\varphi), \quad \forall \varrho_1, \varrho_2 \in \eG.
\]
\end{itemize} 
\end{remark}

Let $\eK\subset \ek^{\text{sep}}$ be a CM field satisfying that
$\eH_\eK \assign  \Gal(\ek^{\text{sep}}/\eK) \supset [\eG,\eG_\infty]$.
For each $\xi \in J_\eK$, 
take $\varrho_\xi \in \eG$ so that the embedding $\nu_\xi \in \Emb(\eK,\CC_\infty)$ satisfies $$\nu_\xi = (\nu \circ \varrho_\xi)\big|_\eK.$$
The coset $\varrho_\xi \eH_\eK$ in $\eG$ is uniquely determined by $\xi$.
Conversely, given a coset $\varrho \eH_\eK$ in $\eG$, there exists a unique point $\xi_\varrho \in J_\eK$ so that 
\[
\nu_{\xi_\varrho} = (\nu \circ \varrho)\big|_\eK
\quad \in \Emb(\eK,\CC_\infty).
\]
This gives us a bijection (with respect to the fixed $\nu: \ek^{\sep} \cong k^{\sep} \subset \CC_\infty$):
\begin{eqnarray}\label{eqn: JK-coset}
    J_\eK & \stackrel{\sim}{\longleftrightarrow}&  \eG/\eH_\eK  \\
    \xi & \longmapsto & \varrho_\xi \eH_\eK
    \nonumber \\
    \xi_{\varrho} & \longleftmapsto & \varrho \eH_\eK. \nonumber 
\end{eqnarray}

\begin{remark}\label{rem: JK+-coset}
Given another embedding $\nu':\ek^{\sep} \hookrightarrow \CC_\infty$,
there exists $\varrho' \in \eG$ so that $\nu' = \nu \circ \varrho'$.
Hence the coset corresponding to a given $\xi \in J_\eK$ under the above bijection with respect to $\nu'$ becomes $(\varrho')^{-1}\varrho_\xi \eH_\eK$.
\end{remark}
Let $\eK^+$ be the maximal totally real subfield of $\eK$, and $\eH_{\eK^+} = \gal(\ek^{\sep}/\eK^+) = \eG_\infty \eH_\eK$.
Put $\eH_{\eK,\infty}\assign \eG_\infty \cap \eH_\eK$.
Then \eqref{eqn: JK-coset} induces the following one-to-one correspondence for each $\xi^+ \in J_{\eK^+}$:
\begin{equation}\label{eqn: xi+coset}
\{\xi \in J_\eK\mid \pi_{\bX/\bX^+}(\xi) = \xi^+\} \stackrel{\sim}{\longleftrightarrow}
\varrho_{\xi^+}\eH_{\eK^+}/\eH_\eK \assign 
\big\{\varrho_{\xi^+}\varrho_\infty \eH_\eK\mid \varrho_\infty \eH_{\eK,\infty} \in \eG_\infty/\eH_{\eK,\infty}\big\}.
\end{equation}
From Remark~\ref{rem: S(G)}~(5),
we may write each $\varphi \in \Sscr(\eG/\eH_\eK)$ in the following unique form:
\begin{equation}\label{eqn: phi-1}
\varphi = \sum_{\xi \in J_\eK} m_\xi(\varphi) \cdot \mathbf{1}_{\varrho_\xi \eH_\eK}, \quad \text{where $m_\xi(\varphi) = \varphi(\varrho_\xi)$ for every $\xi \in J_\eK$,}
\end{equation}
and the condition (2) in Definition~\ref{defn: S(G)} becomes:
\begin{equation}\label{eqn: cond (2)-equiv}
\sum_{\substack{\xi_1 \in J_\eK \\ \pi_{\bX/\bX^+}(\xi_1) = \xi_1^+}}
m_{\xi_1}(\varphi) = \sum_{\substack{\xi_2 \in J_\eK \\ \pi_{\bX/\bX^+}(\xi_2) = \xi_2^+}}
m_{\xi_2}(\varphi), \quad \forall \xi_1^+,\xi_2^+ \in J_{\eK^+}.
\end{equation}

Given $\Phi^0 \in I_\eK^0$,
write 
\[
\Phi^0 = \sum_{\xi \in J_\eK} m_\xi \xi, \quad \text{ with\quad $m_\xi \in \ZZ$.}
\]
We set
\begin{equation}\label{eqn: phi-K-Phi}
\varphi_{\eK,\Phi^0}\assign \sum_{\xi \in J_\eK} m_\xi \cdot \mathbf{1}_{\varrho_\xi \eH_\eK},
\end{equation}

\begin{lemma}\label{lem: S(G)}
Let $\eK \subset \ek^{\sep}$ be a CM field satisfying that $\eH_\eK \assign \gal(\ek^{\sep}/\eK)$ contains the commutator subgroup $[\eG,\eG_\infty]$.
\begin{itemize}
\item[(1)] 
For every $\Phi^0 \in I_\eK^0$, we have that $\varphi_{\eK,\Phi^0} \in \Sscr(\eG/\eH_\eK)$.
Moreover, the homomorphism
$\varphi_\eK\assign (\Phi^0\mapsto \varphi_{\eK,\Phi^0}): I_\eK^0\rightarrow \Sscr(\eG/\eH_\eK)$ is injective.
\item[(2)] The space $\Sscr(\eG/\eH_\eK)$ consists of all the rational multiples of $\varphi_{\eK,\Phi^0}$ for $\Phi^0 \in I_\eK^0$, for which $\varphi_\eK$ induces a $\QQ$-linear isomorphism $\QQ\otimes_\ZZ I_\eK^0 \cong \Sscr(\eG/\eH_\eK)$.
\end{itemize}
\end{lemma}

\begin{proof}
Comparing the condition of $\Phi^0$ in Definition~\ref{defn: A.CM type} with \eqref{eqn: cond (2)-equiv}, we get that $\varphi_{\eK,\Phi^0}$ lies in $\Sscr(\eG/\eH_\eK)$.
Suppose $\varphi_{\eK,\Phi^0} = 0$.
As distinct cosets of $\eH_\eK$ in $\eG$ are disjoint, the expression of $\varphi_{\eK,\Phi^0}$ in \eqref{eqn: phi-K-Phi} tells us that $m_\xi = 0$ for every $\xi \in J_\eK$, whence $\Phi^0 = 0$.
This implies that $\varphi_\eK$ is injective and completes the proof of (1).

To prove $(2)$, take $\varphi \in \Sscr(\eG/\eH_\eK)$ and suppose $\varphi$ is $\ZZ$-valued without loss of generality.
Expressing $\varphi$ as in \eqref{eqn: phi-1},
we get $m_\xi(\varphi) \in \ZZ$, and \eqref{eqn: cond (2)-equiv} ensures that
\[
\Phi^0_\varphi \assign \sum_{\xi \in J_\eK} m_\xi(\varphi) \xi\quad  \in I_\eK^0.
\]
It is clear that $\varphi_{\eK,\Phi^0_\varphi} = \varphi$ and the result holds.
\end{proof}

\begin{remark}\label{rem: S(G/H)-dim}
For each CM field $\eK\subset \ek^{\sep}$ satisfying that $\eH_\eK = \gal(\ek^{\sep}/\eK)$ contains the commutator subgroup $[\eG,\eG_\infty]$, the isomorphism in Lemma~\ref{lem: S(G)}~(2) says that
\[
\dim_{\QQ}\Sscr(\eG/\eH_\eK) = \rank_\ZZ I_\eK^0 = 1+ (1-\frac{1}{[\eK:\eK^+]})[\eK:\ek].
\]
\end{remark}

The following lemma shows that the embeddings $\varphi_\eK: I_\eK^0\rightarrow \Sscr(\eG/\eH_\eK)\subset \Sscr(\eG)$ factor through the inflation maps:

\begin{lemma}\label{lem: Inf}
Let $\eK, \eK' \subset \ek^{\text{\rm sep}}$ be two CM fields over $\ek$ fixed by the commutator subgroup $[\eG,\eG_\infty]$ and $\eK\subset \eK'$.
For each $\Phi^0 \in I_\eK^0$, we have
$$
\varphi_{\eK,\Phi^0} = \varphi_{\eK',\text{\rm Inf}_{\eK'/\eK}(\Phi^0)} \quad \in \Sscr(\eG).
$$
\end{lemma}

\begin{proof}
Let $X$ and $X'$ be the curves corresponding to $\eK$ and $\eK'$, respectively.
Let
$\pi_{X'/X}:X'\rightarrow X$ be the $\FF_q$-morphism corresponding to the inclusion $\eK\subset \eK'$, and $\pi_{\bX'/\bX}$ its base change from $\FF_q$ to $\ok$.
Note that for $\xi' \in J_{\eK'}$ and $\xi \in J_{\eK}$, one has that 
$\pi_{\bX'/\bX}(\xi') = \xi$ if and only if $\nu_{\xi'}\big|_\eK = \nu_\xi$.
Let $\varrho_{\xi'} \eH_{\eK'}$ and $\varrho_\xi \eH_\eK$ be the cosets in $\eG$ corresponding to $\xi'$ and $\xi$, respectively.
Then
$$
\nu_{\xi'}\big|_\eK = \nu_\xi \quad \text{ if and only if } \quad 
\varrho_{\xi'}\eH_{\eK'} \subset \varrho_\xi \eH_\eK.
$$
Given $\Phi^0 \in I_\eK^0$, write
$$
\Phi^0 = \sum_{\xi \in J_\eK}m_\xi \cdot \xi \quad \text{ and so } \quad
\text{Inf}_{\eK'/\eK}(\Phi^0) = 
\sum_{\xi \in J_\eK}m_\xi \cdot \sum_{\subfrac{\xi' \in J_{\eK'}}{\pi_{\bX'/\bX}(\xi') = \xi}} \xi'.
$$
Therefore we have
\begin{eqnarray}
\varphi_{\eK,\Phi^0}
&=& \sum_{\xi \in J_\eK} m_\xi \cdot \mathbf{1}_{\varrho_\xi \eH_\eK} \nonumber \\
&=& \sum_{\xi \in J_\eK} m_\xi \cdot \sum_{\varrho' \in \eH_\eK/\eH_{\eK'}} \mathbf{1}_{\varrho_\xi \varrho' \eH_{\eK'}} \nonumber \\
&=& \sum_{\xi \in J_\eK} m_\xi \cdot \sum_{\subfrac{\xi' \in J_{\eK'}}{\pi_{X'/X}(\xi') = \xi}} \mathbf{1}_{\varrho_{\xi'} \eH_{\eK'}} \nonumber \\
&=& \varphi_{\eK',\text{Inf}_{\eK'/\eK}(\Phi^0)}. \nonumber
\end{eqnarray}
\end{proof}

\begin{remark}\label{rem: S(G)-iso}
For each $\varphi \in \Sscr(\eG)$, recall that $\eH_\varphi$ is the stabilizer of $\varphi$ in $\eG$ and $\eK_\varphi$ is the fixed field of $\eH_{\varphi}$.
By Remark~\ref{rem: S(G)}~(2), we have that $\eK_\varphi\subset \ek^{\sep}$ is a CM field, and
$\eH_{\eK_\varphi} = \eH_\varphi$.
Hence $\varphi \in \Sscr(\eG/\eH_{\eK_\varphi})$.
In conclusion,
Lemma~\ref{lem: S(G)} and \ref{lem: Inf} induces a $\QQ$-linear isomorphism
\begin{align*}
    \QQ \otimes_\ZZ \Big(\varinjlim_{\eK} I_\eK^0\Big) &\stackrel{\sim}{\longrightarrow} \Sscr(\eG), \\
    I_{\eK}^0\  \ni \ \Phi^0\quad  &\longmapsto \ \varphi_{\eK,\Phi^0},
\end{align*}
where $\eK$ runs through all CM fields in $\ek^{\sep}$ fixed by $[\eG,\eG_\infty]$, and the direct limit is with respect to the inflation maps.
\end{remark}

\subsection{The period distribution}

We introduce an analogue of Anderson's period distribution (see \cite[Proposition~1.5]{Anderson82}) by proving the following theorem:

\begin{theorem}\label{thm: PD}
Fix an embedding $\nu: \ek^{\text{\rm sep}}\hookrightarrow \bar{k}$.
There exists a unique $\QQ$-linear distribution $\Pscr_\nu : \Sscr(\eG)\rightarrow \CC_\infty^\times/\bar{k}^{\times}$ (with respect to $\nu$) satisfying
\begin{eqnarray}\label{eqn: Distribution}
\Pscr_\nu(\varphi_{\eK,\Phi^0}) = \Pcal_\eK(\xi_\nu,\Phi^0), \quad \forall \Phi^0 \in I_\eK^0,
\end{eqnarray}
where $\eK\subset \ek^{\text{\rm sep}}$ is a CM field fixed by the commutator subgroup $[\eG,\eG_\infty]$,
and $\xi_\nu \in J_\eK$ is the unique point satisfying that $\nu_{\xi_\nu} = \nu\big|_\eK$.
\end{theorem}

\begin{proof}
For each CM field $\eK\subset \ek^{\text{\rm sep}}$ fixed by the commutator subgroup $[\eG,\eG_\infty]$,
define $\tilde{\Pscr}_{\nu,\eK}: I_\eK^0 \rightarrow \CC_\infty^\times/\bar{k}^\times$ by
\[
\tilde{\Pscr}_{\nu,\eK}(\Phi^0)\assign \Pcal_\eK(\xi_\nu,\Phi^0), \quad \forall \Phi^0 \in I_\eK^0.
\]
By Lemma~\ref{lem: Inf}, we obtain a $\QQ$-linear homomorphism
\[
\tilde{\Pcal}_\nu \assign \text{id}_\QQ \otimes \left(\varinjlim_{\eK} \tilde{\Pcal}_{\nu,\eK}\right): \QQ \otimes_\ZZ \Big(\varinjlim_{\eK} I_\eK^0\Big) \longrightarrow \CC_\infty^\times/\ok^\times,
\]
and $\Pcal_\nu: \Sscr(\eG)\rightarrow \CC_\infty^\times/\ok^\times$ is nothing but the composition of $\tilde{\Pcal}_\nu$ with the inverse of the isomorphism in Remark~\ref{rem: S(G)-iso}.
\end{proof}

\begin{remark}
Notice that $\mathbf{1}_{\eG} = \varphi_{\ek,\xi_\theta}$, where $\xi_\theta \in J_{\ek}$ is the unique point $t=\theta \in \PP^1_{/\bar{k}}$.
By Remark~\ref{rem: Carlitz period} we have that
$$
\Pscr_\nu(\mathbf{1}_{\eG}) = \Pcal_{\ek}(\xi_\theta,\xi_\theta) = \tilde{\pi} \cdot \bar{k}^\times \quad \in \CC_\infty^\times/\bar{k}^\times.
$$
\end{remark}

Let $\eK$ be a CM field in $\ek^{\sep}$ fixed by $[\eG,\eG_\infty]$.
Suppose further that $\eK$ is Galois over $\ek$.
Let $\eK^+$ be the maximal totally real subfield of $\eK$.
By \cite[Corollary~1.3.6~(2)]{BCPW22}, we know that
\begin{equation}\label{eqn: ps-trdeg}
\trdeg_{\bar{k}} \bar{k}\big(p_\eK(\xi_v,\Phi^0)\mid \Phi^0 \in I_\eK^0\big) = 1+ (1-\frac{1}{[\eK:\eK^+]}) \cdot [\eK:\ek].
\end{equation}
Put $\eG_\eK \assign \gal(\eK/\ek) \cong \eG/\eH_\eK$.
Then $\Sscr(\eG/\eH_\eK)$ can be identified with a subspace $\Sscr(\eG_\eK)$ of $\QQ$-valued functions on $\eG_\eK$.
Let $\bar{k}\big(\Pscr_\nu\big(\Sscr(\eG_\eK)\big)\big)$ be the field generated by arbitrary representatives of $\Pscr_\nu(\varphi) \in \CC_\infty^\times/\bar{k}^\times$ for all $\varphi \in \Sscr(\eG_\eK)$.
By Theorem~\ref{thm: PD} and \eqref{eqn: ps-trdeg}, we obtain:

\begin{theorem}\label{thm: trdeg}
Let $\eK$ be a CM field in $\ek^{\sep}$ which is Galois over $\ek$ and fixed by $[\eG,\eG_\infty]$.
Let $\eK^+$ be the maximal totally real subfield of $\eK$.
Then 
\[
\trdeg_{\bar{k}}\bar{k}\big(\Pscr_\nu\big(\Sscr(\eG_\eK)\big)\big) = 1+ (1-\frac{1}{[\eK:\eK^+]}) \cdot [\eK:\ek].
\]
Consequently,
the period distribution $\Pscr_\nu: \Sscr(\eG)\rightarrow \CC_\infty^\times/\bar{k}^\times$ is injective.
\end{theorem}

\subsection{The restriction to cyclotomic function fields}
\label{sec: cyclotomic extn}

Given $a \in \eA$, recall the {\it $a$-th division polynomial} $C_a(t,z) \in \FF_q[t,z]$ defined recursively as follows: write $a = bt + \epsilon$ where $b \in A$ and $\epsilon \in \FF_q$, then
$$
C_a(t,z) \assign \begin{cases}
0 & \text{ if $a = 0$;}\\
C_b(t,tz+z^q) + \epsilon z, & \text{ if $a \neq 0$.}
\end{cases}
$$
For each $\nfk \in \eA_+$,
the $\nfk$-th Carlitz cyclotomic polynomial $C_\nfk^*(t,z)$ is the unique (irreducible) factor of $C_\nfk(t,z)$ which is monic in $z$ and satisfies (cf.\ \cite[Section~6.3.2]{ABP})
$$
C_\nfk(t,z) = \prod_{\subfrac{\mfk \in \eA_{\scaleto{+}{4pt}}}{\mfk \mid \nfk}} C^*_\mfk(t,z), \quad \forall \nfk \in \eA_+.
$$
Let $O_\nfk \assign \FF_q[t,z]/(C_\nfk^*(t,z))$, which is a Dedekind domain, and let $\eK_\nfk$ be the field of fractions of $O_\nfk$, called the \textit{$\nfk$-th Carlitz cyclotomic function field}.
It is known that $\eK_\nfk /\ek$ is a finite Galois extension with 
$(\eA/\nfk)^\times \cong \gal(\eK_\nfk/\ek)\rassign \eG_\nfk$,
where the isomorphism comes from the Artin map:
for $\alpha = a \bmod \nfk \in (\eA/\nfk)^\times$,
the corresponding $\varrho_\alpha \in \eG_\nfk$ is uniquely determined by the condition that (see~\cite[Theorem~2.3]{Hayes74} or~\cite[Theorem~12.8]{Rosen})
\begin{equation}\label{eqn: geo-Artin}
\varrho_\alpha(z) \bmod C_\nfk^*(t,z) \assign C_a(t,z) \bmod C_\nfk^*(t,z) \in O_\nfk.
\end{equation}
Moreover, $\eK_\nfk$ is a geometric CM field over $\ek$, where every place $\infty^+$ of its maximal totally real subfield $\eK_\nfk^+$ lying over $\infty$ is totally ramified in $\eK_\nfk$ with ramification index $q-1$
(see~\cite[Theorem~3.2]{Hayes74} or~\cite[Theorem~12.14]{Rosen}).
In fact, under the Artin map we have that
\begin{equation}\label{eqn: inertia gp}
\Gal(\eK_\nfk/\eK_\nfk^+) \cong \FF_q^\times \hookrightarrow (\eA/\nfk)^\times.
\end{equation}

For $\mfk,\nfk \in \eA_+$ with $\mfk|\nfk$, define the inclusion map $\eK_\mfk \hookrightarrow \eK_\nfk$ by
$$
O_\mfk =\frac{\FF_q[t,z]}{(C_\mfk^*(t,z))} \hookrightarrow O_\nfk = \frac{\FF_q[t,z]}{(C_\nfk^*(t,z))}, \quad z \bmod C_\mfk^*(t,z)  \longmapsto C_{\nfk/\mfk}(t,z) \bmod C_\nfk^*(t,z).
$$
Let $\eK^{\text{geo}}$ be the compositum of all the Carlitz cyclotomic function fields.
The Galois group $\eG^{\text{geo}} \assign \gal(\eK^{\geo}/\ek)$ is isomorphic (via \eqref{eqn: geo-Artin}) to
\[
\underset{\nfk \in \eA_{\scaleto{+}{4pt}}}{\varprojlim} (\eA/\nfk)^\times
\overset{\sim}{\longleftrightarrow}
\eG^{\text{geo}}.
\]
We may suppose that our chosen separable closure $\ek^{\text{sep}}$ contains $\eK^{\geo}$.\\

Let $\overline{\FF}_q$ be the algebraic closure of $\FF_q$ in $\ek^{\sep}$, and put $\eK^{\ari} \assign \overline{\FF}_q \cdot \ek$, the compositum of $\overline{\FF}_q$ and $\ek$ in $\ek^{\sep}$. 
Identify 
\[\eG^{\ari} \assign
\gal(\eK^{\ari}/\ek) \cong \gal(\overline{\FF}_q/\FF_q) = \varprojlim_{\ell\in \NN} \gal(\FF_{q^\ell}/\FF_q) 
\quad \text{ with }\quad 
\varprojlim_{\ell \in \NN} (\ZZ/\ell \ZZ)
\]
by Frobenius twistings: for $\ell \in \NN$ and $c \in \ZZ/\ell \ZZ$, the corresponding $\varrho_c \in \eG_\ell \assign \gal(\FF_{q^\ell}/\FF_q)$ is the $q^c$-power Frobenius map, i.e.\ 
\begin{equation}\label{eqn: ari-Artin}
\varrho_c(\varepsilon) = \varepsilon^{q^c}, \quad \forall \varepsilon \in \FF_{q^\ell}.
\end{equation}

Let
$\eK^{\cyc}$ be the compositum of $\eK^{\geo}$ and $\eK^{\ari}$.
The Galois group $\eG^{\text{cyc}}\assign \gal(\eK^{\text{cyc}}/\ek)$ is isomorphic to
\[
\underset{\nfk \in \eA_{\scaleto{+}{4pt}}}{\varprojlim} (\eA/\nfk)^\times \times \varprojlim_{\ell \in \NN} (\ZZ/\ell \ZZ) \overset{\sim}{\longleftrightarrow}
\eG^{\text{geo}}\times \eG^{\ari}
\overset{\sim}{\longleftrightarrow}
\eG^{\text{cyc}}.
\]
Put $\eH^{\text{cyc}} \assign \gal(\ek^{\text{sep}}/\eK^{\text{cyc}})$ and let $\eG^{\cyc}_\infty$ be the image of $\eG_\infty$ in $\eG^{\cyc}$.
We have the following diagram:

\begin{equation}\label{eqn: Cyc ram}
\begin{tabular}{ccccc}
$\displaystyle \frac{\eG_\infty \cdot \eH^{\text{cyc}}}{\eH^{\text{cyc}}} $& $\cong $& $\eG_\infty^{\text{cyc}}$
  & $\cong$ & \hspace{1cm}
  $\FF_q^\times \hspace{0.25cm} \times \ \underset{\ell \in \NN}{\varprojlim}\ (\ZZ/\ell \ZZ)$\\
  $\cap$ && $\cap$ && $\cap$ \\
  $\displaystyle\frac{\eG}{\eH^{\text{cyc}}}$
  &$\cong$& $\eG^{\text{cyc}}$
  &$\cong$&
  $\underset{\nfk \in \eA_{\scaleto{+}{4pt}}}{\varprojlim} (\eA/\nfk)^\times \times\  \underset{\ell \in \NN}{\varprojlim}\ (\ZZ/\ell \ZZ)$.
\end{tabular}
\end{equation}

Given $\nfk \in \eA_+$ and $\ell \in \NN$, let $\eK_{\nfk,\ell}\assign \FF_{q^\ell} \cdot \eK_{\nfk}$ be the compositum of $\FF_{q^\ell}$ and $\eK_\nfk$ in $\eK^{\text{cyc}}$, called the {\it $(\nfk,\ell)$-th cyclotomic function field}.
Via the identification
\begin{equation}\label{eqn: cyc-Artin}
\begin{tabular}{ccccl}
$(\eA/\nfk)^\times \times (\ZZ/\ell \ZZ)$ & $\cong$ & $\eG_{\nfk} \times \eG_{\ell}$ & $\cong$ & 
$\eG_{\nfk,\ell}\assign \gal(\eK_{\nfk,\ell}/\ek)$ \\
$(\alpha \ \ ,\ \ c)$ & $\longmapsto$ & $(\varrho_\alpha,\varrho_c)$ & $\longmapsto$ & $\varrho_{\alpha,c}$,
\end{tabular}
\end{equation}
the maximal totally real subfield $\eK_{\nfk,\ell}^+$ of $\eK_{\nfk,\ell}$ is the fixed field of the subgroup corresponding to
$\FF_q^\times \times (\ZZ/\ell \ZZ)$.
In particular, one has that $\eG_{\nfk,1} = \eG_{\nfk}$ and $\eG_{\ell} = \eG_{1,\ell}$.
Moreover, for each irreducible $\pfk \in \eA_+$ with $\pfk \nmid \nfk$, let $\Frob_\pfk \in \eG_{\nfk,\ell}$ be the Frobenius automorphism corresponding to $\pfk$.
We have that (see \cite[Theorem~12.10]{Rosen})
\begin{equation}\label{eqn: 2var-artin}
\Frob_\pfk = \varrho_{\alpha_\pfk, c_\pfk}, \quad \text{ where  $\alpha_\pfk = \pfk \bmod \nfk \in (\eA/\nfk)^\times$ and $c_\pfk = \deg \pfk \bmod \ell$.}
\end{equation}

For each quotient group $\overline{\eG}$ of $\eG$,
we let $\Sscr(\overline{\eG})$ be the space consisting of the functions in $\Sscr(\eG)$ which factor through $\overline{\eG}$.
We may identify $\Sscr(\overline{\eG})$ as a subspace of the $\QQ$-valued functions on $\overline{\eG}$.
Then $\Sscr(\eG^{\text{cyc}})$ is the union of the spaces 
$\Sscr(\eG_{\nfk,\ell})$ for all $(\nfk,\ell) \in \eA_+\times \NN$, and by Remark~\ref{rem: S(G)}~(5) we have that
\begin{equation}\label{eqn: S(G^c)}
\Sscr(\eG_{\nfk,\ell}) =
\left\{\varphi:\eG_{\nfk,\ell}\rightarrow \QQ \ \Bigg|\ \sum_{c=0}^{\ell-1}\sum_{\epsilon \in \FF_q^\times} \varphi(\varrho\varrho_{\epsilon,c}) \text{ is independent of $\varrho \in \eG_{\nfk,\ell}$}\right\}.
\end{equation}

Let $\Pscr_\nu^{\rm cyc}$ be the restriction of $\Pscr_\nu$ to $\Sscr(\eG^{\text{cyc}})$.
We shall connect the image of $\Pscr_\nu^{\rm cyc}$ with the special gamma values
in Section~\ref{sec: 2var-Gamma}.
To do so, we first discuss three types of ``diamond brackets'' and the corresponding ``Stickelberger distributions'' in the next section.

\section{Diamond brackets and Stickelberger distributions}

\subsection{Diamond brackets}\label{sec: diamond bracket}

\subsubsection*{I. The arithmetic case}

Given $y \in \RR$, let $\langle y \rangle_{\text{ari}}$ be the
fractional part of $y$, i.e.\ the unique real number with $0\leq \langle y \rangle_{\text{ari}}<1$ so that
\[
y-\langle y \rangle_{\text{ari}} \in \ZZ.
\]
We may view $\langle \cdot \rangle_{\text{ari}}$
as a function on the quotient group $\RR/\ZZ$, called the {\it arithmetic diamond bracket}.
Recall the classical reflection and multiplication formulas in the following: for $y \in \RR$ and $N \in \NN$,
\[
\langle y\rangle_{\text{ari}} + \langle -y\rangle_{\text{ari}} = 
\begin{cases} 1 & \text{ if $y \not\in \ZZ$;}\\
0 & \text{ otherwise,}
\end{cases}
\quad\text{and}\quad
\sum_{i=0}^{N-1}\langle y+\frac{i}{N}\rangle_{\text{ari}}
= \langle Ny\rangle_{\text{ari}} + \frac{N-1}{2}.
\]

Let $\ZZ_{(p)}$ be the localization of $\ZZ$ at the prime number $p$, i.e.\
\[
\ZZ_{(p)}\assign \{a/b \in \QQ \mid a,b \in \ZZ \text{ with } b \not\equiv 0 \bmod p\}.
\]
For each $y \in \ZZ_{(p)}$, there exist $\ell \in \NN$ and integers $y_0,...,y_{\ell-1}$ with $0\leq y_0,...,y_{\ell-1}< q$ so that
$$
\langle y\rangle_{\ari} = \sum_{i=0}^{\ell-1} y_i \cdot \frac{ q^i}{q^\ell-1}.
$$
Set 
\begin{equation}\label{eqn: wt0-ari}
\text{wt}_0^{\ari}(y)\assign \frac{1}{\ell}\sum_{i=0}^{\ell-1} y_i,
\end{equation}
which depends only on $y \in \ZZ_{(p)}/\ZZ$.
Restricting $\langle \cdot \rangle_{\text{ari}}$ to $\ZZ_{(p)}$, we have the following relations:

\begin{lemma}\label{lem: ari-relation}
Given $y \in \ZZ_{(p)}$, take $\ell \in \NN$ so that $(q^\ell -1) y \in \ZZ$.
\begin{itemize}
\item[(1)] There exist unique integers $y_0,...,y_{\ell-1}$ with $0\leq y_0,...,y_{\ell-1}<q$ satisfying 
\[
\langle y \rangle_{\text{\rm ari}} = \sum_{i=0}^{\ell-1} y_i \langle \frac{q^i}{q^\ell-1}\rangle_{\text{\rm ari}}.
\]
\item[(2)] The following equality holds:
\[
\sum_{i=0}^{\ell-1}\langle q^i y\rangle_{\text{\rm ari}} = \frac{\ell \cdot \text{\rm wt}_0^{\ari}(y)}{q-1}.
\]
\end{itemize}
\end{lemma}

\subsubsection*{II. The geometric case}

For $N \in \ZZ_{\geq 0}$, define $\langle \cdot \rangle_N : k_\infty \rightarrow \{0,1\}$ by (see \cite[5.5.1~{\it Definitions}]{ABP}, \cite{Sinb}, and \cite[Definition~7.6.1]{Thakur91}):
$$
\langle \sum_{i} \epsilon_i \theta^{-i} \rangle_N \assign
\begin{cases}
1, & \text{ if $\epsilon_i = 0$ for $0<i<N+1$ and $\epsilon_{N+1} = 1$;}\\
0, & \text{ otherwise.}
\end{cases}
$$
Since $\langle x + a \rangle_N = \langle x\rangle_N$ for every $x \in k_\infty$ and $a \in A$, 
we may view $\langle \cdot\rangle_N$ as a function on the quotient group $k_\infty/A$.
The {\it{geometric diamond bracket}} is:
\[
\langle \cdot \rangle_{\text{geo}} \assign \sum_{N = 0}^\infty \langle \cdot \rangle_N : k_\infty/A \rightarrow \{0,1\}.
\]
Observe that for $a \in A$ and $\nfk \in A_+$ with $\deg a < \deg \nfk$, one has that
$$
\langle \frac{a}{\nfk}\rangle_N = 
\begin{cases}
1, & \text{ if $a \in A_+$ with $\deg \nfk - \deg a = N + 1$;} \\
0, & \text{ otherwise,}
\end{cases}
$$
whence
$$
\langle \frac{a}{\nfk}\rangle_{\text{geo}} = 
\begin{cases}
1, & \text{ if $a \in A_+$;} \\
0, & \text{ otherwise.}
\end{cases}
$$

For each $x \in k/A$, set
\begin{equation}\label{eqn: wt0-geo}
\text{wt}_0^{\geo}(x) \assign 
\begin{cases}
1, & \text{ if $0 \neq x \in k/A$;} \\
0, & \text{ otherwise.}
\end{cases}
\end{equation}

\begin{lemma}\label{lem: AL-geoG}
{\rm (Cf.\ \cite[5.5.5]{ABP})} Given $x \in k/A$ and $\nfk \in A_+$, the following relations hold:
\begin{itemize}
    \item[(1)] (Reflection formula)
    \[
    \sum_{\epsilon \in \FF_q^\times} \langle \epsilon x \rangle_{\text{\rm geo}} = \text{\rm wt}_0^{\geo}(x).
    \]
    \item[(2)] (Multiplication formula)
    \[
    \sum_{\subfrac{a \in A}{\deg a < \deg \nfk}}\langle x+\frac{a}{\nfk}\rangle_{\text{\rm geo}} = \langle \nfk \cdot x \rangle_{\text{\rm geo}} + \frac{|\nfk|_\infty-1}{q-1}.
    \]
\end{itemize}
\end{lemma}
${}$

\subsubsection*{III. The two-variable case}
We start with the following:

\begin{proposition}
There exists a unique map $\langle\cdot,\cdot\rangle : k \times \ZZ_{(p)} \rightarrow \QQ$ satisfying that
\begin{itemize}
    \item[(1)] for $x \in k$, $y \in \ZZ_{(p)}$, $a \in A$, and $N \in \ZZ$ we have 
    $\langle x+a, y+N\rangle = \langle x,y\rangle$;
    \item[(2)] given $x \in k$ and $y \in \ZZ_{(p)}$, 
    we have
    \[
    \langle x, y\rangle = \sum_{i=0}^{\ell-1} y_i \langle x, \frac{q^i}{q^\ell-1}\rangle
    \]
    when writing $\langle y\rangle_{\ari}= \sum_{i=0}^{\ell-1} y_i q^i/(q^\ell-1)$ where $\ell \in \NN$ and $y_0,...,y_{\ell-1} \in \ZZ$ with $0\leq y_0,...,y_{\ell-1}< q$;
    \item[(3)] given $x \in k$, $\ell \in \NN$, and $i \in \ZZ$ with $0\leq i <\ell$, we have
    \[
    \langle x,\frac{q^i}{q^\ell-1}\rangle = \sum_{\subfrac{N \in \, \scaleto{\ZZ}{4.6pt}_{\scaleto{\geq 0}{4pt}}}{N \equiv -1-i\bmod \ell}} \langle x\rangle_N.
    \]
\end{itemize}
\end{proposition}
We call $\langle \cdot,\cdot \rangle$ the {\it two-variable diamond bracket}.

\begin{proof}
We may use the properties (1)--(3) to define a unique map $\langle \cdot,\cdot \rangle$ on $k \times \ZZ_{(p)}$, as long as we show that the equality in (2) is independent of the expression of $y$ for a given $y \in \ZZ_{(p)}$ with $0\leq y <1$.
To verify this, take $\ell, \ell' \in \NN$ and $y_0,...,y_{\ell-1},y_0',...,y_{\ell'-1}' \in \ZZ$ with
$0\leq y_i,y'_{i'} < q$ for $0\leq i<\ell$ and $0\leq i'<\ell'$ so that
\[
\sum_{i=0}^{\ell-1} y_i \cdot \frac{q^i}{q^\ell-1} = y = \sum_{i'=0}^{\ell'-1}y_{i'}'\cdot \frac{q^{i'}}{q^{\ell'}-1}.
\]
For $0\leq i'' < \ell\ell'$, put
\[
y_{i''}''\assign y_i  \quad \text{ for a unique $i$ with $0\leq i <\ell$ and $i''\equiv i \bmod \ell$}.
\]
Then
\[
\sum_{i'=0}^{\ell'-1}y_{i'}'\cdot \frac{q^{i'}}{q^{\ell'}-1} = y = \sum_{i=0}^{\ell-1} y_i \cdot \frac{q^i}{q^\ell-1} = \sum_{i=0}^{\ell-1} y_i \cdot \sum_{j=0}^{\ell'-1} \frac{q^{i+\ell j}}{q^{\ell\ell'}-1} = \sum_{i''=0}^{\ell\ell'-1} y_{i''}'' \cdot \frac{q^{i''}}{q^{\ell\ell'}-1},
\]
whence
\[
y_{i''}'' = y_{i'}'  \quad \text{ for a unique $i'$ with $0\leq i' <\ell'
$ and $i''\equiv i' \bmod \ell'$}.
\]
On the other hand, for $x \in k$,
\begin{eqnarray}
\sum_{i=0}^{\ell-1}y_i \sum_{\subfrac{N \in \, \scaleto{\ZZ}{4.6pt}_{\scaleto{\geq 0}{4pt}}}{N\equiv -1-i \bmod \ell}}\langle x \rangle_N
&=& \sum_{i=0}^{\ell-1}y_i \sum_{j=0}^{\ell'-1}\sum_{\subfrac{N \in \, \scaleto{\ZZ}{4.6pt}_{\scaleto{\geq 0}{4pt}}}{N\equiv -1-i-\ell j \bmod \ell\ell'}}\langle x \rangle_N
= \sum_{i''=0}^{\ell\ell'-1} y_{i''}''\sum_{\subfrac{N \in \, \scaleto{\ZZ}{4.6pt}_{\scaleto{\geq 0}{4pt}}}{N\equiv -1-i'' \bmod \ell\ell'}}\langle x \rangle_N \nonumber \\
&=& \sum_{i'=0}^{\ell'-1}y'_{i'} \sum_{j'=0}^{\ell-1}\sum_{\subfrac{N \in \, \scaleto{\ZZ}{4.6pt}_{\scaleto{\geq 0}{4pt}}}{N\equiv -1-i'-\ell' j' \bmod \ell\ell'}}\langle x \rangle_N
= \sum_{i'=0}^{\ell'-1}y'_{i'} \sum_{\subfrac{N \in \, \scaleto{\ZZ}{4.6pt}_{\scaleto{\geq 0}{4pt}}}{N\equiv -1-i' \bmod \ell'}}\langle x \rangle_N. \nonumber 
\end{eqnarray}
Therefore the assertion holds.
\end{proof}

\begin{remark}\label{rem: 2var-diamond-1}
${}$
\begin{itemize}
    \item[(1)]A two-variable diamond bracket was first introduced by Thakur in \cite[Definition~8.3.1]{Thakur91}.
    The version provided here is normalized slightly differently from Thakur's in order to simplify the statements of our results.
    \item[(2)] For $x \in k$, $a \in A$, $y \in \ZZ_{(p)}$, and $N \in \ZZ$, we have $\langle a,y\rangle =0 = \langle x,N\rangle$.
    \item[(3)] For every $x \in k$, we have 
    \[
    \langle x, \frac{1}{q-1} \rangle = \sum_{N=0}^\infty \langle x \rangle_N = \langle x \rangle_{\text{geo}}.
    \]
    \item[(4)] We may view $\langle \cdot, \cdot\rangle$ as a function on $(k/A) \times (\ZZ_{(p)}/\ZZ)$.
\end{itemize}
\end{remark}

The following {\it diamond bracket relations} are straightforward:

\begin{proposition}\label{prop: db-relation}
Take $x \in k$ and $y \in \ZZ_{(p)}$.
\begin{itemize}
    \item[(1)]
    Suppose $|x|_\infty<1$. 
    For $\ell \in \NN$ and $i \in \ZZ$ with $0\leq i < \ell$,
    \[
    \sum_{\epsilon \in \FF_q^\times} \langle \epsilon x, \frac{q^i}{q^\ell-1} \rangle = \begin{cases}
    1, & \text{ if $x \neq 0$ and $\ord_\infty(x) \equiv -i \bmod \ell$,}\\
    0, & \text{ otherwise.}
    \end{cases}
    \]
    \item[(2)]
    \[
    \langle x,y\rangle + \langle x,-y \rangle =
    \begin{cases}
    (q-1)\langle x \rangle_{\text{\rm geo}}, & \text{ if $y \notin \ZZ$;}\\
    0, & \text{otherwise.}
    \end{cases}
    \]
    \item[(3)] Given $\nfk \in A_+$,
    \[
    \sum_{\subfrac{a \in A}{\deg a < \deg \nfk}}
    \left(\langle \frac{x+a}{\nfk}, y \rangle - \langle \frac{a}{\nfk},y\rangle \right)
    = 
    \langle x, |\nfk|_\infty y\rangle.
    \]
\end{itemize}
\end{proposition}

\begin{remark}\label{rem: db-relation}
${}$
\begin{itemize}
    \item[(1)] (Reflection formula.) Given $x \in k/A$ and $y \in \ZZ_{(p)}/\ZZ$,
    set
    \begin{equation}\label{eqn wt_0}
    \text{wt}_0(x,y) \assign \text{wt}_0^{\geo}(x)\cdot \text{wt}_0^{\ari}(y).
    \end{equation}
    By Proposition~\ref{prop: db-relation} (1), we have that for $\ell \in \NN$  so that $(q^\ell-1)y \in \ZZ$,
    \[
    \sum_{i=0}^{\ell-1} \sum_{\epsilon \in \FF_q^\times} \langle \epsilon x, q^i y  \rangle = 
    \ell \cdot \text{wt}_0(x,y).
    \]
    \item[(2)] (Multiplication formula.)
    Given $\nfk \in A_+$, Proposition~\ref{prop: db-relation} (3) is equivalent to
    \[
    \sum_{\subfrac{a \in A}{\deg a < \deg \nfk}}
    \langle x+\frac{a}{\nfk},y\rangle 
    = \langle \nfk x, |\nfk|_\infty y\rangle - \langle  |\nfk|_\infty y\rangle_{\text{ari}} + |\nfk|_\infty \cdot \langle y\rangle_{\text{ari}}, \quad \forall x \in k \text{ and } y \in \ZZ_{(p)}.
    \]
    To see this, we may assume that $y = q^i/(q^\ell-1)$ for $\ell \in \NN$ and $i \in \ZZ$ with $0\leq i < \ell$ without loss of generality.
    Then the result follows from:
    \begin{align*}
    \sum_{\subfrac{a \in A}{\deg a < \deg \nfk}} \langle \frac{a}{\nfk}, \frac{q^i}{q^\ell-1}\rangle 
    &= \#\{a \in A_+\mid \deg a < \deg \nfk,\ \deg a \equiv \deg \nfk + i \bmod \ell\} \\
    &= -\langle \frac{q^{\deg \nfk + i}}{q^\ell-1}\rangle_{\ari} + \frac{q^{\deg \nfk + i}}{q^\ell-1}.
    \end{align*}
\end{itemize}
\end{remark}

\subsection{Stickelberger distributions}

Define actions of $\eG$ both on $k/A$ and on $\ZZ_{(p)}/\ZZ$ as follows:
Given $\varrho \in \eG^{\cyc}$, write $\varrho = \varrho_{\alpha,c}$
where $(\alpha,c) = (\alpha_\nfk,c_\ell)_{\nfk,\ell} \in \varprojlim_{\nfk,\ell}(\eA/\nfk)^\times \times (\ZZ/\ell\ZZ)$
corresponds to $\varrho$ via the Artin map (see \eqref{eqn: cyc-Artin}).
For each $x \in \frac{1}{\nfk(\theta)}A/A$ and $y \in \frac{1}{q^\ell-1}\ZZ/\ZZ$, set
\begin{equation}\label{eqn: G-action}
\varrho \star x \assign \alpha_\nfk(\theta) \cdot x \ \bmod A \quad \text{ and } \quad 
\varrho \star y \assign q^{c_\ell} \cdot y \ \bmod \ZZ.
\end{equation}
We also let $\eG$ act diagonally on $(k/A) \times (\ZZ_{(p)}/\ZZ)$, i.e.\
\[
\varrho \star (x,y) \assign (\varrho \star x, \varrho \star y), \quad \forall (x,y) \in (k/A) \times (\ZZ_{(p)}/\ZZ).
\]
Recall that $\eG^{\ari} = \gal(\eK^{\ari}/\ek)$, $\eG^{\geo} = \gal(\eK^{\geo}/\ek)$, and $\eG^{\cyc} = \gal(\eK^{\cyc}/\ek)$.
Given $x \in k/A$ and $y \in \ZZ_{(p)}/\ZZ$, define
$\St^{\ari}_0(y): \eG^{\ari}\rightarrow \QQ$, 
$\St^{\geo}_0(x): \eG^{\geo}\rightarrow \QQ$, and 
$\St_0(x,y): \eG^{\cyc} \rightarrow \QQ$ respectively in the following:
\begin{align*}
\St^{\ari}_0(y)(\varrho) &\assign \langle \varrho \star y \rangle_{\ari}, \quad \quad \forall \varrho \in \eG^{\ari}, \\
\St^{\geo}_0(x)(\varrho) &\assign \langle \varrho \star x\rangle_{\geo}, \quad \quad \forall \varrho \in \eG^{\geo}
\\  
\St_0(x,y)(\varrho) & \assign \langle \varrho \star x, \varrho \star y\rangle, \quad \quad \forall \varrho \in \eG^{\cyc}.
\end{align*}
In particular, one has
\[
\St_0(x,\frac{1}{q-1}) = \St_0^{\geo}(x),
\quad \forall x \in k/A.
\]

Let $\nfk \in \eA_+$ and $\ell \in \NN$.
Note that from \eqref{eqn: S(G^c)}, $\Sscr(\eG_{1,\ell})$ is actually the whole space of $\QQ$-valued functions on $\eG_{1,\ell}$.
Moreover, for $x \in \frac{1}{\nfk(\theta)}A/A$ and $y \in \frac{1}{q^\ell-1}\ZZ/\ZZ$,
the reflection formulas in Lemma~\ref{lem: AL-geoG} (1) and Remark~\ref{rem: db-relation} (1) imply
\begin{align*}
\sum_{\epsilon \in \FF_q^\times} \St^{\geo}_0(x)(\varrho \varrho_\epsilon) & = \text{wt}_0^{\geo}(x), \quad \forall \varrho \in \eG_{\nfk,1}, \\
\sum_{c=0}^{\ell-1}
\sum_{\epsilon \in \FF_q^\times} \St_0(x,y)(\varrho \varrho_{\epsilon,c}) & = \ell \cdot \text{wt}_0(x,y), \quad \forall \varrho \in \eG_{\nfk,\ell}.
\end{align*}
From $\eqref{eqn: S(G^c)}$ we have that 
\[
\St^{\ari}_0(y) \in \Sscr(\eG_{1,\ell}), \quad 
\St^{\geo}_0(x) \in \Sscr(\eG_{\nfk,1}), \quad 
\text{ and } \quad
\St_0(x,y) \in \Sscr(\eG_{\nfk,\ell}).
\]

\begin{definition}\label{defn: ST}
For $x \in k/A$ and $y \in \ZZ_{(p)}/\ZZ$, we set
\begin{align*}
\St^{\ari}(y) &\assign \St_0^{\ari}(-y), & \text{wt}^{\ari}(y)& \assign \text{wt}^{\ari}_0(-y), \\
\St^{\geo}(x) & \assign \St^{\geo}_0(x) - \frac{1}{q-1}\mathbf{1}_{\eG^{\geo}}, & \text{wt}^{\geo}(x) & \assign \text{wt}_0^{\geo}(x) - 1,\\
\St(x,y) & \assign \St_0(x,-y)-\St_0^{\ari}(-y), & \text{wt}(x,y) &\assign \text{wt}^{\geo}(x)\cdot \text{wt}^{\ari}(y).
\end{align*}
We call $\St^{\ari}(y)$, $\St^{\geo}(x)$, and $\St(x,y)$ the \emph{Stickelberger functions associated to $y$, $x$, and $(x,y)$}, respectively.
The maps 
$\St^{\ari}: \ZZ_{(p)}/\ZZ \rightarrow \Sscr(\eG^{\ari})$,
$\St^{\geo}: k/A\rightarrow \Sscr(\eG^{\geo})$, 
and
$\St: (k/A)\times (\ZZ_{(p)}/\ZZ) \rightarrow \Sscr(\eG^{\cyc})$
are called the \emph{Stickelberger distributions associated to $\langle\cdot \rangle_{\ari}$, $\langle\cdot \rangle_{\geo}$,
and $\langle \cdot,\cdot \rangle$}, respectively.
\end{definition}

\begin{remark}\label{rem: St-relations}
Let $\nfk \in \eA_+$ and $\ell \in \NN$.
Given $x\in \frac{1}{\nfk(\theta)}/A$ and $y \in \frac{1}{q^\ell-1}\ZZ/\ZZ$,
suppose $\langle -y \rangle_{\text{ari}} = \sum_{i=0}^{\ell-1} y_i q^i/(q^\ell-1)$ where 
$y_0,...,y_{\ell-1} \in \ZZ$ with $0\leq y_0,...,y_{\ell-1}<q$.
The diamond bracket relations imply the following:
\begin{itemize}
    \item[(1)] 
    \[
    \St^{\ari}(y) = \sum_{i=0}^{\ell-1} y_i 
    \St^{\ari}(\frac{q^i}{1-q^\ell})
    \quad \text{ and } \quad 
    \sum_{c=0}^{\ell-1} \St^{\ari}(y)(\varrho\varrho_c) = \frac{\ell \cdot  \text{wt}^{\ari}(y)}{q-1}, \quad \forall \varrho \in \eG_{1,\ell}.
    \]
    \item[(2)]
    \[
    \sum_{\subfrac{a \in A}{\deg a<\deg \nfk'}}\St^{\geo}(x+\frac{a}{\nfk'}) = \St^{\geo}(\nfk' x),\quad \forall \nfk' \in A_+,
    \]
    and
    \[
    \sum_{\epsilon \in \FF_q^\times} \St^{\geo}(x)(\varrho \varrho_\epsilon) = \text{wt}^{\geo}(x), \quad \forall \varrho \in \eG_{\nfk,1}.
    \]
    \item[(3)]
    \[
    \St(x,y) = \sum_{i=0}^{\ell-1} y_i \St(x,\frac{q^i}{1-q^\ell}), \quad 
    \sum_{c=0}^{\ell-1} \sum_{\epsilon \in \FF_q^\times}\St(x,y)(\varrho \varrho_{\epsilon,c}) = \ell \cdot \text{wt}(x,y), \quad \forall \varrho \in \eG_{\nfk,\ell},
    \]
    and
    \[
    \sum_{\subfrac{a \in A}{\deg a < \deg \nfk'}}
    \St(x+\frac{a}{\nfk'},y) = \St(\nfk' x,|\nfk'|_\infty y), \quad \forall \nfk' \in A_+.
    \]
\end{itemize}
In particular, from the definitions of $\St^{\ari}$, $\St^{\geo}$, and $\St$, we get that
\begin{itemize}
    \item[(4)] For every $x \in k/A$ and $y \in \ZZ_{(p)}/\ZZ$,
    \[
    \St(0,y) = -\St^{\ari}(y),
    \quad \St(x,\frac{1}{1-q}) = \St^{\geo}(x), 
    \]
    and
    \[
    \St(x,y)+\St(x,-y) = 
    \begin{cases}
    (q-1)\St^{\geo}(x), & \text{ if $y \neq 0 \in \ZZ_{(p)}/\ZZ$;}\\
    0, & \text{ otherwise.}
    \end{cases}
    \]
\end{itemize}
\end{remark}

\begin{proposition}\label{prop: St-span}
Let $\nfk \in \eA_+$ and $\ell \in \NN$.
\begin{itemize}
\item[(1)] The space
$\Sscr(\eG_{1,\ell})$ is spanned by $\St^{\ari}(y)$ for all $y \in \frac{1}{q^\ell-1} \ZZ/\ZZ$.
Consequently,
$\Sscr(\eG^{\ari})$ is spanned by 
$\St^{\ari}(y)$ for all $y \in  \ZZ_{(p)}/\ZZ$.
\item[(2)] The space
$\Sscr(\eG_{\nfk,1})$ is spanned by $\St^{\geo}(x)$ for all $x \in \frac{1}{\nfk(\theta)}A/A$.
Consequently, $\Sscr(\eG^{\geo})$ is spanned by $\St^{\geo}(x)$ for all $x \in k/A$.
\item[(3)] The space
$\Sscr(\eG_{\nfk,\ell})$ is spanned by 
$\St(x,y)$
for all
$x \in \frac{1}{\nfk(\theta)}A/A$ and $y \in \frac{1}{q^\ell-1}\ZZ/\ZZ$.
Consequently,
the space $\Sscr(\eG^{\rm cyc})$ is spanned by $\St(x,y)$ for all $x \in k/A$ and $y \in \ZZ_{(p)}/\ZZ$.
\end{itemize}
\end{proposition}

This is an analogue of Deligne's Theorem stated in \cite[Theorem~3.1]{Anderson82} for the classical case.
The key step is to regard the Stickelberger functions $\St(x,y)$ for $x \in \frac{1}{\nfk(\theta)}A/A$ and $y \in \frac{1}{q^\ell-1}\ZZ/\ZZ$, where $\nfk \in \eA_+$ and $\ell \in \NN$ are given, as the ``evaluators'' of the Artin $L$-functions associated to the characters on the Galois group $\eG_{\nfk,\ell}$.
Let $\widehat{\eG}_{\nfk,\ell}$ be the Pontryagin dual of $\eG_{\nfk,\ell}$.
Identifying $\eG_{\nfk,\ell}$ with $(\eA/\nfk)^\times \times \ZZ/\ell \ZZ$ via the Artin map in \eqref{eqn: cyc-Artin}, 
for each character $\bochi \in \widehat{\eG}_{\nfk,\ell}$
we let $\bochi_f\assign \bochi\big|_{(\eA/\nfk)^\times}: (\eA/\nfk)^\times \rightarrow \CC^\times$, and put $\cfk_{\bochi}\assign \cfk_{\bochi_f} \in \eA_+$ where $\cfk_{\bochi_f}$ is the conductor of $\bochi_f$.
The (finite part of the) \emph{Artin $L$-function associated to $\bochi$} can be written as the following:
\[
L_{\eA}(s,\bochi) \assign \prod_{\substack{\text{irr.~}\pfk \in \eA_{\scaleto{+}{4pt}} \\ \pfk \nmid \cfk_{\bochi}}}
\left(1-\frac{\bochi(\Frob_\pfk)}{|\pfk|_\infty^s}\right)^{-1}, \quad \re(s)>1.
\]
Let $L(s,\bochi_f)$ be the \emph{Dirichlet $L$-function associated to the character $\bochi_f:(\eA/\nfk)^\times \rightarrow \CC^\times$}
(see \cite[equality~(1) in p.~11 and Proposition 4.3]{Rosen}):
\begin{eqnarray}
L(s,\bochi_f) &\assign& \prod_{\substack{\text{irr.~}\pfk \in \eA_{\scaleto{+}{4pt}} \\ \pfk \nmid \cfk_{\bochi_f}}}
\left(1-\frac{\bochi_f(\pfk)}{|\pfk|_\infty^s}\right)^{-1}, \quad \re(s)>1. \nonumber \\
&=&
\begin{cases}
\displaystyle \frac{1}{1-q^{1-s}}, & \text{ if $\bochi_f$ is trivial;} \\
& \nonumber \\
\displaystyle\sum_{\subfrac{\afk \in \eA_{\scaleto{+}{4pt}},\ \deg \afk < \deg \cfk_{\bochi_f}}{(\afk, \cfk_{\bochi_f}) = 1}}\frac{\bochi_f(\afk)}{|\afk|_\infty^s}, & \text{ otherwise.}
\end{cases}
\end{eqnarray}
Write
\[
\bochi(1,1) = q^{-s_{\bochi}}
\quad \text{ where } \quad
s_{\bochi} = -\frac{2\pi \sqrt{-1}}{\ln q} \cdot \frac{d_{\bochi}}{\ell}
\quad \text{ for a unique integer $d_{\bochi}$ with $0\leq d_{\bochi}<\ell$.}
\]
By \eqref{eqn: 2var-artin} we have that 
\[
L_{\eA}(s,\bochi) = L(s+s_{\bochi},\bochi_f)
= \begin{cases}
\displaystyle \frac{1}{1-q^{1-{(s+s_{\bochi})}}}, & \text{ if $\bochi_f$ is trivial;} \\
& \nonumber \\
\displaystyle\sum_{\subfrac{\afk \in \eA_{\scaleto{+}{4pt}},\ \deg \afk < \deg \cfk_{\bochi}}{(\afk, \cfk_{\bochi}) = 1}}\frac{\bochi(\afk,\deg \afk)}{|\afk|_\infty^s}, & \text{ otherwise.}
\end{cases}
\]
For $\mfk \in \eA_+$ with $\cfk_{\bochi}\mid \mfk \mid \nfk$, put
\begin{eqnarray}\label{eqn: L-relation}
L_{\eA}^\mfk(s,\bochi)
&\assign & L_{\eA}(s,\bochi)\prod_{\subfrac{\text{irr.~} \pfk \in \eA_{\scaleto{+}{4pt}}}{\pfk\mid \mfk,\ \pfk \nmid \cfk_{\bochi}}}(1-\frac{\bochi(\pfk,\deg \pfk)}{|\pfk|_\infty^s}) \nonumber \\
&=&
\begin{cases}
\displaystyle \frac{1}{1-q^{1-(s+s_{\bochi})}} \cdot 
\prod_{\subfrac{\text{irr.~} \pfk \in \eA_{\scaleto{+}{4pt}}}{\pfk\mid \mfk}}(1-\frac{1}{|\pfk|_{\infty}^{s+s_{\bochi}}})
, & \text{ if $\bochi_f$ is trivial;} \\
& \\
\displaystyle\sum_{\subfrac{\afk \in \eA_{\scaleto{+}{4pt}},\ \deg \afk < \deg \mfk}{(\afk, \mfk) = 1}}\frac{\bochi(\afk,\deg \afk)}{|\afk|_\infty^s}, & \text{ otherwise.}
\end{cases}
\end{eqnarray}

We derive the following:

\begin{lemma}\label{lem: evaluator}
Let $\nfk \in \eA_+$ and $\ell \in \NN$.
Take $\bochi \in \widehat{\eG}_{\nfk,\ell}$ and $\mfk \in \eA_+$ with $\mfk \mid \nfk$.
For each $a \in \eA$ with $\text{\rm gcd}(a,\mfk) = 1$ and $c \in \ZZ$ with $0\leq c < \ell$, we have that when $\bochi_f$ is non-trivial,
\begin{equation}\label{eqn: evaluator}
\sum_{\varrho \in \eG_{\nfk,\ell}}\St(\frac{a(\theta)}{\mfk(\theta)},\frac{q^c}{1-q^{\ell}})(\varrho)\cdot \overline{\bochi(\varrho)}  = \frac{\#(\eA/\nfk)^\times}{\#(\eA/\mfk)^\times}\cdot \bochi(a,c+\deg \mfk) \cdot
\begin{cases}
L_{\eA}^\mfk(0,\overline{\bochi}), & \text{ if $\cfk_{\bochi}\mid \mfk$;} \\
0, &\text{ otherwise;}
\end{cases}
\end{equation}
and when $\bochi_f$ is trivial,
\begin{equation}\label{eqn: evaluator2}
\sum_{\varrho \in \eG_{\nfk,\ell}}\St(0,\frac{q^c}{1-q^{\ell}})(\varrho)\cdot \overline{\bochi(\varrho)} = \#(\eA/\nfk)^\times\cdot \bochi(1,c)\cdot L_{\eA}(0,\overline{\bochi}).
\end{equation}
\end{lemma}

\begin{proof}
When 
$\cfk_{\bochi} \nmid \mfk$, there exists $\alpha \in (\eA/\nfk)^\times$ with $\alpha \equiv 1 \bmod \mfk$ such that $\bochi_f(\alpha) \neq 1$.
Thus 
\begin{align*}
\sum_{\varrho \in \eG_{\nfk,\ell}}\St(\frac{a(\theta)}{\mfk(\theta)},\frac{q^c}{1-q^{\ell}})(\varrho)\cdot \overline{\bochi(\varrho)} &= \sum_{\varrho \in \eG_{\nfk,\ell}}\St(\frac{a(\theta)}{\mfk(\theta)},\frac{q^c}{1-q^{\ell}})(\varrho\varrho_{\alpha,0})\cdot \overline{\bochi(\varrho\varrho_{\alpha,0})}\\
&=\overline{\bochi_f(\alpha)}\cdot \sum_{\varrho \in \eG_{\nfk,\ell}}\St(\frac{a(\theta)}{\mfk(\theta)},\frac{q^c}{1-q^{\ell}})(\varrho)\cdot \overline{\bochi(\varrho)},
\end{align*}
which implies that 
$\displaystyle\sum_{\varrho \in \eG_{\nfk,\ell}}\St(\frac{a(\theta)}{\mfk(\theta)},\frac{q^c}{1-q^{\ell}})(\varrho)\cdot \overline{\bochi(\varrho)} = 0$.

When $\bochi_f$ is non-trivial with $\cfk_{\bochi}\mid \mfk$,
we have that
\begin{align*}
\sum_{\varrho \in \eG_{\nfk,\ell}}\St(\frac{a(\theta)}{\mfk(\theta)},\frac{q^c}{1-q^{\ell}})(\varrho)\cdot \overline{\bochi(\varrho)} &=\frac{\#(\eA/\nfk)^\times}{\#(\eA/\mfk)^\times}\cdot 
\sum_{i=0}^{\ell-1}\sum_{\alpha \in (\eA/\mfk)^\times}\langle\frac{\alpha(\theta) a(\theta)}{\mfk(\theta)},\frac{q^{i+c}}{q^{\ell}-1}\rangle \cdot \overline{\bochi(\alpha,i)} \\
&= \frac{\#(\eA/\nfk)^\times}{\#(\eA/\mfk)^\times}\cdot  \bochi(a,c) \cdot 
\sum_{\substack{\afk \in \eA_{\scaleto{+}{4pt}},\ \deg \afk<\deg \mfk, \\ \text{gcd}(\afk,\mfk) =1}} \overline{\bochi(\afk,\deg \afk-\deg\mfk)} \\
&= \frac{\#(\eA/\nfk)^\times}{\#(\eA/\mfk)^\times}\cdot  \bochi(a,c+\deg \mfk) \cdot L_{\eA}^{\mfk}(0,\overline{\bochi}),
\end{align*}
where the last equality follows from \eqref{eqn: L-relation}.

Finally, when $\bochi_f$ is trivial, we get
\begin{align*}
\sum_{\varrho \in \eG_{\nfk,\ell}}\St(0,\frac{q^c}{1-q^{\ell}})(\varrho)\cdot \overline{\bochi(\varrho)}
&= \#(\eA/\nfk)^\times \cdot \sum_{i=0}^{\ell-1}\left(-\langle \frac{q^{c+i}}{q^\ell-1}\rangle_{\ari} \cdot q^{-i s_{\overline{\bochi}}}\right) \\
&= \#(\eA/\nfk)^\times \cdot \frac{q^{-cs_{\bochi}}}{1-q^\ell} \cdot  \sum_{i=0}^{\ell-1}q^{i(1-s_{\overline{\bochi}})} \\
&= \#(\eA/\nfk)^\times \cdot \bochi(1,c) \cdot L_{\eA}(0,\overline{\bochi}).
\end{align*}
\end{proof}

\begin{remark}\label{rem: L-vanish}
Let $\nfk \in \eA_+$ and $\ell \in \NN$.
Let $\eK_{\nfk,\ell}^+$ be the maximal totally real subfield of $\eK_{\nfk,\ell}$.
Put $\eG_{\nfk,\ell}^+ \assign \gal(\eK_{\nfk,\ell}^+/\ek)$.
Then we may identify $\widehat{\eG}_{\nfk,\ell}^+$, the Pontryagin dual of $\eG_{\nfk,\ell}^+$, with the subgroup of $\widehat{\eG}_{\nfk,\ell}$ consisting of all characters $\bochi$ satisfying that $\bochi(\varrho_{\epsilon,c}) = 1$ for all $\epsilon \in \FF_q^\times$ and $c \in \ZZ/\ell\ZZ$.
By \cite[Proposition~14.12]{Rosen}, we get that for each $\bochi^+ \in \widehat{\eG}_{\nfk,\ell}^+$,
\[
\ord_{s=0}L_\eA(s,\bochi^+) = 1 \quad \text{ if and only if \quad $\bochi^+$ is non-trivial.}
\]
Denote by $O_{\nfk,\ell}$ and $O_{\nfk,\ell}^+$ the integral closure of $\eA$ in $\eK_{\nfk,\ell}$ and $\eK_{\nfk,\ell}^+$, respectively.
Let $\zeta_{O_{\nfk,\ell}}(s)$ and $\zeta_{O_{\nfk,\ell}^+}(s)$ be the Dedekind-Weil zeta functions of $O_{\nfk,\ell}$ and $O_{\nfk,\ell}^+$, respectively, i.e.
\[
\zeta_{O_{\nfk,\ell}}(s) \assign \sum_{\substack{\text{nonzero ideal} \\ \Afk \subset O_{\nfk,\ell}}}\frac{1}{\#(O_{\nfk,\ell}/\Afk)^{s}}
\quad \text{and} \quad 
\zeta_{O_{\nfk,\ell}^+}(s) \assign \sum_{\substack{\text{nonzero ideal} \\ \Afk^+ \subset O_{\nfk,\ell}^+}}\frac{1}{\#(O_{\nfk,\ell}^+/\Afk^+)^{s}}, \  \re(s)>1.
\]
Then it is known that (cf.\ \cite[Proposition~14.11]{Rosen})
\[
\zeta_{O_{\nfk,\ell}}(s) = \zeta_{O_{\nfk,\ell}^+}(s) \cdot \prod_{\bochi \in \widehat{\eG}_{\nfk,\ell}\setminus \widehat{\eG}^+_{\nfk,\ell}}L_\eA(s,\bochi)
\quad \text{ and } \quad
\zeta_{O_{\nfk,\ell}^+}(s) = 
\prod_{\bochi^+ \in  \widehat{\eG}^+_{\nfk,\ell}}
L_\eA(s,\bochi).
\]
As $\ord_{s=0}\zeta_{O_{\nfk,\ell}}(s) = [\eK_{\nfk,\ell}^+:\ek]-1 = \ord_{s=0}\zeta_{O_{\nfk,\ell}^+}(s)$ (see \cite[Theorem~14.4]{Rosen}), we obtain that:
\end{remark}

\begin{corollary}\label{cor: vanishing-cond}
For $\bochi \in \widehat{\eG}_{\nfk,\ell}$, $L_\eA(0,\bochi) = 0$ if and only if
\[
\text{$\bochi$ is non-trivial \quad  and } \quad 
\bochi(\varrho_{\epsilon,c}) =1, \quad \forall \epsilon \in \FF_q^\times \text{ and } c \in \ZZ/\ell \ZZ.
\]
\end{corollary}

\begin{proof}[Proof of Proposition~\ref{prop: St-span}]
By Remark~\ref{rem: St-relations}~(4), it suffices to prove (3).
Let $\Fcal_\CC(\eG_{\nfk,\ell})$ be the space of $\CC$-valued functions on $\eG_{\nfk,\ell}$,
which is a $\CC$-algebra under the convolution product:
$$
f_1*f_2(\varrho) \assign  \sum_{\subfrac{\varrho_1,\varrho_2 \in \eG_{\nfk,\ell}}{\varrho_1\varrho_2 = \varrho}} f_1(\varrho_1)f_2(\varrho_2).
$$
For each character $\bochi \in \widehat{\eG}_{\nfk,\ell}$, the element $\#(\eG_{\nfk,\ell})^{-1} \cdot \bochi$ is an idempotent in $\Fcal_\CC(\eG_{\nfk,\ell})$, and
$$
\Fcal_\CC(\eG_{\nfk,\ell}) = \bigoplus_{\bochi \in \widehat{\eG}_{\nfk,\ell}} \CC \cdot \bochi.
$$
In particular, for $\phi \in \Fcal_\CC(\eG_{\nfk,\ell})$ and $\bochi \in \widehat{\eG}_{\nfk,\ell}$, one has
$$
\phi * \bochi = (\phi\mid \bochi)_{\nfk,\ell} \cdot \bochi, \quad \text{ where }\quad  (\phi\mid \bochi)_{\nfk,\ell} \assign 
\sum_{\varrho\in \eG_{\nfk,\ell}}\phi(\varrho)\bar{\bochi}(\varrho),
$$
and so
\begin{equation}\label{eqn: chi-span}
\phi = \sum_{\bochi \in \widehat{\eG}_{\nfk,\ell}} \phi * \frac{\bochi}{\#(\eG_{\nfk,\ell})} =  \frac{1}{\#\eG_{\nfk,\ell}}\cdot \sum_{\bochi \in \widehat{\eG}_{\nfk,\ell}}(\phi\mid \bochi)_{\nfk,\ell}\cdot \bochi.
\end{equation}

Let $\Hcal_\CC(\eG_{\nfk,\ell})$ be the subspace spanned over $\CC$ by $\St(x,y)$ in $\Fcal_\CC(\eG_{\nfk,\ell})$ for all $x \in \frac{1}{\nfk(\theta)}A/A$ and $y \in \frac{1}{q^{\ell} -1} \ZZ/\ZZ$.
As $\Hcal_\CC(\eG_{\nfk,\ell})$ is invariant under the action of $\eG_{\nfk,\ell}$,
Lemma~\ref{lem: evaluator} implies that
$$
\Hcal_\CC(\eG_{\nfk,\ell}) = \bigoplus_{\subfrac{\bochi \in \widehat{\eG}_{\nfk,\ell}}{ L_{\eA}(0,\overline{\bochi}) \neq 0}} \CC \cdot \bochi = \Bigl\{\varphi \in \Fcal_\CC(\eG_{\nfk,\ell})\ \Big|\ (\varphi \mid \bochi)_{\nfk,\ell} = 0,\ \forall \bochi \in \widehat{\eG}_{\nfk,\ell} \text{ with } L_\eA (0,\bochi) = 0\Bigr\}.
$$
For $\bochi \in \widehat{\eG}_{\nfk,\ell}$, by Corollary~\ref{cor: vanishing-cond} we have that the condition $L_\eA(0,\bochi) = 0$ is equivalent to saying that $\bochi$ is non-trivial and $\bochi(\varrho_{\epsilon,c}) =1$ for every $(\epsilon,c) \in \FF_q^\times \times(\ZZ/\ell\ZZ)$.
Therefore 
\[
\Hcal_\CC(\eG_{\nfk,\ell}) =
\left\{\varphi \in \Fcal_\CC(\eG_{\nfk,\ell})\ \Bigg|\ \sum_{c=0}^{\ell -1} \sum_{\epsilon \in \FF_q^\times} \varphi( \varrho \varrho_{\epsilon,c} ) \text{ is independent of $\varrho \in \eG_{\nfk,\ell}$}\right\}.
\]

Let $\Fcal_\QQ(\eG_{\nfk,\ell})$ be the space of $\QQ$-valued functions on $\eG_{\nfk,\ell}$.
The above description of $\Hcal_\CC(\eG_{\nfk,\ell})$ and \eqref{eqn: S(G^c)} shows the desired identity:
$$
\Sscr(\eG_{\nfk,\ell}) = \Hcal_\CC(\eG_{\nfk,\ell}) \cap \Fcal_\QQ(\eG_{\nfk,\ell}) = \sum_{x \in \frac{1}{\nfk(\theta)}A/A}\ \sum_{y \in \frac{1}{q^{\ell}-1}\ZZ/\ZZ} \QQ \cdot \St(x,y).
$$
\end{proof}

\begin{remark}\label{rem: S(G)-chi}
Let $\nfk \in \eA_+$ and $\ell \in \NN$.
For each CM field $\eK$ contained in $\eK_{\nfk,\ell}$, 
let $\Fcal_\CC(\eG_\eK)$ be the space of $\CC$-valued functions on $\eG_\eK = \gal(\eK/\ek)$.
The above proof also shows that
\[
\CC \otimes_\QQ \Sscr(\eG_\eK) \cong \Fcal_\CC(\eG_\eK) \cap \Hcal_\CC(\eG_{\nfk,\ell})
= \bigoplus_{\substack{\bochi \in \widehat{\eG}_\eK \\ L_\eA(0,\overline{\bochi}) \neq 0}} \CC \cdot \bochi.
\]
In particular, let $\eK^+$ be the maximal totally real subfield of $\eK$ and $\eG_{\eK^+} \assign \gal(\eK^+/\ek)$. 
For each $\varphi \in \Sscr(\eG_\eK)$, by Corollary~\ref{cor: vanishing-cond} and \eqref{eqn: chi-span} we have the following expression:
\[
\varphi = \frac{1}{\#\eG_\eK}\cdot \left(\sum_{\varrho \in \eG_\eK} \varphi(\varrho)\right) +  \frac{1}{\#\eG_\eK}\cdot
\sum_{\bochi \in \widehat{\eG}_\eK\setminus \widehat{\eG}_{\eK^+}}\left(\sum_{\varrho \in \eG_\eK}\varphi(\varrho)\overline{\bochi(\varrho)}\right)\cdot \bochi.
\]
This will be used in Section~\ref{sec: CSF}.
\end{remark}

\subsection{Universal distributions associated to diamond bracket relations}
\label{sec: Uni-dis}

\subsubsection*{I. The arithmetic case}

Let $\Acal^{\text{ari}}$ be the free abelian group generated by all elements in $\ZZ_{(p)}/\ZZ$.
We identify $\Acal^{\text{ari}}$ with a subgroup of $\Acal^{\text{ari}}_\QQ \assign \QQ \otimes_\ZZ \Acal^{\text{ari}}$.
Every element in $\Acal^{\text{ari}}$
(resp.\ $\Acal^{\text{ari}}_\QQ$) can be written uniquely as a formal sum
\[
\boy = \sum_{y \in \ZZ_{(p)}/\ZZ} n_y [y], \quad
\text{where $n_y \in \ZZ$ (resp.\ $\QQ$) and $n_y = 0$ for almost all $y$.}
\]
Given $\ell \in \NN$, let $\Acal^{\ari}_\ell$ (resp.\ $\Acal^{\ari}_{\ell,\QQ}$)
be the subgroup of $\Acal^{\ari}$ (resp.\ subspace of $\Acal^{\ari}_\QQ$) generated by $y \in \frac{1}{q^\ell-1}\ZZ/\ZZ$.
Also, we let $\Rcal^{\text{ari}}_\ell$ be the subspace of $\Acal^{\text{ari}}_{\ell,\QQ}$ generated by
\[
[y] - \sum_{i=0}^{\ell-1} y_i [\frac{q^i}{1-q^\ell}], \quad \forall y \in \frac{1}{q^\ell-1} \ZZ/\ZZ,
\]
where $y_0,...,y_{\ell-1} \in \ZZ$ with
$0\leq y_0,...,y_{\ell-1}<q$ so that
$\langle -y \rangle_{\text{ari}} = \sum_{i=0}^{\ell-1}y_i q^i/(q^\ell-1)$.
Put $\Ucal^{\ari}_\ell\assign \Acal^{\ari}_{\ell,\QQ}/\Rcal^{\ari}_\ell$ and 
$\Ucal^{\ari}\assign \Acal^{\ari}_\QQ/\Rcal^{\ari}$, where
$\Rcal^{\text{ari}} \assign \cup_\ell \Rcal^{\text{ari}}_\ell$.
We call $\Ucal^{\ari}$ the \emph{universal distribution associated to the arithmetic diamond bracket relations}.

\begin{lemma}\label{lem: Uari-dim}
$\dim_{\QQ}\Ucal^{\ari}_\ell \leq \ell$ for each $\ell \in \NN$.
\end{lemma}

\begin{proof}
As $\Ucal^{\ari}_\ell$ is generated by the images of $[\displaystyle \frac{q^i}{1-q^\ell}]$ for $0\leq i < \ell$, the result follows.
\end{proof}

On the other hand,
extending $\St^{\ari}$ to a $\QQ$-linear homomorphism from $\Acal^{\ari}_{\ell,\QQ}$ (resp.\ $\Acal^{\ari}_\QQ$) to $\Sscr(\eG_{1,\ell})$ (resp.\ $\Sscr(\eG^{\ari})$),
by Proposition~\ref{prop: St-span}~(1) we have 

\begin{lemma}\label{lem: Stari-surj}
For each $\ell \in \NN$,
$\St^{\ari}: \Acal^{\text{\rm ari}}_{\ell,\QQ} \rightarrow \Sscr(\eG_{1,\ell})$ is surjective.
Consequently,
$\St^{\ari} : \Acal^{\text{\rm ari}}_\QQ \rightarrow \Sscr(\eG^{\text{\rm ari}})$
is surjective.
\end{lemma}

From Remark~\ref{rem: St-relations}~(1), one has that $\St^{\ari}(\Rcal^{\ari}) = 0$, which says that the surjection $\St^{\ari}: \Acal^{\ari}_{\ell,\QQ} \rightarrow \Sscr(\eG_{1,\ell})$ factors through $\Ucal^{\ari}_\ell$.
Since $\dim_\QQ \Sscr(\eG_{1,\ell}) = \ell$ by Remark~\ref{rem: S(G/H)-dim}, combining Lemma~\ref{lem: Uari-dim} and \ref{lem: Stari-surj} we obtain 

\begin{corollary}\label{cor: Uari-iso}
The map from $\Ucal^{\ari}_\ell$ to  $\Sscr(\eG_{1,\ell})$ induced by $\St^{\ari}$ is a $\QQ$-linear isomorphism.
Consequently,
\begin{equation}\label{eqn: Uari-dim}
\dim_\QQ \Ucal^{\ari}_\ell = \dim_\QQ \Sscr(\eG_{1,\ell}) = \ell,
\end{equation}
and we have a $\QQ$-linear isomorphism
$\Ucal^{\ari} \cong \Sscr(\eG^{\ari})$.
\end{corollary}

\subsubsection*{II. The geometric case}

Let $\Acal^{\geo}$ be the free abelian group generated by all elements in $k/A$.
We may identify $\Acal^{\geo}$ with a subgroup of $\Acal^{\geo}_\QQ \assign \QQ \otimes_\ZZ \Acal^{\geo}$.
Every element in $\Acal^{\geo}$ (resp.\ $\Acal^{\geo}_\QQ$)
can be written uniquely as a formal sum
\[
\boxx = \sum_{x \in k/A} n_x [x],
\quad
\text{where $n_x \in \ZZ$ (resp.\ $\QQ$) and $n_x = 0$ for almost all $x$.}
\]
Given $\nfk \in \eA_+$, let $\Acal^{\geo}_{\nfk,\QQ}$ be the subspace of $\Acal^{\geo}_\QQ$ spanned by elements in 
$\frac{1}{\nfk(\theta)}A/A$.
Also, we let $\Rcal_\nfk$ be the subspace generated by
\[
[\nfk'\cdot x] -  \sum_{\subfrac{a \in A}{\deg a < \deg \nfk'}} [x+\frac{a}{\nfk'}],\quad  \forall x \in \frac{1}{\nfk(\theta)}A/A \quad \text{and} \quad 
\nfk' \in A_+ \text{ with } \nfk' \mid \nfk(\theta),
\]
and
\[
\sum_{\epsilon \in \FF_q^\times} [\epsilon x], \quad \forall x \in (\frac{1}{\nfk(\theta)}A/A) \setminus \{0\}.
\]
Put $\Ucal^{\geo}_\nfk \assign \Acal^{\geo}_{\nfk,\QQ}/\Rcal^{\geo}_\nfk$ and
$\Ucal^{\geo}\assign \Acal^{\geo}_\QQ/\Rcal^{\geo}$, where
$\Rcal^{\geo} \assign \cup_{\nfk} \Rcal_\nfk^{\geo}$.
We call $\Ucal^{\geo}$ the \emph{universal distribution associated to the geometric diamond bracket relations.}

\begin{lemma}\label{lem: Ugeo-dim}
$\dim_\QQ \Ucal^{\ari}_\nfk \leq 1+ \displaystyle (1-\frac{1}{(q-1)^{\epsilon_\nfk}}) \cdot \#(A/\nfk)^\times$ for each $\nfk \in \eA_+$, where $\epsilon_\nfk \assign 1$ if $\deg \nfk>0$ and $0$ otherwise.
\end{lemma}

\begin{proof}
This is similar to Lemma~\ref{lem: Ucyc-dim} when taking $\ell=1$.
We postpone the proof to there.
\end{proof}

On the other hand,
extending $\St^{\geo}$ to a $\QQ$-linear homomorphism from $\Acal^{\geo}_{\ell,\QQ}$ (resp.\ $\Acal^{\geo}_\QQ$) to $\Sscr(\eG_{\nfk,1})$ (resp.\ $\Sscr(\eG^{\geo})$),
by Proposition~\ref{prop: St-span}~(2) we have 

\begin{lemma}\label{lem: Stgeo-surj}
For each $\nfk \in \eA_+$,
$\St^{\geo}: \Acal^{\text{\rm geo}}_{\nfk,\QQ} \rightarrow \Sscr(\eG_{\nfk,1})$ is surjective.
Consequently,
$\St^{\geo} : \Acal^{\text{\rm geo}}_\QQ \rightarrow \Sscr(\eG^{\text{\rm geo}})$
is surjective.
\end{lemma}

Notice that $\dim_\QQ \Sscr(\eG_{\nfk,1}) = \displaystyle 1 + (1-\frac{1}{(q-1)^{\epsilon_\nfk}}) \cdot \#(\eA/\nfk)^\times$ by Remark~\ref{rem: S(G/H)-dim}.
From a similar approach as in the arithmetic case, we obtain 

\begin{corollary}\label{cor: Ugeo-iso}
The map from $\Ucal^{\geo}_\nfk$ to  $\Sscr(\eG_{\nfk,1})$ induced by $\St^{\geo}$ is a $\QQ$-linear isomorphism.
Consequently,
\begin{equation}\label{eqn: Ugeo-dim}
\dim_\QQ \Ucal^{\geo}_\nfk = \dim_\QQ \Sscr(\eG_{\nfk,1}) = \displaystyle 1 + (1-\frac{1}{(q-1)^{\epsilon_\nfk}}) \cdot \#(\eA/\nfk)^\times,
\end{equation}
and we have a $\QQ$-linear isomorphism
$\Ucal^{\geo} \cong \Sscr(\eG^{\geo})$.
\end{corollary}

\subsubsection*{III. The two-variable case}

Let $\Acal^{\cyc}$ be the free abelian group generated by elements in $(k/A) \times (\ZZ_{(p)}/\ZZ)$,
and regard $\Acal^{\cyc}$ as a subgroup of $\Acal^{\cyc}_\QQ \assign \QQ \otimes_\ZZ \Acal^{\cyc}$.
Every element in $\Acal^{\cyc}$ (resp.\ $\Acal^{\cyc}_\QQ$) can be written uniquely as
$$
\ba = \sum_{x,y} m_{x,y} \cdot [x,y], \quad m_{x,y} \in \ZZ \ \text{ (resp.\ $\QQ$)}
$$
and 
$m_{x,y} = 0$ for almost all $(x,y) \in (k/A)\times (\ZZ_{(p)}/\ZZ)$.
Given $\nfk \in \eA_+$ and $\ell \in \NN$,
let $\Acal^{\cyc}_{\nfk,\ell,\QQ}$ be the subspace of $\Acal^{\cyc}_\QQ$ generated by $[x,y]$ for all $x \in \frac{1}{\nfk(\theta)}A/A$ and $y \in \frac{1}{q^{\ell}-1}\ZZ/\ZZ$. 
Also, we let $\Rcal^{\cyc}_{\nfk,\ell}$ be the subspace of $\Acal^{\cyc}_{\nfk,\ell,\QQ}$ generated by
\begin{equation}\label{eqn: mul-re}
[x,y]-\sum_{i=0}^{\ell-1} y_i [x,\frac{q^i}{1-q^\ell}]
\quad \text{ and } \quad
[\nfk'x,|\nfk'|_\infty y]-\sum_{\subfrac{a \in A}{\deg a < \deg \nfk'}} [x+\frac{a}{\nfk'},y],
\end{equation}
for all $x \in \frac{1}{\nfk(\theta)}A/A$, 
$\nfk' \in A_+$ with $\nfk' \mid \nfk(\theta)$, and
$y \in \frac{1}{q^\ell-1}\ZZ/\ZZ \text{ with }\langle -y \rangle_{\ari} = \sum_{i=0}^{\ell-1}y_i q^i/(q^\ell-1)$
where $0\leq y_1,...,y_{\ell-1}<q$,
and
\begin{equation}\label{eqn: ref-re}
\sum_{i=0}^{\ell-1}\sum_{\epsilon \in \FF_q^\times }[\epsilon x, q^i y],
\quad \forall x \in (\frac{1}{\nfk(\theta)}A/A)\setminus \{0\} \text{ and } y \in \frac{1}{q^\ell-1}\ZZ/\ZZ.
\end{equation}
Put $\Ucal^{\cyc}_{\nfk,\ell}\assign \Acal^{\cyc}_{\nfk,\ell,\QQ}/\Rcal^{\cyc}_{\nfk,\ell}$ and $\Ucal^{\cyc}\assign 
\Acal^{\cyc}_\QQ/\Rcal^{\cyc}$, where
$\Rcal^{\cyc} \assign \cup_{\nfk,\ell} \Rcal^{\cyc}_{\nfk,\ell}$.
We call $\Ucal^{\cyc}$ the \emph{universal distribution associated to the two-variable diamond bracket relations}.

\begin{lemma}\label{lem: Ucyc-dim}
$\dim_\QQ \Ucal^{\cyc}_{\nfk,\ell} \leq  1+ \displaystyle (\ell-\frac{1}{(q-1)^{\epsilon_\nfk}})\cdot \#(\eA/\nfk)^\times$
for each $\nfk \in \eA_+$ and $\ell \in \NN$, where $\epsilon_\nfk$ is defined in Lemma~\ref{lem: Ugeo-dim}.
\end{lemma}

\begin{proof}
When $\deg \nfk = 0$, i.e.\ $\nfk = 1$, observe that $\Ucal_{1,\ell}^{\cyc} = \Ucal_\ell^{\ari}$, and the result follows from Lemma~\ref{lem: Uari-dim}.
Suppose $\deg \nfk >0$.
Write $\nfk(\theta) = \pfk_1^{e_1}\cdots \pfk_s^{e_s}$, where $\pfk_1,...,\pfk_s \in A_+$ are distinct monic irreducible polynomials.
Adapting the argument in \cite[Theorem~9.1 and 9.2]{La90} (see also \cite{Ku79}),
the multiplication relation in \eqref{eqn: mul-re} shows that 
the space $\Ucal^{\cyc}_{\nfk,\ell}$ is generated by the images of  $[\displaystyle x,\frac{q^i}{1-q^\ell}]$ in $\Ucal^{\cyc}_{\nfk,\ell}$ for all integers $i$ with $0\leq i <\ell$ and $x$ in the following set:
\[
\left\{\sum_{j=1}^s\frac{a_j}{\pfk_j^{e_j}}\ \Bigg|\
1\neq a_j \in A,\ \deg a_j < \deg \pfk_j^{e_j}, \text{ and either }\ \pfk_j \nmid a_j\ \text{ or } \ a_j = 0
\right\}.
\]
Moreover, the reflection relation in \eqref{eqn: ref-re}
allows us to drop $\displaystyle\frac{\#(\eA/\nfk)^\times}{q-1}-1$ elements in the above generating set of $\Ucal^{\cyc}_{\nfk,\ell}$.
Hence
\begin{align*}
\dim_\QQ \Ucal^{\cyc}_{\nfk,\ell} &\leq  \ell \cdot \#(\eA/\nfk)^\times - \left(\frac{\#(\eA/\nfk)^\times }{q-1}  -1\right)\\
&= 1 + (\ell-\frac{1}{q-1})\cdot \#(\eA/\nfk)^\times 
\end{align*}
as desired.
\end{proof}

On the other hand,
extending $\St$ to a $\QQ$-linear homomorphism from $\Acal^{\cyc}_{\nfk,\ell,\QQ}$ (resp.\ $\Acal^{\cyc}_\QQ$) to $\Sscr(\eG_{\nfk,\ell})$ (resp.\ $\Sscr(\eG^{\cyc})$),
by Proposition~\ref{prop: St-span}~(c) we have

\begin{lemma}\label{lem: Stcyc-surj}
For each $\nfk \in \eA_+$ and $\ell \in \NN$,
$\St: \Acal^{\cyc}_{\nfk,\ell,\QQ} \rightarrow \Sscr(\eG_{\nfk,\ell})$ is surjective.
Consequently,
$\St: \Acal^{\cyc}_\QQ \rightarrow \Sscr(\eG^{\cyc})$
is surjective.
\end{lemma}

Notice that 
$
\dim_\QQ \Sscr(\eG_{\nfk,\ell}) = 1 + \displaystyle(\ell-\frac{1}{(q-1)^{\epsilon_\nfk}})\cdot \#(\eA/\nfk)^\times$ by Remark~\ref{rem: S(G/H)-dim}.
We then obtain 

\begin{corollary}\label{cor: Ucyc-iso}
The map from $\Ucal^{\cyc}_{\nfk,\ell}$ to  $\Sscr(\eG_{\nfk,\ell})$ induced by $\St$ is a $\QQ$-linear isomorphism.
Consequently, 
\begin{equation}\label{eqn: Ucyc-dim}
\dim_\QQ \Ucal^{\cyc}_{\nfk,\ell} = \dim_\QQ \Sscr(\eG_{\nfk,\ell}) = \displaystyle 1 + (\ell-\frac{1}{(q-1)^{\epsilon_\nfk}}) \cdot \#(\eA/\nfk)^\times,
\end{equation}
and we have a $\QQ$-linear isomorphism
$\Ucal^{\cyc} \cong \Sscr(\eG^{\cyc})$.
\end{corollary}

\begin{remark}
${}$
\begin{itemize}
\item[(1)] Corollary~\ref{cor: Uari-iso},~\ref{cor: Ugeo-iso}, and~\ref{cor: Ucyc-iso} say that the Stickelberger distributions $\St^{\ari}$, $\St^{\geo}$, and $\St$ are actually ``universal'' with respect to the corresponding diamond bracket relations. 
\item[(2)] Extending the action of $\eG^{\cyc}$ to $\Acal^{\ari}_\QQ$, 
$\Acal^{\geo}_{\QQ}$, and
$\Acal^{\cyc}_\QQ$, respectively,
induces isomorphisms in Corollary~\ref{cor: Uari-iso},~\ref{cor: Ugeo-iso}, and~\ref{cor: Ucyc-iso}, which are in fact $\eG^{\cyc}$-equivariant.
\end{itemize}
\end{remark}

In the next section, we shall 
compare the compositions $\Pscr^{\cyc}_{\nu_1}\circ \St^{\ari}:\Acal^{\ari}_\QQ\rightarrow \CC_\infty^\times/\bar{k}^\times$, $\Pscr^{\cyc}_{\nu_1}\circ \St^{\geo}: \Acal^{\geo}_\QQ \rightarrow \CC_\infty^\times/\bar{k}^\times$,
and $\Pscr^{\cyc}_{\nu_1}\circ \St:\Acal^{\cyc}_\QQ\rightarrow \CC_\infty^\times/\bar{k}^\times$
for a particular embedding $\nu_1:\eK^{\cyc}\hookrightarrow \CC_\infty$
with the corresponding ``gamma distributions''.

\section{\texorpdfstring{Gamma distributions}{Gamma distributions}}

\subsection{Generalized CM types corresponding to Stickelberger distributions}

Given $\nfk \in \eA_+$,
recall that $C_\nfk^*(t,z)$ is the $\nfk$-th Carlitz cyclotomic polynomial given in Section~\ref{sec: cyclotomic extn}.
Let $U_\nfk$ be the affine smooth curve in the $(t,z)$-plane over $\FF_q$ defined by $C_\nfk^*(t,z) = 0$, (i.e.\ $O_\nfk$ is the affine coordinate ring of $U_\nfk$), and $X_\nfk$ be the projective model of $U_\nfk$ over $\FF_q$.
The embedding $\ek \hookrightarrow \eK_\nfk$ induces a finite morphism $\pi_\nfk$ from $X_\nfk$ to $\PP^1$, the projective $t$-line over $\FF_q$.
Let $\infty$ be the point at infinity in $\PP^1$.
There exists $\#(\eA/\nfk)^\times/(q-1)$ points in $X_\nfk$ lying above $\infty$, and each point is ramified over $\infty$ with the inertia group equal to the image of $\FF_q^\times$ under the Artin map (see \eqref{eqn: inertia gp}). In particular, all of the points of $X_\nfk$ lying above $\infty$ are $\FF_q$-rational.
Finally, for $\mfk,\nfk \in \eA_+$ with $\mfk|\nfk$, the inclusion map $\eK_\mfk \hookrightarrow \eK_\nfk$ given by
$$
O_\mfk =\frac{\FF_q[t,z]}{(C_\mfk^*(t,z))} \hookrightarrow O_\nfk = \frac{\FF_q[t,z]}{(C_\nfk^*(t,z))}, \quad z \bmod C_\mfk^*(t,z)  \longmapsto C_{\nfk/\mfk}(t,z) \bmod C_\nfk^*(t,z),
$$
corresponds to a morphism $\pi_{X_\nfk/X_\mfk}: X_\nfk \rightarrow X_\mfk$.
\\

Let $\bU_\nfk$ and $\bX_\nfk$ be the base changes of $U_\nfk$ and $X_\nfk$ to $\bar{k}$, respectively.
The \textit{Carlitz exponential function} $\exp_\Ccal:\CC_\infty\rightarrow \CC_\infty$ is given by:
$$
\exp_\Ccal(z) \assign z \cdot \prod_{a \in A-\{0\}}\left(1-\frac{z}{\tilde{\pi} \cdot a}\right), \quad \forall z \in \CC_\infty,
$$
where $\tilde{\pi}$ is the Carlitz fundamental period  given in Remark~\ref{rem: Carlitz period}.
For each $\nfk \in \eA_+$, take
$$\lambda_\nfk \assign \exp_{\Ccal}\left( \frac{\tilde{\pi}}{\nfk(\theta)}\right) \quad \in \bar{k}.$$
One has that $C_\nfk^*(\theta,\lambda_\nfk) = 0$, and $C_a(\theta,\lambda_\nfk) = C_b(\theta,\lambda_\nfk)$ for $a,b \in \eA$ with $a \equiv b \bmod \nfk$.
All points in $\bX_\nfk$ lying above the point $ \theta$ in $\PP^1_{/\bar{k}}$ are precisely
$$ \xi_{\nfk,\alpha} \assign \big(\theta, C_a(\theta,\lambda_\nfk)\big) \ (\in \bU_\nfk), \quad \text{ where } \alpha = a \bmod \nfk \in (\eA/\nfk)^\times.$$
Note that the evaluation at the point $\xi_{\nfk,\alpha}$ gives an embedding $\nu_{\nfk,\alpha} : K_\nfk \hookrightarrow \bar{k}$, and 
$$
\nu_{\nfk,\alpha} = \nu_{\nfk,1} \circ \varrho_{\alpha}, \quad \forall \alpha \in (\eA/\nfk)^\times,
$$
where $\varrho_{\alpha} \in \eG_\nfk$ is the automorphism corresponding to $\alpha$ via the Artin map in \eqref{eqn: geo-Artin}.
In particular, for $\mfk \in \eA_+$ with $\mfk|\nfk$, one has that $\pi_{\bX_\nfk/\bX_\mfk}(\xi_{\nfk,1}) = \xi_{\mfk,1}$.
This induces a unique embedding $\nu_1: \eK^{\rm cyc} \hookrightarrow \bar{k}$ so that 
\begin{equation}\label{eqn: nu_1}
\nu_1\big|_{\overline{\FF}_q} = \text{id}_{\overline{\FF}_q} \quad \text{ and } \quad 
\nu_1\big|_{\eK_\nfk} = \nu_{\nfk,1}, \quad \forall \nfk \in \eA_+.
\end{equation}
Moreover, let $\ell \in \NN$.
For $(\alpha,c)  \in (\eA/\nfk)^\times
\times (\ZZ/\ell \ZZ)$,
put $\nu_{\alpha,c} \assign \nu_{\nfk,1} \circ \varrho_{\alpha,c}$ where $\varrho_{\alpha,c}$ is given in \eqref{eqn: cyc-Artin}.
Then
$$
\nu_{\alpha,c}\big|_{\FF_{q^\ell}} = \text{Frob}_q^{c_\ell} \quad \text{ and } \quad \nu_{\alpha,c} \big|_{\eK_\nfk} = \nu_{\nfk,\alpha}.
$$

Set
\[
U_{\nfk,\ell}\assign \FF_{q^\ell} \times_{\FF_q} U_\nfk \quad \subset \quad X_{\nfk,\ell} \assign \FF_{q^\ell} \times_{\FF_q} X_{\nfk},
\]
and
\[
\bU_{\nfk,\ell}\assign \bar{k}\times_{\FF_q} U_{\nfk,\ell} = \coprod_{i=0}^{\ell-1} \bU_{\nfk,\ell,(i)}
\quad \subset \quad
\bX_{\nfk,\ell}\assign \bar{k}\times_{\FF_q} X_{\nfk,\ell} = \coprod_{i=0}^{\ell-1}
\bX_{\nfk,\ell,(i)},
\]
where
\begin{align*}
\bU_{\nfk,\ell,(i)} & \assign
(\bar{k} \underset{\text{Frob}_q^{-i},\FF_{q^\ell}}{\times} \FF_{q^\ell}) \times_{\FF_q} U_{\nfk} \hspace{-0.7cm}
& \subset \quad 
\bX_{\nfk,\ell,(i)} &\assign
(\bar{k} \underset{\text{Frob}_q^{-i},\FF_{q^\ell}}{\times} \FF_{q^\ell}) \times_{\FF_q} X_{\nfk} \\
&\, (\cong \bar{k} \times_{\FF_q} U_\nfk = \bU_\nfk)
& 
&\, (\cong \bar{k} \times_{\FF_q} X_\nfk = \bX_\nfk).
\end{align*}
Moreover, $J_{\eK_{\nfk,\ell}} = \{\xi_{\alpha,c}\mid \alpha \in (\eA/\nfk)^\times \text{ and } c \in \ZZ/\ell \ZZ\}$, where $\xi_{\alpha,c} \in \bU_{\nfk,\ell,(\ell-c)}$ corresponds to
the embedding of $\nu_{\alpha,c} :\eK_{\nfk,\ell} \hookrightarrow \CC_\infty$.
In particular, the point $\xi_{\alpha,c}$ actually coincides with $\xi_{\alpha}$ when identifying $\bU_{\nfk,\ell,(\ell-c)}$ with $\bU_\nfk$.\\

Now, for $x \in \frac{1}{\nfk(\theta)}A/A$ and an integer $c$ with $0\leq c< \ell$,
we set
\begin{align*}
\Xi_{x,c} & \assign (q^\ell-1)\cdot \sum_{i=0}^{\ell-1}\sum_{\alpha \in (\eA/\nfk)^\times} \langle \varrho_{\alpha,i} \star x\, ,\, \varrho_{\alpha,i}  \star \frac{q^{c}}{q^\ell-1}\rangle \cdot \xi_{\alpha,i} \\
&= (q^\ell-1)\cdot \sum_{i=0}^{\ell-1}\sum_{\alpha \in (\eA/\nfk)^\times} \langle \alpha(\theta)\cdot x ,\frac{q^{i+c}}{q^\ell-1}\rangle \cdot \xi_{\alpha,i}.
\end{align*}
We also put
\begin{align*}
    \Xi_{x} & \assign (q^\ell-1)\cdot  \sum_{i=0}^{\ell-1} \sum_{\alpha \in (\eA/\nfk)^\times} \langle \varrho_{\alpha,i} \star x\rangle_{\geo} \cdot \xi_{\alpha,i} \\
    &= (q^\ell-1) \cdot \sum_{\alpha \in (\eA/\nfk)^\times} \langle \alpha(\theta) \cdot x\rangle_{\geo} \cdot \pi^*_{\bX_{\nfk,\ell}/\bX_{\nfk}}(\xi_\alpha),
\end{align*}
and
\begin{align*}
\Xi_{\text{ari},c}
&\assign 
(q^\ell-1) \cdot \sum_{i=0}^{\ell-1}\sum_{\alpha \in (\eA/\nfk)^\times}\langle \varrho_{\alpha,i} \star \frac{q^{c}}{q^\ell-1}\rangle_{\text{ari}} \cdot \xi_{\alpha,i} \\
&= (q^\ell-1) \cdot \sum_{i=0}^{\ell-1}\sum_{\alpha \in (\eA/\nfk)^\times}\langle \frac{q^{i+c}}{q^\ell-1}\rangle_{\text{ari}} \cdot \xi_{\alpha,i} \\
&= \sum_{j=0}^{c-1}q^j \cdot \pi_{\bX_{\nfk}/\PP^1}^*(\theta)_{j-c+\ell} +\sum_{j=c}^{\ell-1} q^j \cdot \pi_{\bX_{\nfk}/\PP^1}^*(\theta)_{j-c}.
\end{align*}
Here for $0\leq i \leq \ell-1$, $\pi_{\bX_\nfk/\PP^1}^*(\theta)_i \in \Div(\bX_{\nfk,\ell,(\ell-i)})$ is the divisor
corresponding to $\pi_{\bX_\nfk/\PP^1}^*(\theta)$
under the identification $\bX_{\nfk,\ell,(\ell-i)}\cong \bX_\nfk$.
In particular, by Remark~\ref{rem: 2var-diamond-1}~(3) one has that $\Xi_{x} = \sum_{c=0}^{\ell-1} \Xi_{x,c}$.
In general, for each $y \in \frac{1}{q^\ell-1}\ZZ/\ZZ$, write
$\langle y\rangle_{\ari} = \sum_{c=0}^{\ell-1}y_c q^c/(q^\ell-1)$ where $y_0,...y_{\ell-1} \in \ZZ$ with $0\leq y_0,...,y_{\ell-1}<q$.
We set
\[
\Xi_{(x,y)}\assign
\sum_{c=0}^{\ell-1}y_c \Xi_{x,c}, \quad 
\Xi_{y}
\assign
\sum_{c=0}^{\ell-1}y_c \Xi_{\ari,c},
\]
and
\[
\Phi_{x} \assign \Xi_{x}-\Xi_{1/(q-1)},\quad  
\Phi_{(x,y)} \assign 
\Xi_{(x,-y)} - \Xi_{-y}.
\]

From the above construction, the following lemma is straightforward:

\begin{lemma}\label{lem: ST}
Let $\nfk \in \eA_+$ and $\ell \in \NN$.
Given $x \in \frac{1}{\nfk(\theta)}A/A$ and $y \in \frac{1}{q^\ell-1}\ZZ/\ZZ$,
we have that
$\Xi_{x}$, $\Xi_{y}$, and $\Xi_{(x,y)}$ are generalized CM types of $\eK_{\nfk,\ell}$,
and $\Phi_{x},\ \Phi_{(x,y)} \in I_{\eK_{\nfk,\ell}}^0$.
Moreover,
\begin{align*}
\St_0^{\geo}(x) &= (q^\ell-1)^{-1}\varphi_{\eK_{\nfk,\ell},\Xi_{x}}, \\
\St^{\ari}(-y) = \St^{\ari}_0(y)  &=
(q^\ell-1)^{-1} \varphi_{\eK_{\nfk,\ell},\Xi_{y}} ,  \\
\St_0(x,y) & = 
(q^\ell -1)^{-1} \varphi_{\eK_{\nfk,\ell},\Xi_{(x,y)}} , \\
\St^{\geo}(x) &= 
(q^\ell -1)^{-1}\varphi_{\eK_{\nfk,\ell},\Phi_{x}}, \\
\St(x,y)
&=
(q^\ell-1)^{-1}\varphi_{\eK_{\nfk,\ell},\Phi_{(x,y)}}.
\end{align*}
\end{lemma}

\subsection{Shtuka functions}

To construct a CM dual $t$-motive with generalized CM type $(\eK_{\nfk,\ell},\Xi)$
where $\Xi = \Xi_{x}$, $\Xi_{y}$, or $\Xi_{(x,y)}$,
we shall find the corresponding ``shtuka functions'' and
apply the geometric approach in \cite[Appendix~B]{BCPW22}.
First, recall the following result:

\begin{theorem}\label{thm: GCFn}
{\rm (\cite[Section 6.3.9]{ABP})} 
Let $\nfk \in \eA_+$ with $\deg \nfk > 0$.
Given nonzero $x$ in $\frac{1}{\nfk(\theta)}A/A$, there exists a regular function 
$g_x = 1 + \sum_{i,j} c_{ij}t^iz^j \in O_{\bK_\nfk} = \ok\otimes_{\FF_q}O_\nfk$ where $c_{ij} \in \bar{k}$ with $|c_{ij}|_\infty < 1$ satisfying that 
$$
\divv(g_x) = -  \infty_{X_\nfk} +
\sum_{\alpha \in (\eA/\nfk)^\times}\sum_{N=0}^\infty \langle \alpha(\theta) \cdot x\rangle_N \cdot \xi_\alpha^{(N)}
\quad \in \Div(\bX_\nfk),
$$
where
\[
\infty_{X_\nfk} \assign  \sum_{\subfrac{\tilde{\infty} \in X_\nfk}{\scaleto{\pi}{3.3pt}_{\!\scaleto{X_\nfk/\PP^1}{6.3pt}}(\tilde{\infty}) = \infty}}\tilde{\infty}.
\]
\end{theorem}

\begin{remark}\label{rem: g-twist}
Let $n$ be an integer.
For each rational function $f$ on $\bX_{\nfk}$ (i.e.\ $f$ is in $\bX_{\nfk}$'s function field $\bK_{\nfk} \assign \bar{k}(t) \otimes_{\FF_q(t)} \eK_{\nfk}$),
let $f^{(n)}$ be the {\it $q^n$-power Frobenius twist of $f$}, i.e.\ $f^{(n)}$ is obtained by raising every coefficient of $f$ to the $q^n$-power.
Moreover, for each ($\bar{k}$-)point $P \in \bX_\nfk$, 
its \emph{$q^n$-power Frobenius twist} is the point $P^{(n)}$ on $\bX_{\nfk}$ obtained by raising every coordinate of $P$ to the $q^n$-power.
Extending the $q^n$-power Frobenius twist additively to the divisors on $\bX_\nfk$,
by Theorem~\ref{thm: GCFn} we have
\[
\divv(g_x^{(n)}) =  \divv(g_x)^{(n)} =
- \infty_{X_\nfk} +
\sum_{\alpha \in (\eA/\nfk)^\times}\sum_{N=0}^\infty \langle \alpha(\theta)\cdot x\rangle_N \cdot \xi_\alpha^{(n+N)}
\quad \in \Div(\bX_\nfk).
\]
\end{remark}

Given nonzero $x \in \frac{1}{\nfk(\theta)}A/A$ and $0\leq c<\ell$,
we may write
\[
\Xi_{x,c} = (q^\ell-1)\cdot \sum_{i=0}^{\ell-1}\sum_{\alpha \in (\eA/\nfk)^\times}
\sum_{\subfrac{N\geq 0}{N \equiv -1-i-c \bmod \ell}} \langle \alpha(\theta)\cdot x \rangle_N \cdot \xi_{\alpha,i}
= \Xi_{x,c,(0)}+\sum_{i=1}^{\ell-1} \Xi_{x,c,(i)},
\]
where
\[
\Xi_{x,c,(0)} \assign 
(q^\ell-1)\cdot \sum_{\alpha \in (\eA/\nfk)^\times}
\sum_{\subfrac{N\geq 0}{N \equiv -1-c \bmod \ell}} \langle \alpha(\theta)\cdot x \rangle_N \cdot \xi_{\alpha,0} \quad \in \Div(\bX_{\nfk,\ell,(0)}),
\]
and for $1\leq i < \ell$,
\[
\Xi_{x,c,(i)} \assign 
(q^\ell-1)\cdot \sum_{\alpha \in (\eA/\nfk)^\times}
\sum_{\subfrac{N\geq 0}{N \equiv -1-(\ell-i)-c \bmod \ell}} \langle \alpha(\theta)\cdot x \rangle_N \cdot \xi_{\alpha,\ell-i} \quad \in \Div(\bX_{\nfk,\ell,(i)}).
\]
In particular, put
\begin{align*}
\Xi_{x,c}^{\#} 
&\assign  (q^\ell-1)\cdot 
\sum_{i=0}^{\ell-1} \sum_{\alpha \in (\eA/\nfk)^\times} 
\left(\sum_{\subfrac{N\geq 0}{N \equiv -1-i-c \bmod \ell}}
\langle \alpha(\theta)\cdot x \rangle_N\right) \cdot \xi_{\alpha}^{(-i)}
 \quad \in \Div(\bX_{\nfk}) \\
&(= \Xi_{x,c,(0)} + \sum_{i=1}^{\ell-1} \Xi_{x,c,(i)}^{(i-\ell)} \quad \text{when identifying $\bX_{\nfk,\ell,(i)}$ with $\bX_\nfk$ for $0\leq i<\ell$}) \\
& = (q^\ell-1)\cdot \Bigg[\sum_{i=0}^{\ell-c-1} \sum_{\alpha \in (\eA/\nfk)^\times} \langle \alpha(\theta)\cdot x\rangle_{-1-i-c+\ell}\cdot \xi_\alpha^{(-i)} \\ 
& \hspace{2cm} + 
\sum_{i=0}^{\ell-1} \ \sum_{\alpha \in (\eA/\nfk)^\times}\left(\sum_{N'\geq 1}\langle \alpha(\theta)\cdot x\rangle_{-1-i-c+\ell+\ell N'}
\right)  \cdot \xi_\alpha^{(-i)}\Bigg].
\end{align*}
Take 
\[
W_{x,c}\assign (q^\ell-1)\cdot \sum_{i=0}^{\ell-1}\sum_{\alpha \in (\eA/\nfk)^\times}
\left(
\sum_{N' \geq 1}
\langle \alpha(\theta)\cdot x\rangle_{-1-i-c+\ell+\ell N'} \sum_{j=0}^{N'-1}\xi_\alpha^{(-i+\ell j)}
\right) \quad \in \Div(\bX_\nfk)
\]
and $h_{x,c} \assign (g_x^{(1+c-\ell)})^{q^\ell-1} \in O_{\bK_\nfk}$.
Under the identification between $\bX_{\nfk}$ and $\bX_{\nfk,\ell,(0)}$,
we obtain that
\begin{align}
\divv(h_{x,c}) 
& = -(q^\ell-1) \cdot \infty_{X_\nfk}+(q^\ell-1)\cdot
\sum_{\alpha \in (\eA/\nfk)^\times}
\sum_{N=0}^\infty \langle \alpha(\theta)\cdot x\rangle_N \cdot \xi_\alpha^{(1+c-\ell + N)} \nonumber \\
& = -(q^\ell-1) \cdot \infty_{X_\nfk}+(q^\ell-1)\cdot \sum_{i=0}^{\ell-1}\sum_{\alpha \in (\eA/\nfk)^\times}\sum_{\substack{N\geq 0 \\ N \equiv -1-i-c \bmod \ell}}\langle \alpha(\theta)\cdot x\rangle_N \cdot \xi_\alpha^{(1+c-\ell+N)} \nonumber \\
& = 
-(q^\ell-1) \cdot \infty_{X_\nfk}+(q^\ell-1)\cdot
\Bigg[
\sum_{i=0}^{\ell-c-1} \sum_{\alpha \in (\eA/\nfk)^\times} \langle \alpha(\theta)\cdot x\rangle_{-1-i-c+\ell}\cdot \xi_\alpha^{(-i)}  \nonumber \\ 
& \hspace{4cm} + 
\sum_{i=0}^{\ell-1} \ \sum_{\alpha \in (\eA/\nfk)^\times}\left(\sum_{N'\geq 1}\langle \alpha(\theta)\cdot x\rangle_{-1-i-c+\ell+\ell N'}
\right)  \cdot \xi_\alpha^{(-i+\ell N')}\Bigg]
\nonumber \\
& = -(q^\ell-1) \cdot \infty_{X_\nfk}\ +\  \Xi_{x,c}^{\#}\ +\ W_{x,c}^{(\ell)}\ -\ W_{x,c} \hspace{1.2cm} \in \Div(\bX_{\nfk,\ell,(0)}).
\end{align}
For convenience, we also put $h_{0,c}=1$, $\Xi_{0,c}^\# = W_{0,c} = 0$ for $0\leq c <\ell$.\\

In the arithmetic case,
we put
\[
\Xi_{\text{ari},c}^{\#} = \sum_{i=0}^{c-1}q^i \cdot \pi_{\bX_{\nfk}/\PP^1}^*(\theta)^{(c-i-\ell)} +\sum_{i=c}^{\ell-1} q^i \cdot \pi_{\bX_{\nfk}/\PP^1}^*(\theta)^{(c-i)} \quad \in \Div(\bX_{\nfk}).
\]
Take 
\begin{align*}
h_{\text{ari},c} & \assign \prod_{i=0}^{c-1}(1-\frac{t}{\theta^{q^{c-i-\ell}}})^{q^i} \cdot \prod_{i=c}^{\ell-1}(1-\frac{t}{\theta^{q^{c-i}}})^{q^i} \\
& =
\prod_{i=0}^{c-1}(1-\frac{t^{q^i}}{\theta^{q^{c-\ell}}}) \cdot \prod_{i=c}^{\ell-1}(1-\frac{t^{q^i}}{\theta^{q^{c}}}) \quad \in \bar{k}[t] \subset O_{\bK_{\nfk}}
\end{align*}
and
$W_{\text{ari},c} \assign 0 \in \Div(\bX_{\nfk})$.
Note that $\pi^*_{\bX_{\nfk}/\PP^1}(\infty) = (q-1)\cdot \infty_{\bX_\nfk}$ and $\ord_{\infty}(h_{\ari,c}) = -\displaystyle\frac{q^\ell-1}{q-1}$.
Under the identification between $\bX_{\nfk}$ and $\bX_{\nfk,\ell,(0)}$, we have
\begin{equation}
\divv(h_{\text{ari},c}) = -(q^\ell -1)\cdot  \infty_{X_\nfk} + \Xi_{\text{ari},c}^{\#} + W_{\text{ari},c}^{(\ell)} - W_{\text{ari},c} \quad \in \Div(\bX_{\nfk,\ell,(0)}).
\end{equation}

In general, let $y \in \frac{1}{q^\ell-1}\ZZ/\ZZ$ and write
$\langle y\rangle_{\ari} = \sum_{c=0}^{\ell-1}y_c q^c/(q^\ell-1)$ where $y_0,...y_{\ell-1} \in \ZZ$ with $0\leq y_0,...,y_{\ell-1}<q$.
We set
\[
\Xi^{\#}_x \assign \sum_{c=0}^{\ell-1}\Xi^{\#}_{x,c}, \quad 
\Xi^{\#}_y \assign \sum_{c=0}^{\ell-1}y_c \Xi^{\#}_{\ari,c},\quad 
\Xi^{\#}_{(x,y)} \assign \sum_{c=0}^{\ell-1}y_c \Xi^{\#}_{x,c},
\]
\[
W_{x} \assign \sum_{c=0}^{\ell-1} W_{x,c},\quad 
W_{y} \assign 
\sum_{c=0}^{\ell-1} y_c W_{\text{ari},c},\quad W_{(x,y)}\assign \sum_{c=0}^{\ell-1}y_c W_{x,c},
\]
and
\[
h_{x} \assign \prod_{c=0}^{\ell-1} h_{x,c},\quad 
h_{y} \assign \prod_{c=0}^{\ell-1} h_{\text{ari},c}^{y_c}, \quad
h_{(x,y)}\assign
\prod_{c=0}^{\ell-1}h_{x,c}^{y_c}.
\]
Then:
\begin{proposition}
Let $\nfk \in \eA_+$ and $\ell \in \NN$.
Given $x \in \frac{1}{\nfk(\theta)}A/A$ and $y \in \frac{1}{q^\ell-1}\ZZ/\ZZ$, we have that
\begin{align*}
\divv(h_{x}) &= - (q^\ell-1)\cdot \ell\text{\rm wt}_0^{\geo}(x) \infty_{X_\nfk} + \Xi_{x}^{\#} +W_{x}^{(\ell)} - W_{x}, \\
\divv(h_{y}) &= - (q^\ell-1)\cdot \ell\text{\rm wt}^{\ari}_0(y) \infty_{X_\nfk} + \Xi_{y}^{\#} +W_{y}^{(\ell)} - W_{y},\\
\text{and }\quad \divv(h_{(x,y)}) &= - (q^\ell-1)\cdot \ell \text{\rm wt}_0(x,y) \infty_{X_\nfk} + \Xi_{(x,y)}^{\#} +W_{(x,y)}^{(\ell)} - W_{(x,y)}. 
\end{align*}
\end{proposition}

\subsection{\texorpdfstring{Generalized soliton dual $t$-motives}{Generalized soliton dual t-motives}}\label{sec: soliton}

We now apply the geometric method from \cite[(B.2.1)]{BCPW22} to construct the desired CM dual $t$-motives.
Set 
\[
O_{\nfk,\ell} \assign \FF_{q^\ell} \otimes_{\FF_q} O_{\nfk}
\quad \text{ and } \quad 
O_{\bK_{\nfk,\ell}} \assign \bar{k}\otimes_{\FF_q} O_{\nfk,\ell} \cong \prod_{i=0}^{\ell-1} O_{\bK_{\nfk,\ell},(i)},
\]
where for $0\leq i<\ell$,
\[
O_{\bK_{\nfk,\ell},(i)} \assign \bar{k}\underset{\text{Frob}_q^{-i},\FF_{q^\ell}}{\otimes} O_{\nfk,\ell} \quad 
(\cong \bar{k} \otimes_{\FF_q} O_\nfk = O_{\bK_{\nfk}})
\]
is the affine coordinate ring of $\bU_{\nfk,\ell,(i)}$.
The field of fractions of $O_{\bK_{\nfk,\ell},(i)}$ is denoted by $\bK_{\nfk,\ell,(i)}$, which is the function field of $\bX_{\nfk,\ell,(i)}$.

Given $x \in \frac{1}{\nfk(\theta)}A/A$ and $y \in \frac{1}{q^\ell-1}\ZZ/\ZZ$,
let $*$ be either $x$, $y$, or $(x,y)$, and
set
$$
\eM(*) = \eM_{(W_{*},h_*)} = \prod_{i=0}^{\ell-1} \eM_{(W_*,h_*),(i)},
$$
where
$$
\eM_{(W_*,h_*),(0)}\assign  \Gamma\big(\bU_{\nfk,\ell,(0)}, \Ocal_{\bX_{\nfk,\ell,(0)}}(-W_*^{(\ell)})\big) \quad \subset O_{\bK_{\nfk,\ell},(0)}, $$
and
$$
\eM_{(W_*,h_*),(i)}\assign  \Gamma\big(\bU_{\nfk,\ell,(i)}, \Ocal_{\bX_{\nfk,\ell, (i)}}(-W_*^{(\ell-i)} + \sum_{j=1}^i \Xi_{*,(j)}^{(j-i)})\big) \quad \subset \bK_{\nfk,\ell,(i)}, \quad 1\leq i < \ell.
$$
The $\sigma$-action on $\eM_{(W_*,h_*)}$ is given by
$$
\sigma \cdot (m_{(0)},\ldots, m_{(\ell-1)}) \assign  \Big(h_* \cdot m_{(\ell-1)}^{(-1)}, m_{(0)}^{(-1)},\ldots, m_{(\ell-2)}^{(-1)}\Big), \quad
\forall (m_{(0)},\ldots, m_{(\ell-1)}) \in \eM_{(W_*,h_*)}.
$$
By \cite[Theorem~B.2.3]{BCPW22}, we have the following:

\begin{lemma}\label{lem: M(*)}
Let $\nfk \in \eA_+$ and $\ell \in \NN$.
Given $x \in \frac{1}{\nfk(\theta)}A/A$
and $y \in \frac{1}{q^\ell-1}\ZZ/\ZZ$,
take $*$ to be either $x$, $y$, or $(x,y)$.
$\eM(*)$ is a CM dual $t$-motive with generalized CM type $(\eK_{\nfk,\ell},\Xi_{*})$ over $\ok$.
\end{lemma}

In particular, Theorem~\ref{thm: GCFn} shows that:

\begin{lemma}\label{lem: psi*}
Keep the notation from Lemma~\ref{lem: M(*)}.
Let $\OO_\nfk \assign  \TT \otimes_{\FF_{q}[t]} O_\nfk $.
The infinite product
\[
\psi_{*} \assign  \prod_{j=1}^\infty \left(h_*^{-1}\right)^{(\ell j)} \quad \text{ lies in }\quad  \OO_{\nfk}^\times,
\]
and
\[
\tilde{\psi}_* \assign (\psi_*,\psi_*^{(-1)},...,\psi_*^{(1-\ell)})
\in H_{\rm Betti}(\eM(*)) \subset \MM(*) \assign \TT \otimes_{\bar{k}[t]} \eM(*).
\]
\end{lemma}

\begin{proof}
The proof is adapted from \cite[Lemma 6.4.3]{ABP}.
We include the details for completeness.

The support of $W^{(\ell)}_*$ in $\bX_\nfk$ is contained in the pre-image of $\{t=\theta^{q^i}\mid i \in \NN\}$.
Thus when identifying $O_{\bK_{\nfk,\ell},(0)}$ with $O_{\bK_\nfk}$,
there exist $f_{1},...,f_{\ell} \in \bK_\nfk$ (the function field of $\bX_\nfk$) with possible poles lying over $\{t=\theta^{q^i}\mid i \in \NN \}$ and $m_1,...,m_\ell \in \eM_{(W_*,h_*),(0)}$ such that 
$$
f_{1} m_1 + \cdots + f_{\ell} m_\ell = 1 \in O_{\bK_\nfk}.
$$
Then the restriction on the possible poles of $f_{1},...,f_{\ell}$ forces that $f_1,...,f_\ell \in \OO_{\nfk}$, which means that $1 \in \MM_{(W_*,h_*),(0)} \assign \TT\otimes_{\bar{k}[t]} \eM_{(W_*,h_*),(0)}$ and
$$
\MM_{(W_*,h_*),(0)} = \OO_\nfk.
$$
Moreover, by Remark~\ref{rem: g-twist} we are able to show that $h_*^{(\ell j)}$ lies in $\OO_\nfk^\times$ for all $j \in \NN$,
and
the condition on the coefficients of $g_x$ in Theorem~\ref{thm: GCFn} ensures that the infinite product $\psi_*$ lies in $\OO_\nfk^\times$.

To prove the second statement,
observe that when identifying $\bK_{\nfk,\ell,(i)}$ with $\bK_{\nfk}$ for $1\leq i < \ell$,
one has that
\[
\eM_{(W_*,h_*),(i)} \supset 
\Gamma\big(\bU_{\nfk}, \Ocal_{\bX_{\nfk}}(-W_*^{(\ell-i)})\big),
\]
whence
\[
\MM_{(W_*,h_*),(i)} \assign \TT\otimes_{\bar{k}[t]} \eM_{(W_*,h_*),(i)}  \ \supset \TT\otimes_{\bar{k}[t]} \Gamma\big(\bU_{\nfk}, \Ocal_{\bX_{\nfk}}(-W_*^{(\ell-i)})\big) = \OO_\nfk \ \ni \psi^{(-i)}.
\]

In particular,
$(\psi_*,\psi_{*}^{(-1)},...,\psi_*^{(1-\ell)}) \in \MM(*)$.
Finally, it is straightforward that
$$
\sigma\cdot 
(\psi_{*},\psi_{*}^{(-1)},...,\psi_*^{(1-\ell)})
= (\psi_{*}^{(-\ell)} h_*, \psi_{*}^{(-1)},...,\psi_*^{(1-\ell)})
= (\psi_{*},\psi_{*}^{(-1)},...,\psi_*^{(1-\ell)}).
$$
Therefore $(\psi_{*},\psi_{*}^{(-1)},...,\psi_*^{(1-\ell)}) \in H_{\text{Betti}}(\eM(*))$.
\end{proof}

\begin{remark}\label{rem: EV-psi}
Identifying $O_{\bK_{\nfk,\ell},(0)}$ with $O_{\bK_\nfk}$,
the above proof actually shows that 
\[
\psi_* \in 
\TT^\dagger \otimes_{\bar{k}[t]} O_{\bK_{\nfk,\ell},(0)} \rassign \OO_{\nfk,\ell,(0)}^\dagger.
\]
This enables us to evaluate $\psi_*$ at $\xi_{1,0} = (\theta, \lambda_\nfk) \in \bU_{\eK_{\nfk,\ell},(0)}$:
\end{remark}

\begin{proposition}\label{prop: key-equality}
Let $\nfk \in \eA_+$ and $\ell \in \NN$.
Given nonzero $x \in \frac{1}{\nfk(\theta)}A$
with $|x|_\infty <1$ and $y = q^c/(q^\ell-1)$
for an integer $c$ with $0\leq c <\ell$, we have that
\[
\psi_{(x,y)}(\xi_{1,0})
= \prod_{j=0}^\infty \left(\prod_{\subfrac{a \in A}{\deg a = c+\ell j}}\Big(1+\frac{x}{a}\Big)^{1-q^\ell}\right)
\]
and
\[
\psi_{y}(\xi_{1,0})
=
\prod_{j=0}^\infty\left(
\prod_{i=0}^{c-1}(1-\frac{\theta^{q^i}}{\theta^{q^{c+\ell j}}}) \cdot \prod_{i=c}^{\ell-1}(1-\frac{\theta^{q^i}}{\theta^{q^{c+\ell+\ell j}}})\right)^{-1}.
\]
\end{proposition}

\begin{proof}
Given nonzero $x \in \frac{1}{\nfk(\theta)}A$ with $|x|_\infty <1$ and a non-negative integer $N$,
from \cite[6.3.7]{ABP} one has that
\[
g_x^{(N+1)}(\theta,\lambda_\nfk) = \prod_{\subfrac{a \in A_{\scaleto{+}{4pt}}}{\deg a = N}}\left(1+\frac{x}{a}\right).
\]
Thus for $0\leq c < \ell$ and $j \in \NN$,
we get
\[
(h_{x,c}^{-1})^{(\ell j)}(\xi_{1,0}) = g_x^{(c+1-\ell+\ell j)}(\theta,\lambda_\nfk)^{1-q^\ell} = \prod_{\subfrac{a \in A_{\scaleto{+}{4pt}}}{\deg a = c+\ell (j-1)}}\left(1+\frac{x}{a}\right)^{1-q^\ell}.
\]
Moreover, from the definition of $h_{\ari,c}$ it is straightforward that
\[
(h_{\ari,c}^{-1})^{(\ell j)}(\xi_{1,0})
= \left(\prod_{i=0}^{c-1}(1-\frac{\theta^{q^i}}{\theta^{c-\ell+\ell j}}) \prod_{i=c}^{\ell-1}(1-\frac{\theta^{q^i}}{\theta^{q^{c+\ell j}}})\right)^{-1}.
\]
Therefore the result holds.
\end{proof}

\begin{remark}\label{rem: psi_x}
Using the notation in Proposition~\ref{prop: key-equality},
we get in particular that
\[
\psi_x(\xi_{1,0}) = \prod_{c=0}^{\ell-1}\psi_{(x,\frac{q^c}{q^\ell-1})}(\xi_{1,0})
= \prod_{i=0}^\infty \prod_{\subfrac{a \in A_{\scaleto{+}{4pt}}}{\deg a = i}}\left(1+\frac{x}{a}\right)^{1-q^\ell} = \prod_{a \in A_{\scaleto{+}{4pt}}}\left(1+\frac{x}{a}\right)^{1-q^\ell} .
\]
\end{remark}

\subsection{Algebraic relations among special gamma values}\label{sec: 2var-Gamma}

Let $\ZZ_p$ be the ring of $p$-adic integers (i.e.\ the completion of $\ZZ_{(p)}$).
Put $-A_+\assign \{-a\mid a \in A_+\}$ and
\[
A_{+,i}\assign \{a \in A_+\mid \deg a = i\} \quad \text{for each non-negative integer $i$.}
\]
Recall the following {\it factorial functions} (see \cite{Goss88} and \cite{Thakur91}): for $x \in \CC_\infty \setminus -A_+$ and $y = \sum_{i=0}^\infty y_i q^i \in \ZZ_p$ with $0\leq y_i<q$ for each $i \in \ZZ_{\geq 0}$,
\begin{align*}
\Pi_{\text{ari}}(y) & \assign 
\prod_{i=1}^\infty\left(\prod_{j=0}^{i-1}(1-\frac{\theta^{q^j}}{\theta^{q^{i}}})
\right)^{y_i}, &
\Pi_{\text{geo}}(x)
& \assign
\prod_{a \in A_{\scaleto{+}{4pt}}}\left(1+\frac{x}{a}\right)^{-1}, \\
\Pi_{\geo}(x,y)
&\assign 
\prod_{i=0}^{\infty}
\left(
\prod_{a \in A_{{\scaleto{+}{4pt}},i}} \Big(1+\frac{x}{a}\Big)^{-y_i}\right),  &
\Pi(x,y)
&\assign \frac{\Pi_{\geo}(x,y)}{\Pi_{\text{ari}}(y)}.
\end{align*}
In particular, for $x \in \CC_\infty \setminus -A_+$ one has that
\[
\Pi_{\geo}(x,\frac{1}{1-q}) = \Pi_{\text{geo}}(x)
\quad \text{and}\quad 
\Pi(x,\frac{1}{1-q}) = \frac{\Pi_{\text{geo}}(x)}{\Pi_{\text{ari}}(1/(1-q))}.
\]
The corresponding {\it gamma functions} are defined as follows:
$x \in \CC_\infty \setminus (-A_+ \cup \{0\})$ and $y = \sum_{i=0}^\infty y_i q^i \in \ZZ_p$ with $0\leq y_i<q$ for $i \in \ZZ_{\geq 0}$,
\begin{align*}
\Gamma_{\text{ari}}(y) &\assign
\Pi_{\text{ari}}(y-1),
&
\Gamma_{\text{geo}}(x)
&\assign \frac{1}{x} \cdot \Pi_{\text{geo}}(x), \\
\Gamma_{\geo}(x,y)
&\assign
\frac{1}{x}\cdot \Pi_{\geo}(x,y-1), &
\Gamma(x,y)
& \assign 
\frac{1}{x} \cdot \Pi(x,y-1) = \frac{\Gamma_{\geo}(x,y)}{\Gamma_{\text{ari}}(y)}.
\end{align*}
In particular, for $x \in \CC_\infty \setminus (-A_+ \cup \{0\})$ one has that
\[
\Gamma_{\geo}(x,1-\frac{1}{q-1}) = \Gamma_{\text{geo}}(x)
\quad \text{and}\quad 
\Gamma(x,1-\frac{1}{q-1}) = \Gamma_{\text{geo}}(x)\Big/ \Gamma_{\text{ari}}(1-\frac{1}{q-1}).
\]

\begin{remark}\label{rem: Goss-gamma}
The original {\it Goss gamma function} is equal to (see \cite[Theorem~4.2.6]{Goss88}):
\[
\Gamma_{\text{Goss}}(x,y)\assign \Pi_{\geo}(-x,-y)\cdot \Gamma_{\ari}(y).
\]
From the reflection formula (cf.\ \cite[Theorem~1.4 and Lemma~2.3]{Thakur91}):
\[
\Gamma_{\ari}(y) \Gamma_{\ari}(1-y) =\prod_{i=1}^{\infty} \biggl(
1-\frac{\theta}{\theta^{q^{i}}} \biggr)^{-1}
= (-\theta)^{\frac{-q}{q-1}}\cdot \tilde{\pi},
\]
we get that
\[
\Gamma_{\text{Goss}}(x,y) = (-\theta)^{\frac{-q}{q-1}}\cdot \tilde{\pi} \cdot  \Pi(-x,-y)
= (-\theta)^{\frac{-q}{q-1}}\cdot \tilde{\pi} \cdot (-x)\cdot \Gamma(-x,1-y).
\]
\end{remark}

We list the monomial relations among gamma values in the following (see \cite{Thakur91} for arithmetic gamma values and \cite{Goss88} for two-variable gamma values):

\begin{proposition}\label{prop: AR-Gamma}
${}$
\begin{itemize}
    \item[(1)]
    For $x \in \CC_\infty\setminus (-A_+\cup\{0\})$ and $y \in \ZZ_p$,
    \[
    \Gamma_{\ari}(y)\Gamma_{\ari}(1-y) \sim \tilde{\pi}
    \quad \text{and}\quad
    \Gamma(x,y)\cdot\Gamma(x,1-y) \sim \frac{\Gamma_{\geo}(x)^{q-1}}{\tilde{\pi}}.
    \]
    \item[(2)]
    For $x \in k\setminus (-A_+\cup\{0\})$, $y \in \ZZ_p$, $a \in A \setminus (-A_+\cup \{0\})$ and $N \in \NN$,
    \[
    \Gamma_{\geo}(a,y) = \Gamma(a,y)\cdot \Gamma_{\ari}(y) \sim 
    \Gamma_{\ari}(N) \sim 
    \Gamma(x,N) \sim 1.
    \]
    \item[(3)] Given $x \in k \setminus A$, $y \in \ZZ_{(p)}\setminus \ZZ$, $a \in A \setminus (-A_+\cup \{0\})$ and $N \in \NN$,
    \[
    \Gamma(x+a,y+N)\sim \Gamma(x,y)
    \quad\text{ and }\quad
    \Gamma_{\ari}(y+N) \sim \Gamma_{\ari}(y).
    \]
    \item[(4)]
    Given $x \in k\setminus (-A_+\cup\{0\})$ and $y \in \frac{1}{q^\ell-1}\ZZ \setminus \ZZ$,
    write $\langle -y\rangle_{\ari} = \sum_{i=0}^{\ell-1} y_i q^i/(q^\ell-1)$ where $y_0,...,y_{\ell-1}$ are integers with $0\leq y_0,...,y_{\ell-1}<q$.
    Then
    \[
    \Gamma(x,y) \sim \prod_{i=0}^{\ell-1} \Gamma(x,1-\frac{q^i}{q^\ell-1})^{y_i}
    \quad \text{ and } \quad 
    \Gamma_{\ari}(y) \sim \prod_{i=0}^{\ell-1} \Gamma_{\ari}(1-\frac{q^i}{q^\ell-1})^{y_i}.
    \]
        \item[(5)]
    Given $x \in k\setminus A$ and $y \in \frac{1}{q^\ell-1}\ZZ \setminus \ZZ$,
    \[
    \prod_{c=0}^{\ell-1}\prod_{\epsilon \in \FF_q^\times}\Gamma(\epsilon x,q^c y) \sim 1
    \quad \text{ and } \quad
    \prod_{c=0}^{\ell-1}\Gamma_{\ari}(q^c y) \sim \tilde{\pi}^{\ell \text{\rm wt}^{\color{blue}\ari}(y)/(q-1)}.
    \]
    \item[(6)]
    For $x \in k\setminus A$, $y\in \ZZ_p$, and
    $\nfk \in A_+$,
    \[
    \prod_{\subfrac{a \in A}{\deg a < \deg \nfk}}\Gamma(\frac{x+a}{\nfk},y)
    \sim \Gamma(x,|\nfk|_\infty y).
    \]
\end{itemize}
\end{proposition}

\begin{remark}\label{rem: AR-geoGamma}
Since 
\[\Gamma_{\geo}(x) = \Gamma(x,1-\frac{1}{q-1})\cdot \Gamma_{\ari}(1-\frac{1}{q-1}), \quad \forall x \in \CC_\infty \setminus (-A_+\cup \{0\}),
\]
we have the following relations for geometric gamma values (see also \cite{ABP}):
\begin{itemize}
    \item[(1)]
    \[
    \Gamma_{\geo}(a) \sim 1, \quad \forall a \in A\setminus (-A_+\cup \{0\})
    \quad \text{ and } \quad 
    \Gamma_{\geo}(x+a) \sim \Gamma_{\geo}(x);
    \]
    \item[(2)]
    Given $x \in k\setminus A$ and $\nfk \in A_+$,
    \[
    \prod_{\epsilon \in \FF_q^\times} \Gamma_{\geo}(\epsilon x) \sim \tilde{\pi}
    \quad \text{ and } \quad 
    \prod_{\subfrac{a \in A}{\deg a < \deg \nfk}}
    \Gamma_{\geo}(\frac{x+a}{\nfk}) \sim \Gamma_{\geo}(x)\cdot \tilde{\pi}^{\frac{|\nfk|_\infty-1}{q-1}}.
    \]
\end{itemize}
\end{remark}

For each $x \in k_\infty$, put $\{x\} \in k_\infty$ to be the fractional part of $x$, i.e.\ $|\{x\}|_\infty <1$ and $x \equiv \{x\} \bmod A$.
We may view $\{\cdot \}$ as a function on $k_\infty/A$.
Let $\nfk \in \eA_+$ and $\ell \in \NN$.
The formulas for the evaluations of $\psi_{x},\psi_{\ari,c}, \psi_{x,c} \in \OO_{\nfk,\ell,(0)}^\dagger$ at $\xi_{1,0}$
in Proposition~\ref{prop: key-equality}
for every $x \in \frac{1}{\nfk(\theta)}A/A$ and $0\leq c < \ell$
give us the following key identities:

\begin{proposition}\label{prop: psi-factorial eq}
Let $\nfk \in \eA_+$ and $\ell \in \NN$.
Given $x \in \frac{1}{\nfk(\theta)}A/A$ and $y \in \frac{1}{q^\ell-1}\ZZ/\ZZ$,
we have that
\[
\psi_{x}(\xi_{1,0}) = \Pi_{\geo}(\{x\})^{q^\ell-1},
\quad 
\psi_{y}(\xi_{1,0})
=\Pi_{\text{\rm ari}}(-\langle y\rangle_{\ari})^{q^\ell-1},
\]
and
\[
\psi_{(x,y)}(\xi_{1,0}) = \Pi_{\geo}(\{x\},-\langle y\rangle_{\ari})^{q^\ell-1},
\quad 
\left(\frac{\psi_{(x,y)}(\xi_{1,0})}{ \psi_{y}(\xi_{1,0})}\right)
= \Pi(\{x\},-\langle y\rangle_{\ari})^{q^\ell-1}.
\]
\end{proposition}

Finally, define $\tilde{\Gamma}_{\ari}:\ZZ_{(p)}/\ZZ \rightarrow \CC_\infty^\times$,
$\tilde{\Gamma}_{\geo}:k/A\rightarrow \CC_\infty^\times$,
and
$\tilde{\Gamma}: (k/A)\times (\ZZ_{(p)}/\ZZ)\rightarrow \CC_\infty^\times$ as follows:
for $x \in k/A$ and $y \in \ZZ_{(p)}/\ZZ$,
\[
\tilde{\Gamma}_{\ari}(y) \assign \Gamma_{\ari}(1-\langle -y \rangle_{\ari}) ,
\quad 
\tilde{\Gamma}(x,y) \assign
\begin{cases}
\Gamma(\{x\},1-\langle -y\rangle_{\ari})  , & \text{ if $x \neq 0 \in k/A$;} \\
\tilde{\Gamma}_{\ari}(y)^{-1}, & \text{ otherwise,}
\end{cases}
\]
and
\[
\tilde{\Gamma}_{\geo}(x) \assign
\tilde{\Gamma}(\{x\},\frac{1}{1-q}) = \Gamma_{\ari}(1- \frac{1}{q-1})^{-1}\cdot  
\begin{cases}
\Gamma_{\geo}(\{x\}), &\text{ if $x \neq 0 \in k/A$;}\\
1,&\text{ otherwise.}
\end{cases}
\]
We may extend $\tilde{\Gamma}_{\ari}$, $\tilde{\Gamma}_{\geo}$, and $\tilde{\Gamma}$
to $\QQ$-linear homomorphisms from $\Acal^{\ari}_\QQ$, $\Acal^{\geo}_\QQ$, and $\Acal^{\cyc}_\QQ$ to $\CC_\infty^\times$, respectively, and their compositions with the quotient map $\CC_\infty^\times \rightarrow \CC_\infty^\times/\bar{k}^\times$ are denoted by $\hat{\Gamma}_{\ari}$, $\hat{\Gamma}_{\geo}$, and $\hat{\Gamma}$, respectively.
The above proposition
leads us to the following analogue of Deligne's theorem on gamma distributions (stated in \cite[Theorem~4.7]{Anderson82}):

\begin{theorem}\label{thm: Gamma dist}
\[
\hat{\Gamma}_{\text{\rm ari}}
= \Pscr^{\cyc}_{\nu_1} \circ \St^{\ari},
\quad 
\hat{\Gamma}_{\text{\rm geo}}
= \Pscr^{\cyc}_{\nu_1} \circ \St^{\geo},
\quad
\text{ and } \quad 
\hat{\Gamma}
= \Pscr^{\cyc}_{\nu_1} \circ \text{\rm St}.
\]
\end{theorem}

\begin{proof}
Let $\nfk \in \eA_+$ and $\ell \in \NN$.
By Lemma~\ref{lem: ST}, we have that for $x \in \frac{1}{\nfk(\theta)}A/A$ and $y \in \frac{1}{q^\ell-1}\ZZ/\ZZ$,
\begin{align*}
\Pscr_{\nu_1}^{\cyc}\circ \St^{\ari}(y)^{q^\ell-1}
& = \Pcal_{\eK_{\nfk,\ell}}(\xi_{1,0},\Xi_{\ari,-y}) = \psi_{\ari,-y}(\xi_{1,0}) \cdot \bar{k}^\times \quad \in \CC_\infty^\times/\bar{k}^\times, \\
\Pscr_{\nu_1}^{\cyc}\circ \St^{\geo}(x)^{q^\ell-1}
& = \Pcal_{\eK_{\nfk,\ell}}(\xi_{1,0},\Phi_x) = \left(\frac{\psi_{\geo,x}(\xi_{1,0})}{\psi_{\ari,1/(q-1)}(\xi_{1,0})}\right) \cdot \bar{k}^\times \quad \in \CC_\infty^\times/\bar{k}^\times, \\
\Pscr_{\nu_1}^{\cyc}\circ \St(x,y)^{q^\ell-1}
& = \Pcal_{\eK_{\nfk,\ell}}(\xi_{1,0},\Phi_{(x,y)}) = \left(\frac{\psi_{x,-y}(\xi_{1,0})}{\psi_{\ari,-y}(\xi_{1,0})}\right) \cdot \bar{k}^\times \quad \in \CC_\infty^\times/\bar{k}^\times.
\end{align*}
Comparing with the definition of $\tilde{\Gamma}_{\ari}$, $\tilde{\Gamma}_{\geo}$, and $\tilde{\Gamma}$,
the result thereby follows from Proposition~\ref{prop: psi-factorial eq}.
\end{proof}

\begin{remark}\label{rem: Gamma-dist}
For each $\varphi \in \Sscr(\eG^{\cyc})$, we let $\wp_{\nu_1}^{\cyc}(\varphi) \in \CC_\infty^\times$ be an arbitrary representative of $\Pscr_{\nu_1}^{\cyc}(\varphi)$.
Then Theorem~\ref{thm: Gamma dist} says that for $\boy \in \Acal_\QQ^{\ari}$, $\boxx \in \Acal_\QQ^{\geo}$, and $\ba \in \Acal_\QQ^{\cyc}$,
\[
\tilde{\Gamma}_{\ari}(\boy) \sim \wp_{\nu_1}^{\cyc}\big(\St^{\ari}(\boy)\big),\quad 
\tilde{\Gamma}_{\geo}(\boxx) \sim \wp_{\nu_1}^{\cyc}\big(\St^{\geo}(\boxx)\big), 
\quad\text{and}\quad 
\tilde{\Gamma}(\ba)  \sim \wp_{\nu_1}^{\cyc}\big(\St(\ba)\big).
\]
\end{remark}

\begin{theorem}\label{thm: Lang-Rohrlich-conj}
Let $\nfk \in \eA_+$ and $\ell \in \NN$.
We have:
\begin{align*}
& \trdeg_{\bar{k}} \bar{k}\Big(\Gamma_{\geo}(x),\Gamma_{\ari}(y),\ \Gamma(x,y)\ \Big|\ x \in \frac{1}{\nfk(\theta)}A \setminus (-A_+ \cup \{0\}),\ y \in \frac{1}{q^\ell-1}\ZZ\Big) \\
=& 
1+ (\ell-\frac{1}{(q-1)^{\epsilon_\nfk}}) \cdot \#(A/\nfk)^\times,
\end{align*}
where $\epsilon_\nfk = 1$ if $\deg \nfk >0$ and $0$ otherwise.
\end{theorem}

\begin{proof}
Notice that for $y \in \ZZ_{(p)}$ and $0\neq x \in k/A$,
\[
\tilde{\Gamma}(0,y) = \Gamma_{\ari}(1-\langle -y\rangle_{\ari})^{-1} 
\quad \text{ and } \quad 
\tilde{\Gamma}(x,\frac{1}{1-q})
=\Gamma_{\geo}(\{x\}) \cdot \Gamma_{\ari}(1-\frac{1}{q-1})^{-1}.
\]
Thus by Proposition~\ref{prop: AR-Gamma}~(1)--(3) we get
\begin{align*}
& \ \bar{k}\Big(\Gamma_{\geo}(x),\Gamma_{\ari}(y), \Gamma(x,y)\ \Big|\ x \in \frac{1}{\nfk(\theta)}A \setminus (-A_+ \cup \{0\}),\ y \in \frac{1}{q^\ell-1}\ZZ\Big) \\
 =&\ 
\bar{k}\Big(\tilde{\Gamma}(x,y)\ \Big|\ x \in \frac{1}{\nfk(\theta)}A/A,\ y \in \frac{1}{q^\ell-1}\ZZ/\ZZ\Big).
\end{align*}
From Proposition~\ref{prop: St-span}~(3), Theorem~\ref{thm: Gamma dist}, and Remark~\ref{rem: Gamma-dist}, we know that the field 
$\bar{k}\big(\wp_{\nu_1}^{\cyc}(\varphi)\mid \varphi \in \Sscr(\eG_{\nfk,\ell})\big)$
is algebraic over $\bar{k}\Big(\tilde{\Gamma}(x,y)\ \Big|\ x \in \frac{1}{\nfk(\theta)}A/A,\ y \in \frac{1}{q^\ell-1}\ZZ/\ZZ\Big)$.
Hence
\begin{align*}
& \trdeg_{\bar{k}}\bar{k}\Big(\tilde{\Gamma}(x,y)\ \Big|\ x \in \frac{1}{\nfk(\theta)}A/A,\ y \in \frac{1}{q^\ell-1}\ZZ/\ZZ\Big) \\
= & \trdeg_{\bar{k}} \bar{k}\big(\wp_{\nu_1}^{\cyc}(\varphi)\mid \varphi \in \Sscr(\eG_{\nfk,\ell})\big) \\
=&1+ (\ell-\frac{1}{(q-1)^{\epsilon_\nfk}}) \cdot \#(A/\nfk)^\times \quad \text{(by Theorem~\ref{thm: trdeg})}.
\end{align*}
\end{proof}

Let $\nfk \in \eA_+$ and $\ell \in \NN$.
By \eqref{eqn: Ucyc-dim} we have 
\[
1+(\ell-\frac{1}{(q-1)^{\epsilon_\nfk}}) \cdot \#(A/\nfk)^\times
=
\dim_\QQ(\Ucal^{\cyc}_{\nfk,\ell})
\] 
where $\Ucal^{\cyc}_{\nfk,\ell}$ is defined in Section~\ref{sec: Uni-dis} III.
As $\nfk$ and $\ell$ are chosen arbitrarily, we arrive at:

\begin{corollary}\label{cor: AR-Gamma}
All algebraic relations among gamma values $\Gamma_{\geo}(x)$, $\Gamma_{\ari}(y)$, and 
$\Gamma(x,y)$ for $x \in \ek\setminus (-A_+\cup \{0\})$ and $y \in \ZZ_{(p)}$ are explained by the monomial relations listed in {\rm Proposition~\ref{prop: AR-Gamma}}.
\end{corollary}

\begin{remark}\label{rem: comparison}
${}$
\begin{itemize}
    \item[(1)] Taking $\ell = 1$ in the above equalities, we re-prove the Lang-Rohrlich conjecture for geometric gamma values in \cite[Corollary~1.2.2]{ABP}:
    \[
    \trdeg_{\bar{k}} \bar{k}\Big(\tilde{\pi},\Gamma_{\geo}(x)\ \Big|\ x \in \frac{1}{\nfk(\theta)}A \setminus (-A_+ \cup \{0\})\Big)
    = 1+(1-\frac{1}{(q-1)^{\epsilon_\nfk}}) \cdot (\eA/\nfk)^\times,
    \]
    where $\epsilon_\nfk = 1$ if $\deg \nfk >0$ and $0$ otherwise.
    \item[(2)] Taking $\nfk =1$ in the above equalities, we re-prove \cite[Corollary~3.3.3]{CPTY10}:
    \[
    \trdeg_{\bar{k}} \bar{k}\Big(\Gamma_{\ari}(y)\ \Big|\ y \in \frac{1}{q^\ell-1}\ZZ\Big) = \ell.
    \]
    \item[(3)] 
    Since
    \[
    \tilde{\Gamma}_{\geo}(x) \sim \tilde{\Gamma}(x,\frac{1}{1-q})
    = \prod_{i=0}^{\ell-1}\tilde{\Gamma}(x,\frac{q^i}{1-q^\ell}),
    \]
    We get that
    \begin{align*}
    &\ \trdeg_{\bar{k}} \bar{k}\Big(\Gamma_{\geo}(x),\Gamma_{\ari}(y)\ \Big|\ x \in \frac{1}{\nfk(\theta)}A \setminus (-A_+ \cup \{0\}),\ \ y \in \frac{1}{q^\ell-1}\ZZ\Big) \\
    =&\
    \ell + (1-\frac{1}{(q-1)^{\epsilon_\nfk}})\cdot (\eA/\nfk)^\times,
    \end{align*}
    and so all algebraic relations among arithmetic and geometric gamma values come from their relations with the Carlitz period $\tilde{\pi}$.
\end{itemize}
\end{remark}

\section{The Chowla--Selberg phenomenon}

 Theorem~\ref{thm: Gamma dist} enables us to connect the ``CM periods'', when the CM field is contained in a cyclotomic function field, with special gamma values.
In this section, we shall first recall the relationship between the periods of dual $t$-motives and the periods of ``abelian $t$-modules'', and then apply Theorem~\ref{thm: Gamma dist} to derive an analogue of the Chowla--Selberg formula.

\subsection{\texorpdfstring{Quasi-periods of CM abelian $t$-modules}{Quasi-periods of CM abelian t-modules}}

An \emph{abelian $t$-module of dimension} $d>0$ over $\bar{k}$ is a pair $(E,\rho)\rassign E_\rho$ with the following properties.
\begin{itemize}
\item $E$ is a $d$-dimensional additive group ${\GG_{a}^{d}}_{/\bar{k}}$ defined over $\ok$.
\item $\rho:\eA\rightarrow \End_{\FF_{q}} ({\GG_{a}^{d}}_{/\ok})$ is an $\FF_{q}$-algebra homomorphism so that $ \partial\rho_{t}-\theta I_{d}$ acts nilpotently on $\Lie (\GG_{a}^{d})$, where $\partial \rho_{t}$ denotes the  differential of the homomorphism $\rho_{t}:{\GG_{a}^{d}}_{/\ok}\rightarrow {\GG_{a}^{d}}_{/\ok}$.
\item $\Mscr(\rho) \assign \Hom_{\FF_q}(E,\GG_{a})$ is finitely generated over $\ok[t]$ with respect to the $\eA$-module structure via $\rho$.
\end{itemize}
It is known that $\Mscr(\rho)$ must be free of finite rank over $\ok[t]$, and $\rank_{\ok[t]}\Mscr(\rho)$ is called the \emph{rank of $E_\rho$}.
Set 
\begin{equation}\label{eqn: M(rho)}
\eM(\rho)\assign \Hom_{\bar{k}[t]}(\tau\Mscr(\rho),\bar{k}[t]dt)
\end{equation}
which is constructed by Hartl-Juschka \cite{HJ20},
where $\tau$ is identified with the  $q$-power Frobenius endomorphism on $\GG_a$.
It is known that $\eM(\rho)$ is a dual $t$-motive (and $E_\rho$ is always ``$A$-finite'', see \cite[Theorem 2.3.2]{BCPW22}).
We call $\eM(\rho)$ the \emph{Hartl-Juschka dual $t$-motive associated with $E_\rho$}.

Suppose that $E_\rho$ is uniformizable,
in which case so is $\eM(\rho)$, see \cite[Theorem~2.4.3]{BCPW22}.
Let $\Lambda_\rho$ be the period lattice of $E_\rho$ and let $H_{\mathrm{dR}}(E_\rho,\bar{k})$ be the de Rham module of $E_\rho$ (see \cite[\S 3.1]{BP02} and \cite[Proposition 4.1.3]{NP21}).
Let $[\cdot,\cdot]:H_{\mathrm{dR}}(E_\rho,\bar{k}) \times \Lambda_\rho \rightarrow \CC_\infty$ be the de Rham pairing introduced in \cite[(4.3.3) and (4.3.4)]{NP21}.
We call $[\delta,\lambda]$ for $\delta \in H_{\mathrm{dR}}(E_\rho,\bar{k})$
and $\lambda \in \Lambda_\rho$ a \emph{quasi-period of $E_\rho$}.
Recall that we have the following natural comparison between the associated de Rham pairings
(cf.\ \cite[Proposition~8.3.4]{BCPW22}):
\begin{equation}\label{eqn: dR-P comp}
\SelectTips{cm}{}
\xymatrixrowsep{0.4cm}
\xymatrixcolsep{0.05cm}
\xymatrix{
[\ ,\ ]: & H_{\mathrm{dR}}(E_\rho,\bar{k})\ar@{<->}[dd]^{\rotatebox{90}{\scalebox{1}[1]{$\sim$}}} & \times & \Lambda_\rho\ar@{<->}[dd]^{\rotatebox{90}{\scalebox{1}[1]{$\sim$}}}  \ar[rrr] & & & \CC_\infty \ar@{<->}[dd]^{\rotatebox{90}{\scalebox{1}[1]{$\sim$}}} & z \ar@{<->}[dd] \\
&&&& \\
\int : & H_{\mathrm{dR}}(\eM(\rho),\bar{k}) & \times & H_{\mathrm{Betti}}(\eM(\rho)) \ar[rrr] & & & \CC_\infty & -z\\
}
\end{equation}

We say that $E_\rho$ is a \emph{CM abelian $t$-module with generalized CM type $(\eK,\Xi)$} if $\eM(\rho)$ is a CM dual $t$-motive with generalized CM type $(\eK,\Xi)$.
In this case, the above comparison of de Rham pairings implies that the space of quasi-periods of $E_\rho$ coincides with (cf.~\cite[Thm.~8.4.1]{BCPW22})
\[
\sum_{\xi \in J_\eK}\bar{k}\cdot p_\eK(\xi,\Xi).
\]
In particular, suppose $\eK \subset \eK_{\nfk,\ell}$ for $\nfk \in \eA_+$ and $\ell \in \NN$.
Let $\varphi_{\eK,\Xi} \in \Sscr(\eG^{\cyc})$ be the Stickelberger function associated to the CM type $(\eK,\Xi)$ given in \eqref{eqn: phi-K-Phi}.
Then Lemma~\ref{lem: Stcyc-surj} guarantees that there exists $\ba \in \Acal^{\cyc}_{\nfk,\ell,\QQ}$ such that
\[
\varphi_{\eK,\Xi} = \St(\ba) \quad \in \Sscr(\eG^{\cyc}).
\]
In particular, consider $\xi_0 \in J_\eK$ whose corresponding embedding $\nu_{\xi_0}$ is $\nu_1\big|_{\eK}$ (where $\nu_1:\eK^{\cyc}\hookrightarrow \CC$ is chosen in \eqref{eqn: nu_1}).
As $\eK$ is Galois over $\ek$, one has
$J_\eK = \{\xi_0^{\varrho}\mid \varrho \in \eG^{\cyc}/\eH_\eK\}$. 
By Theorem~\ref{thm: PS}~(2) the space of quasi-periods of $E_\rho$ becomes
\begin{equation}\label{eqn: CM-qp}
\sum_{\varrho \in \eG_\eK} \ok \cdot p_\eK(\xi_0^{\varrho},\Xi) = 
\sum_{\varrho \in \eG_\eK} \ok \cdot p_\eK(\xi_0,\Xi^{\varrho^{-1}}) = \sum_{\varrho \in \eG_\eK} \ok \cdot \wp_{\nu_1}^{\cyc}(\varphi_{\eK,\Xi^{\varrho^{-1}}}).
\end{equation}
Since $\varphi_{\eK,\Xi^{\varrho^{-1}}} = \varrho \cdot \varphi_{\eK,\Xi} = \St(\varrho \star \ba)$ for every $\varrho \in \eG^{\cyc}$, by Theorem~\ref{thm: Gamma dist} (and Remark~\ref{rem: Gamma-dist})
we obtain 

\begin{theorem}\label{thm: qp-CM}
Let $E_\rho$ be a CM abelian $t$-module with generalized CM type $(\eK,\Xi)$ over $\bar{k}$, where $\eK$ is contained in $\eK_{\nfk,\ell}$ for $\nfk \in \eA_+$ and $\ell \in \NN$.
Let $\eH_\eK = \gal(\eG^{\cyc}/\eK)$.
There exists $\ba \in \Acal^{\cyc}_{\nfk,\ell,\QQ}$ such that the space of quasi-periods of $E_\rho$ is spanned over $\bar{k}$ by
\[
\tilde{\Gamma}(\varrho \star \ba) \ (\sim \wp_{\nu_1}^{\cyc}(\varrho \cdot \varphi_{\eK,\Xi})), \quad  \varrho \eH_\eK \in \eG^{\cyc}/\eH_\eK.
\]
\end{theorem}

\subsection{\texorpdfstring{Period vectors of CM Hilbert--Blumenthal $t$-modules}{Period vectors of CM Hilbert--Blumenthal t-modules}}\label{sec: pv-CM}

Let $\eK$ be a CM field with maximal totally real subfield denoted by $\eK^+$, and let
$\Xi$ be a CM type of $\eK$.
Write $\Xi = \xi_1+\cdots+ \xi_d$, where $d = [\eK^+:\ek]$.
Let $E_\rho$ be a CM abelian $t$-module with CM type $(\eK,\Xi)$ over $\bar{k}$.
By \cite[Section~8.5]{BCPW22}, $E_\rho$ is actually a Hilbert-Blumenthal $O_{\eK^+}$-module with CM by $O_{\eK}$.
Put $\eM = \eM(\rho)$.
Let $\omega_{\eM,\xi_1},...,\omega_{\eM,\xi_d}$ be the differentials of $\eM$ associated with $\xi_1,...,\xi_d$, respectively.
From the natural $\bar{k}$-linear isomorphism $\Lie(E_\rho) \cong \eM/\sigma \eM$ (see \cite[(8.5.1)]{BCPW22}), 
we get a $\bar{k}$-linear isomorphism $\Lie(E_\rho) \cong \bar{k}^d$ with respect to $\omega_{\eM,\xi_1},...,\omega_{\eM,\xi_d}$.
Moreover, the image of every period vector $\lambda \in \Lambda_\rho \subset \Lie(E_\rho)(\CC_\infty)$ in $\CC_\infty^d$ under this isomorphism is
\[
(\int_{\gamma_\lambda} \omega_{\eM,\xi_1},...,\int_{\gamma_\lambda}\omega_{\eM,\xi_1}) \quad \in \CC_\infty^d,
\]
where $\gamma_\lambda \in H_{\rm Betti}(\eM)$ is the cycle corresponding to $\lambda$ via the natural isomorphism in the comparison diagram~\eqref{eqn: dR-P comp}.

Suppose $\eK$ is contained in $\eK_{\nfk,\ell}$ for $\nfk \in \eA_+$ and $\ell \in \NN$.
By Lemma~\ref{lem: Stcyc-surj}, there exists $\ba \in \Acal^{\cyc}_{\nfk,\ell,\QQ}$ so that 
$$ 
\varphi_{\eK,\Xi} = \St(\ba) \quad \in \Sscr(\eG^{\rm cyc}).
$$
Let $\xi_0 \in J_\eK$ with $\nu_{\xi_0} = \nu_1\big|_{\eK}$.
For $1\leq i \leq d$, take $\varrho_i \in \eG^{\cyc}$ so that $\nu_1 \circ \varrho_i \big|_\eK = \nu_{\xi_i}$.
Write $\Lambda_\rho = \Ifk_{\rho}\cdot \lambda_0$ where $\Ifk_{\rho}$ is an ideal of $O_\eK$.
By Theorem~\ref{thm: Gamma dist} and \ref{thm: qp-CM} we have that
\begin{equation}\label{eqn: HB-period}
\int_{\gamma_{\lambda_0}} \omega_{\eM,\xi_i}
\sim p_\eK(\xi_i, \Xi) 
\sim p_{\eK}(\xi_0,\Xi^{\varrho_i^{-1}}) 
\sim \wp^{\cyc}_{\nu_1}\big(\varphi_{\eK,\Xi^{\varrho_i^{-1}}}\big)
\sim \wp^{\cyc}_{\nu_1}\big(\varrho_i \cdot \varphi_{\eK,\Xi}\big) 
\sim \tilde{\Gamma}(\varrho_i \star \ba).
\end{equation}
Therefore:

\begin{theorem}\label{thm: HB-period}
Let 
$\eK$ be a CM field whose maximal totally real subfield is denoted by $\eK^+$.
Suppose $\eK \subset \eK_{\nfk,\ell}$ for $\nfk \in \eA_+$ and $\ell \in \NN$.
Put $d = [\eK^+:\ek]$.
Let $\Xi = \xi_1+\cdots+\xi_d$ be a CM type of $\eK$.
For $1\leq i \leq d$, take $\varrho_i \in \eG^{\cyc}$ so that $\nu_{1}\circ \varrho_i\big|_\eK = \nu_{\xi_i}$.
Let $E_\rho$ be a CM abelian $t$-module with CM type $(\eK,\Xi)$ over $\bar{k}$.
There exist $\mathbf{a} \in \Acal^{\cyc}_{\nfk,\ell,\QQ}$, an ideal $\Ifk_\rho$ of $O_\eK$, and a suitable $\bar{k}$-linear isomorphism $\Lie(E_\rho) \cong \bar{k}^d$ so that
image of the period lattice $\Lambda_\rho \subset \Lie(E_\rho)(\CC_\infty)$ in $\CC_\infty^d$ is
$$
\left\{ 
\Big(\nu_{\xi_1}(a)\tilde{\Gamma}(\varrho_1\star \ba),...,\nu_{\xi_d}(a)\tilde{\Gamma}(\varrho_d \star \ba)\Big)\ \bigg|\ a \in \Ifk_\rho
\right\}.
$$
\end{theorem}

Now, let $\eK$ be an \emph{imaginary field} over $\ek$ (i.e.\ the infinite place $\infty$ of $\ek$ does not split in $\eK$).
Let $E_\rho$ be a Drinfeld $\eA$-module of rank $[\eK:\ek]$ over $\ok$ which has CM by $O_\eK$ (the integral closure of $\eA$ in $\eK$).
Extending $\rho$ to an $\FF_q$-algebra homomorphism (still denoted by)
$\rho: O_\eK \rightarrow \End_{\FF_q}(E_\rho)$, we have that the differential $\partial \rho : O_\eK \rightarrow \End_{\ok}(\Lie(E_\rho))$ induces an embedding $\nu_\rho: \eK \hookrightarrow \ok\subset \CC_\infty$.
Let $\xi_\rho \in J_\eK$ be the point corresponding to the embedding $\nu_\rho$.
We may regard $E_\rho$ as a CM abelian $t$-module with CM type $(\eK,\xi_\rho)$.

Suppose that $\eK\subset \eK^{\cyc}$ (which is equivalent to saying that $\eK/\ek$ is an abelian extension and the ramification index of $\infty$ in $\eK$ is divided by $q-1$ from class field theory).
By Theorem~\ref{thm: HB-period}, 
We are able to establish the following result which in particular verifies Thakur's recipe/conjecture on the Chowla--Selberg phenomenon in
\cite{Thakur91}:

\begin{theorem}\label{thm: CSP}
Let $E_\rho$ be a Drinfeld $\eA$-module over $\ok$ with CM by $O_\eK$, where $O_\eK$ is the integral closure of $\eA$ in an imaginary function field $\eK$ with $\eK\subset \eK_{\nfk,\ell}$ for $\nfk \in \eA_+$ and $\ell \in \NN$.
Let $\eH_\eK = \gal(\eK^{\cyc}/\eK)$.
Take any $\ba \in \Acal^{\cyc}_{\nfk,\ell,\QQ}$ so that $\mathbf{1}_{\eH_\eK} = \St(\ba) \in \Sscr(\eG^{\cyc})$.
Then
\[
\lambda \sim \wp^{\cyc}_{\nu_1}(\mathbf{1}_{\eH_\eK}) \sim \tilde{\Gamma}(\ba), \quad \text{ for every nonzero period } \lambda \in \Lambda_\rho \subset \CC_\infty.
\]
Moreover, 
\[\tilde{\Gamma}(\varrho\star \ba) \ (\sim \wp^{\cyc}_{\nu_1}(\varrho \cdot \mathbf{1}_{\eH_\eK})), \quad \varrho \eH_\eK \in \eG^{\cyc}/\eH_\eK
\]
are algebraically independent over $\bar{k}$ and
form a $\bar{k}$-basis of the space of quasi-periods of the CM Drinfeld module $E_\rho$.
\end{theorem}

\begin{proof}
As $\eK$ is imaginary,
we have that $\eK^+=\ek$ and the characteristic function $\mathbf{1}_{\eH_\eK}$ of $\eH_\eK$ is contained in $\Sscr(\eG_{\nfk,\ell})$. Hence
Lemma~\ref{lem: Stcyc-surj} assures us that there exist $\ba \in \Acal^{\cyc}_{\nfk,\ell,\QQ}$ so that 
\[
\mathbf{1}_{\eH_\eK} = \St(\ba).
\]
Let $\xi_0 \in J_\eK$ such that $\nu_{\xi_0} = \nu_1\big|_{\eK}$.
Then Theorem~\ref{thm: HB-period} implies that
\[
\lambda \sim  p_\eK(\xi_0,\xi_0)\sim \tilde{\Gamma}(\ba), \quad \text{ for every nonzero period } \lambda \in \Lambda_\rho \subset \CC_\infty.
\]
Moreover, Theorem~\ref{thm: qp-CM} shows that the space of quasi-periods of $E_\rho$ is spanned over $\ok$ by
\[
\tilde{\Gamma}(\varrho\star \ba), \quad \varrho \eH_\eK \in \eG^{\cyc}/\eH_\eK,
\]
and 
\begin{align*}
\trdeg_{\ok}\ok(\tilde{\Gamma}(\varrho\star \ba)\mid \varrho \eH_\eK \in \eG^{\cyc}/\eH_\eK)
&= \trdeg_{\ok}\ok\Big(\wp_{\nu_1}^{\cyc}\big(\Sscr(\eG_{\eK})\big)\Big) \\
&= [\eK:\ek] \quad \text{ (by Proposition~\ref{thm: trdeg}).}
\end{align*}
\end{proof}

Next, we shall provide explicit constructions for the $\ba$'s in Theorem~\ref{thm: qp-CM}, Theorem~\ref{thm: HB-period}, and Theorem~\ref{thm: CSP},
and present a function field analogue of the Chowla--Selberg formula.

\subsection{The Chowla--Selberg formula}\label{sec: CSF}

Let $\nfk \in \eA_+$ and $\ell \in \NN$.
Identifying $\eG_{\nfk,\ell}$ with $(\eA/\nfk)^\times \times (\ZZ/\ell \ZZ)$ via the Artin map as in \eqref{eqn: cyc-Artin},
every $\bochi \in \widehat{\eG}_{\nfk,\ell}$ can also be regarded as a character on $(\eA/\nfk)^\times \times (\ZZ/\ell \ZZ)$.
We put $\cfk_{\bochi} = \cfk_{\bochi_f}$, where $\bochi_f = \bochi\big|_{(\eA/\nfk)^\times}$.
Recall that for $c \in \ZZ/\ell \ZZ$ and $\alpha \in (\eA/\cfk_{\bochi})^\times$,
by Lemma~\ref{lem: evaluator}
we have
\begin{align}\label{eqn: generating-chi}
\left(\St(\frac{\alpha(\theta)}{c_{\bochi}(\theta)},\frac{q^c}{1-q^\ell})\ \bigg|\ \bochi\right)_{\cfk_{\bochi},\ell}
&= \sum_{(a,i) \in (\eA/\cfk_{\bochi})^\times \times (\ZZ/\ell\ZZ)}\St(\frac{\alpha(\theta)}{c_{\bochi}(\theta)},\frac{q^c}{1-q^\ell})(a,i) \cdot \overline{\bochi(a,i)} \nonumber \\
& = \bochi(\alpha,c+\deg \cfk_{\bochi}) \cdot L_\eA(0,\overline{\bochi}).
\end{align}
Therefore:

\begin{lemma}\label{lem: ST-chi}
Let $\nfk \in \eA_+$ and $\ell \in \NN$.
For each character $\bochi \in \widehat{\eG}_{\nfk,\ell}$, the following equality holds:
\[
\sum_{(a,i) \in (\eA/\cfk_{\bochi})^\times \times (\ZZ/\ell \ZZ)}
\overline{\bochi(a,i+\deg \cfk_{\bochi})} \cdot \St(\frac{a(\theta)}{\cfk_{\bochi}(\theta)},\frac{q^i}{1-q^\ell})
= L_\eA(0,\overline{\bochi}) \cdot \bochi.
\]
\end{lemma}

\begin{proof}
For every $(\alpha,c) \in (\eA/\nfk)^\times \times (\ZZ/\ell \ZZ)$, we get
\begin{align*}
& \sum_{(a,i) \in (\eA/\cfk_{\bochi})^\times \times (\ZZ/\ell \ZZ)}
\overline{\bochi(a,i+\deg \cfk_{\bochi})} \cdot \St(\frac{a(\theta)}{\cfk_{\bochi}(\theta)},\frac{q^i}{1-q^\ell})(\alpha,c)\\
= & \sum_{(a,i) \in (\eA/\cfk_{\bochi})^\times \times (\ZZ/\ell \ZZ)} \overline{\bochi(a,i)}  \cdot \left(\langle \frac{a (\theta)\alpha(\theta)}{\cfk_{\bochi}(\theta)},\frac{q^{i+c}}{q^\ell-1}\rangle - 
\langle \frac{q^{i+c}}{q^\ell-1}\rangle_{\ari} \right) \cdot \overline{\bochi(1,\deg \cfk_{\bochi})} \\
= & 
\left(\St(\frac{\alpha(\theta)}{\cfk_{\bochi}(\theta)},\frac{q^{c}}{1-q^\ell})\ \bigg|\ \bochi\right)_{\cfk_{\bochi},\ell} \cdot \overline{\bochi(1,\deg \cfk_{\bochi})}
= L_\eA(0,\overline{\bochi}) \cdot \bochi(\alpha,c) & \text{(by \eqref{eqn: generating-chi})}
\end{align*}
as desired.
\end{proof}

Given a CM field $\eK$ over $\ek$ which is contained in $\eK_{\nfk,\ell}$, let $\eG_\eK$ be the Galois group $\gal(\eK/\ek)$.
For a generalized CM type $\Xi$ of $\eK$,
from the bijection~\eqref{eqn: JK-coset}
we may write
\begin{equation}\label{eqn: Xi-varrho}
\Xi = \sum_{\varrho \in \eG_\eK}m_{\varrho} \cdot \xi_\varrho,
\end{equation}
where for each $\varrho \in \eG_\eK$, $\xi_\varrho \in J_\eK$ is the point corresponding to the embedding $\nu_1\circ \varrho$.
Let $\eK^+$ be the maximal totally real subfield of $\eK$ and put $\eG_{\eK^+} = \gal(\eK^+/\ek)$.
By Remark~\ref{rem: S(G)-chi}, we may express  
\[
\varphi_{\eK,\Xi} =\frac{1}{\#(\eG_{\eK})}\cdot \left(\sum_{\varrho \in \eG_\eK}m_\varrho\right) +  \frac{1}{\#(\eG_{\eK})}\cdot \sum_{\bochi \in \widehat{\eG}_{\eK}\setminus \widehat{\eG}_{\eK^+}}\left(\sum_{\varrho \in \eG_\eK}m_\varrho \overline{\bochi(\varrho)}\right) \cdot \bochi.
\]
Notice that $\sum\limits_{\varrho \in \eG_\eK} m_\varrho = [\eK^+:\ek] \cdot \wt(\Xi)$, and for $\bochi \in \widehat{\eG}_\eK \setminus \widehat{\eG}_{\eK^+}$ we have (by Remark~\ref{rem: L-vanish} and Lemma~\ref{lem: ST-chi}):
\[
\bochi = \sum_{\substack{a \in (\eA/\cfk_{\bochi})^\times \\ i \in \ZZ/\ell \ZZ}}\frac{\overline{\bochi(a,i+\deg \cfk_{\bochi})}}{L_\eA(0,\overline{\bochi})} \cdot \St(\frac{a(\theta)}{\cfk_{\bochi}(\theta)},\frac{q^i}{1-q^\ell}).
\]
Therefore
\begin{align*}
\varphi_{\eK,\Xi} &= 
\frac{\wt(\Xi)}{[\eK:\eK^+]}
+ \frac{1}{[\eK:\ek]}
\sum_{\bochi \in \widehat{\eG}_{\eK}\setminus \widehat{\eG}_{\eK^+}}\!\!
\left(\sum_{\varrho \in \eG_\eK}m_\varrho \overline{\bochi}(\varrho)\right)\!\! \cdot \!\!
\sum_{\substack{a \in (\eA/\cfk_{\bochi})^\times \\ i \in \ZZ/\ell \ZZ}}\frac{\overline{\bochi}(a,i+\deg \cfk_{\bochi})}{L_\eA(0,\overline{\bochi})}\cdot \St(\frac{a(\theta)}{\cfk_{\bochi}(\theta)},\frac{q^i}{1-q^\ell}) \\
&=
\frac{\wt(\Xi)}{[\eK:\eK^+]}
+ \frac{1}{[\eK:\ek]}  \sum_{\cfk\mid \nfk}\sum_{\substack{a \in (\eA/\cfk)^\times \\ i \in \ZZ/\ell \ZZ}} \left(\sum_{\varrho \in \eG_\eK} m_\rho \!\!\sum_{\subfrac{\bochi \in \widehat{\eG}_\eK\setminus \widehat{\eG}_{\eK^+}}{\cfk_{\bochi} = \cfk}}\!\!
\frac{\bochi(\varrho)\bochi(a,i+\deg \cfk)}{L_\eA(0,\bochi)} \right) \cdot \St(\frac{a(\theta)}{\cfk(\theta)},\frac{q^i}{1-q^\ell}).
\end{align*}
For each $\varrho \in \eG_\eK$, put
\begin{equation}\label{eqn: defn-nc}
n_\cfk(\varrho,a,i) \assign \sum_{\subfrac{\bochi \in \widehat{\eG}_\eK\setminus \widehat{\eG}_{\eK^+}}{\cfk_{\bochi} = \cfk}}
\frac{\bochi(\varrho)\bochi(a,i+\deg \cfk)}{L_\eA(0,\bochi)}, \quad \text{ which lies in $\QQ$.}
\end{equation}
Then 
\[
\varphi_{\eK,\Xi} = 
\frac{\wt(\Xi)}{[\eK:\eK^+]}
+ \frac{1}{[\eK:\ek]}
\sum_{\cfk \mid \nfk}\sum_{\substack{a \in (\eA/\cfk)^{\times} \\ i\in \ZZ/\ell\ZZ}}
\left(\sum_{\varrho \in \eG_\eK}m_\varrho \cdot  n_\cfk(\varrho,a,i)\right)
\cdot \St(\frac{a(\theta)}{\cfk(\theta)},\frac{q^i}{1-q^\ell}).
\]
In general, we obtain that:

\begin{proposition}\label{prop: phi-ST}
Let $\nfk \in \eA_+$ and $\ell \in \NN$.
Given a CM field $\eK$ contained in $\eK_{\nfk,\ell}$ and a generalized CM type $\Xi$ of $\eK$, write $\Xi = \sum_{\varrho \in \eG_\eK}m_\varrho \xi_\varrho$ as in \eqref{eqn: Xi-varrho}.
For each $\varrho_0 \in \eG_\eK$, we have the following expression for the Stickelberger function $\varphi_{\eK,\Xi^{\varrho_0}}$:
\[
\varphi_{\eK,\Xi^{\varrho_0}} = 
\frac{\wt(\Xi)}{[\eK:\eK^+]}
+ \frac{1}{[\eK:\ek]}
\sum_{\cfk \mid \nfk}\sum_{\substack{a \in (\eA/\cfk)^\times \\ i\in \ZZ/\ell\ZZ}}
\left(\sum_{\varrho \in \eG_\eK}m_\varrho \cdot  n_\cfk(\varrho\varrho_0,a,i)\right)
\cdot \St(\frac{a(\theta)}{\cfk(\theta)},\frac{q^i}{1-q^\ell}).
\]
\end{proposition}

\begin{proof}
Suppose $\varrho_0 \in \eG_\eK$ is given.
As
\[
\Xi^{\varrho_0} = \sum_{\varrho \in \eG_\eK}m_\varrho \xi_{\varrho}^{\varrho_0}
=\sum_{\varrho \in \eG_\eK}m_\varrho \xi_{\varrho\varrho_0}
= 
\sum_{\varrho \in \eG_\eK}m_{\varrho\varrho_0^{-1}}\xi_\varrho,
\]
we obtain that 
\begin{align*}
\varphi_{\eK,\Xi^{\varrho_0}} 
&=\frac{\wt(\Xi)}{[\eK:\eK^+]}
+ \frac{1}{[\eK:\ek]}
\sum_{\cfk \mid \nfk}\sum_{\substack{a \in (\eA/\cfk)^\times \\ i\in \ZZ/\ell\ZZ}}
\left(\sum_{\varrho \in \eG_\eK}m_{\varrho\varrho_0^{-1}} \cdot  n_\cfk(\varrho,a,i)\right)
\cdot \St(\frac{a(\theta)}{\cfk(\theta)},\frac{q^i}{1-q^\ell})\\
&= 
\frac{\wt(\Xi)}{[\eK:\eK^+]}
+ \frac{1}{[\eK:\ek]}
\sum_{\cfk \mid \nfk}\sum_{\substack{a \in (\eA/\cfk)^\times \\ i\in \ZZ/\ell\ZZ}}
\left(\sum_{\varrho \in \eG_\eK}m_\varrho \cdot  n_\cfk(\varrho\varrho_0,a,i)\right)
\cdot \St(\frac{a(\theta)}{\cfk(\theta)},\frac{q^i}{1-q^\ell}).
\end{align*}
\end{proof}

\begin{remark}\label{rem: class number}
Let $\eK$ be a CM field and $\eK^+$ its maximal totally real subfield.
Suppose that $\eK$ is contained in $\eK_{\nfk,\ell}$ for some $\nfk \in \eA_+$ and $\ell \in \NN$.
It is known that (cf.\ \cite[Theorem~14.4]{Rosen})
\[
\prod_{\bochi \in \widehat{\eG}_\eK \setminus \widehat{\eG}_{\eK^+}}L_{\eA}(0,\bochi) = \frac{\#\Pic(O_\eK)\cdot \Rcal(O_\eK)}{w_\eK \cdot \#\Pic(O_{\eK^+})\cdot \Rcal(O_{\eK^+})},
\]
where:
\begin{itemize}
    \item $\Pic(O_\eK)$ (resp.\ $\Pic(O_{\eK^+})$) is the ideal class group of $O_\eK$ (resp.\ $O_{\eK^+}$);
    \item $w_\eK \assign \#(\FF_\eK^\times)/\#(\FF_q^\times)$, and $\FF_\eK$ is the constant field of $\eK$;
    \item $\Rcal(O_\eK)$ (resp.\ $\Rcal(O_{\eK^+})$) is the regulator of $O_\eK$ (resp.\ $O_{\eK^+}$), see \cite[p.~245]{Rosen}.
\end{itemize}
Let $h(\eK/\eK^+)$ be the following ``relative class number'':
\[
h(\eK/\eK^+)\assign 
\frac{\#\Pic(O_\eK)\cdot \Rcal(O_\eK)}{ \#\Pic(O_{\eK^+})\cdot \Rcal(O_{\eK^+})} \quad \in \NN.
\]
Then for $\cfk \in \eA_+$ with $\cfk \mid \nfk$, $a \in (\eA/\cfk)^\times$, $i \in \ZZ/\ell\ZZ$, and $\varrho \in \eG_\eK$, the expression of $\nfk_{\cfk}(\varrho,a,i)$ in \eqref{eqn: defn-nc} tells us in particular that
\[
h(\eK/\eK^+)\cdot \nfk_{\cfk}(\varrho,a,i) \ \in \ZZ.
\]
\end{remark}

Proposition~\ref{prop: phi-ST} leads us to the following Chowla--Selberg formula for quasi-periods of CM abelian $t$-modules:

\begin{theorem}\label{thm: ex-qp-CM}\label{thm: ex-HB-period}
Let $E_\rho$ be a CM abelian $t$-module with generalized CM type $(\eK,\Xi)$ over $\bar{k}$, where $\eK$ is contained in $\eK_{\nfk,\ell}$ for some $\nfk \in \eA_+$ and $\ell \in \NN$.
Let $\eG_\eK = \gal(\eK/\ek)$ and write
\[
\Xi = \sum_{\varrho \in \eG_\eK}m_\varrho \xi_\varrho \ \in I_\eK^0 \quad \text{ as in \eqref{eqn: Xi-varrho}}.
\]

(1) The space of quasi-periods of $E_\rho$ is spanned over $\bar{k}$ by
\[
\tilde{\pi}^{\frac{\wt(\Xi)}{[\eK:\eK^+]}}
\cdot  \prod_{\varrho \in \eG_\eK} \left(\prod_{\cfk\mid \nfk}\prod_{\substack{a \in (\eA/\cfk)^\times \\ i \in \ZZ/\ell \ZZ}} \tilde{\Gamma}(\frac{a(\theta)}{\cfk(\theta)},\frac{q^i}{1-q^\ell})^{n_{\cfk}(\varrho\varrho_0,a,i)}\right)^{\frac{m_{\varrho}}{[\eK:\ek]}}, \quad \text{ for } \varrho_0 \in \eG_\eK,
\]
where $n_\cfk(\varrho,a,i)$ is defined in \eqref{eqn: defn-nc}.

(2) Suppose $\Xi$ is a CM type of $\eK$.
Let $\eK^+$ be the maximal totally real subfield of $\eK$, and put $d = [\eK^+:\ek]$.
Write $\Xi = \xi_{\varrho_1}+\cdots + \xi_{\varrho_d}$, and let $\nu_{\varrho_j}\assign (\nu_1\big|_\eK) \circ \varrho_j$ for $1\leq j \leq d$.
There exist an ideal $\Ifk_\rho$ of $O_\eK$ and a suitable $\bar{k}$-linear isomorphism  $\Lie(E_\rho) \cong \bar{k}^d$ so that the
image of every period $\lambda \in \Lambda_\rho \subset \Lie(E_\rho)(\CC_\infty)$ in $\CC_\infty^d$ has the form $(\nu_{\varrho_1}(\alpha)\lambda_1,...,\nu_{\varrho_d}(\alpha)\lambda_d) \in \CC_\infty^d$,
where $\alpha \in \Ifk_\rho$ and for $1\leq j \leq d$,
\[
\lambda_j \assign \tilde{\pi}^{\frac{1}{[\eK:\eK^+]}}
\cdot  \prod_{j'=1}^d \prod_{\cfk\mid \nfk}\prod_{\substack{a \in (\eA/\cfk)^\times \\ i \in \ZZ/\ell\ZZ}} \tilde{\Gamma}(\frac{a(\theta)}{\cfk(\theta)},\frac{q^i}{1-q^\ell})^{\frac{n_{\cfk}(\varrho_{j'}\varrho_{j}^{-1},a,i)}{[\eK:\ek]}}.
\]

(3) Suppose further that $E_\rho$ is a Drinfeld 
$\eA$-module over $\ok$ with CM by $O_\eK$, where $O_\eK$ is the integral closure of $\eA$ in an imaginary field $\eK$.
Then
\[
\lambda^{[\eK:\ek]} \sim  
\tilde{\pi}
\cdot \prod_{\cfk\mid \nfk}\prod_{\substack{a \in (\eA/\cfk)^\times \\ i\in \ZZ/\ell\ZZ}} \tilde{\Gamma}(\frac{a(\theta)}{\cfk(\theta)},\frac{q^i}{1-q^\ell})^{n_{\cfk}(a,i)}
\quad \text{ for every nonzero period } \lambda \in \Lambda_\rho \subset \CC_\infty,
\]
where 
\[n_{\cfk}(a,i) \assign n_{\cfk}(\text{\rm id}_{\eK},a,i)
= \sum_{\subfrac{\bochi \in \widehat{\eG}_\eK\setminus \widehat{\eG}_{\eK^+}}{\cfk_{\bochi} = \cfk}}
\frac{\bochi(a,i+\deg \cfk)}{L_\eA(0,\bochi)}.
\]
Moreover,
\[
\varpi_\eK^{\varrho}\assign \tilde{\pi}^{\frac{1}{[\eK:\ek]}}
\cdot \prod_{\cfk\mid \nfk}\prod_{\substack{a \in (\eA/\cfk)^\times \\ i\in \ZZ/\ell\ZZ}} \tilde{\Gamma}(\frac{a(\theta)}{\cfk(\theta)},\frac{q^i}{1-q^\ell})^{\frac{n_{\cfk}(\varrho,a,i)}{[\eK:\ek]}},
\quad \text{ for } \varrho \in \eG_\eK,
\]
are algebraically independent over $\bar{k}$ and
form a $\bar{k}$-basis of the space of quasi-periods of the CM Drinfeld module $E_\rho$.
\end{theorem}

\begin{proof}
For $\varrho_0 \in \eG_\eK$, take
\[
\ba_{\eK,\Xi^{\varrho_0}} \assign 
\sum_{\cfk \mid \nfk}\sum_{\substack{a \in (\eA/\cfk)^{\times} \\ i\in \ZZ/\ell\ZZ}}
\left(\sum_{\varrho \in \eG_\eK}m_\varrho \cdot  n_\cfk(\varrho\varrho_0,a,i)\right)
\cdot [\frac{a(\theta)}{\cfk(\theta)},\frac{q^i}{1-q^\ell}] \quad \in \Acal_{\nfk,\ell,\QQ}^{\cyc}.
\]
Proposition~\ref{prop: phi-ST} implies that
\begin{equation}\label{eqn: phi-ST}
\varphi_{\eK,\Xi^{\varrho_0}} = \frac{\wt(\Xi)}{[\eK:\eK^+]}\cdot \mathbf{1}_{\eG_\eK} + \frac{1}{[\eK:\ek]}\St(\ba_{\eK,\Xi^{\varrho_0}}).
\end{equation}
Theorem~\ref{thm: qp-CM} implies that the space of quasi-period of $E_\rho$ is spanned over $\ok$ by
\begin{align*}
\wp_{\nu_1}^{\cyc}(\varphi_{\eK,\Xi^{\varrho_0}}) & \sim \wp_{\nu_1}^{\cyc}\left(\frac{\wt(\Xi)}{[\eK:\eK^+]}\cdot \mathbf{1}_{\eG_\eK} + \frac{1}{[\eK:\ek]}\St(\ba_{\eK,\Xi^{\varrho_0}})\right) \\
& \sim \tilde{\pi}^{\frac{\wt(\Xi)}{[\eK:\eK^+]}}
\cdot  \prod_{\varrho \in \eG_\eK} \left(\prod_{\cfk\mid \nfk}\prod_{\substack{a \in (\eA/\cfk)^\times \\ i \in \ZZ/\ell \ZZ}} \tilde{\Gamma}(\frac{a(\theta)}{\cfk(\theta)},\frac{q^i}{1-q^\ell})^{n_{\cfk}(\varrho\varrho_0,a,i)}\right)^{\frac{m_{\varrho}}{[\eK:\ek]}}, \quad \text{ for } \varrho_0 \in \eG_\eK.
\end{align*}
This completes the proof of (1).

Moreover, from the expression \eqref{eqn: phi-ST},
(2) and (3) follow directly from Theorem~\ref{thm: HB-period} and Theorem~\ref{thm: CSP}, respectively.
\end{proof}

\begin{example}\label{ex: CSF-1}
(\emph{Chowla--Selberg formula for constant field extensions}.)
Given $\ell \in \NN$, put $\eK = \eK_{1,\ell} = \FF_{q^\ell}(t)$.
Then
\[
\text{$\eG_\eK =\gal(\eK/\ek)\cong \gal(\FF_{q^\ell}/\FF_q) \cong \ZZ/\ell \ZZ$ \quad as in \eqref{eqn: ari-Artin},}
\]
and $\cfk_{\bochi} = 1$ for every $\bochi \in \widehat{\eG}_\eK$.
Put $\zeta_\ell \assign \exp(2\pi\sqrt{-1}/\ell) \in \CC^\times$.
Recall that we write $\bochi(1,1) = q^{-s_{\bochi}} = \zeta_\ell^{d_{\bochi}}$ for each $\bochi \in \widehat{\eG}_\eK$.
For $i \in \ZZ/\ell \ZZ$, we calculate
\[
n_1(1,i) = \sum_{1\neq \bochi \in \widehat{\eG}_\eK}\frac{\bochi(1,i)}{L_{\eA}(0,\bochi)}
= \sum_{d=1}^{\ell-1} \zeta_\ell^{id}(1-q\zeta_\ell^d)
= q-1+ \begin{cases}
\ell, &\text{ if $i = 0$,}\\
-\ell q, &\text{ if $i=\ell-1$,}\\
0, &\text{ otherwise.}
\end{cases}
\]
Let $E_\rho$ be a Drinfeld $\eA$-module of rank $\ell$ over $\ok$ with CM by $\FF_{q^\ell}[t]$.
Theorem~\ref{thm: ex-qp-CM}~(3) implies that
for every nonzero period $\lambda \in \Lambda_\rho$, we have that
\begin{align*}
\lambda & \sim \tilde{\pi}^{\frac{1}{\ell}} \cdot \left(\prod_{i=0}^{\ell-1}\tilde{\Gamma}(0,\frac{q^i}{1-q^\ell})^{\frac{q-1}{\ell}}\right) \cdot \left(\tilde{\Gamma}(0,\frac{1}{1-q^\ell})\Big/\tilde{\Gamma}(0,\frac{q^{\ell-1}}{1-q^\ell})^q\right) \\
& \sim \Gamma_{\ari}(1-\frac{q^{\ell-1}}{q^\ell-1})^q\Big/\Gamma_{\ari}(1-\frac{1}{q^\ell-1}),
\end{align*}
where the last equivalence comes from the monomial relations among gamma values in Proposition~\ref{prop: AR-Gamma}~(2) and (5).
In other words, the first statement of Theorem~\ref{thm: ex-qp-CM}~(3) implies Thakur's analogue of Chowla--Selberg formula in \cite[Theorem~1.6]{Thakur91}.
Moreover, 
the second statement of Theorem~\ref{thm: ex-qp-CM}~(3) shows the result in \cite[Theorem~3.3.2]{CPTY10} that
\[
\Gamma_{\ari}\Big(1-\frac{q^{c}}{q^\ell-1}\Big)^q\Big/\Gamma_{\ari}\Big(1-\frac{q^{c+1}}{q^\ell-1}\Big), \quad  c = 0,...,\ell-1,
\]
are algebraically independent over $\bar{k}$ and
form a $\bar{k}$-basis of the space of quasi-periods of $E_\rho$.
\end{example}

\begin{example}\label{ex: CSF-geo}
(\emph{Chowla--Selberg formula for the $t$-th Carlitz cyclotomic function field.})
Let $\eK = \eK_{t,1} = \ek(\sqrt[q-1]{-t})$, which is a geometric extension over $\ek$.
Then 
\[
\eG_\eK \cong \left(\frac{\eA}{(t)}\right)^{\times} \cong \FF_q^\times,
\quad 
\cfk_{\bochi} = t,
\text{ and } \quad 
L_\eA(0,\bochi) = 1,
\quad \forall \  1 \neq \bochi \in \widehat{\eG}_\eK.
\]
Thus for $\epsilon \in \FF_q^\times$, we get
\[
n_t(\epsilon,0) = \sum_{1\neq \bochi \in \widehat{\eG}_\eK}\bochi(\epsilon,0) 
= -1+
\begin{cases}
q-1, & \text{ if $\epsilon = 1$,}\\
0, & \text{ otherwise.}
\end{cases}
\]
Let $E_\rho$ be a Drinfeld $\eA$-module of rank $q-1$ over $\ok$ with CM by $\FF_q[\sqrt[q-1]{-t}]$.
Theorem~\ref{thm: ex-HB-period}~(3) implies that for every nonzero period $\lambda \in \Lambda_\rho$, we have that 
\[
\lambda \sim \tilde{\pi}^{\frac{1}{q-1}} \cdot 
\left(\prod_{\epsilon \in \FF_q^\times}\Gamma(\frac{\epsilon}{\theta},1-\frac{1}{q-1})^{\frac{-1}{q-1}}\right)
\cdot \Gamma(\frac{1}{\theta},1-\frac{1}{q-1})
\sim \Gamma_{\geo}(\frac{1}{\theta}),
\]
where the last equivalence comes from Proposition~\ref{prop: AR-Gamma}~(5).
This matches Thakur's formula in \cite[Theorem~4.11.2]{Thakur04}.
Moreover, the second statement of Theorem~\ref{thm: ex-HB-period}~(3) shows that
\[
\Gamma_{\geo}(\frac{\epsilon}{\theta}),\quad \epsilon \in \FF_q^\times,
\]
form a $\ok$-basis of the space of quasi-periods of $E_\rho$ (coinciding with \cite[Theorem~6.4.7~(2)]{Thakur04} when $q=3$), and that they are algebraically independent over $\ok$.
\end{example}

\begin{example}\label{ex: CSF-2}
(\emph{Chowla--Selberg formula for imaginary quadratic extensions.})
Here we apply Theorem~\ref{thm: ex-qp-CM}~(3) to present an analogue of the Chowla--Selberg formula for CM Drinfeld $\eA$-modules of rank $2$ over $\bar{k}$.
Note that by the class field theory, an imaginary quadratic extension $\eK$ over $\ek$ is contained in $\eK^{\cyc}$ if and only if the ramification index of the infinite place $\infty$ of $\ek$ in $\eK$ divides $q-1$.
In this case, let $\dfk \in \eA_+$ be the generator of the discriminant ideal of $O_\eK/\eA$ (where $O_\eK$ is the integral closure of $\eA$ in $\eK$).
Then $\eK$ is actually contained in $\eK_{\dfk,2}$, and $\eG_{\eK}$ is a group of order $2$.

Let $\bochi_\eK:\eG_\eK \rightarrow \{\pm 1\}$ be the unique non-trivial character in $\widehat{\eG}_\eK$.
Then $\cfk_{\bochi_\eK} = \dfk$, and for $a \in (\eA/\dfk)^\times$ and $i \in \ZZ/2\ZZ$, we get
\[
n_{\dfk}(\varrho,a,i) = \frac{\bochi_\eK(a,i+\deg \dfk)}{L_{\eA}(0,\bochi)}\cdot 
\begin{cases}
1, & \text{ if $\varrho =\mathrm{id}_{\eK}$,}\\
-1,& \text{ otherwise.}
\end{cases}
\]
Since $L_\eA(0,\bochi) = \#\Pic(O_\eK)/w_\eK$ (see Remark~\ref{rem: class number}), the first statement of Theorem~\ref{thm: ex-qp-CM}~(3) implies that for every nonzero period $\lambda$ of a Drinfeld $\eA$-module of rank $2$ over $\ok$ with CM by $O_\eK$, we have
\[
\lambda^{2\#\Pic(O_\eK)} \sim \tilde{\pi}^{\#\Pic(O_\eK)} \cdot \prod_{\substack{a \in (\eA/\dfk)^\times \\ i \in \ZZ/2\ZZ}}\Gamma(\frac{a(\theta)}{\dfk(\theta)},1-\frac{q^i}{q^2-1})^{w_\eK \bochi(a,i+\deg \dfk)}.
\]
Moreover, put
\begin{equation}\label{eqn: pi-K-pm}
\varpi_{\eK}^{\pm} \assign 
\sqrt{\tilde{\pi}} \cdot \prod_{(a,i) \in (\eA/\dfk)^\times\times \ZZ/2\ZZ} \Gamma\Big(\frac{a(\theta)}{\dfk(\theta)},1-\frac{q^i}{q^2-1}\Big)^{\pm w_\eK \frac{\bochi_\eK(a,i+\deg \dfk)}{2\#\Pic(O_\eK)}}.
\end{equation}
The second statement of Theorem~\ref{thm: ex-qp-CM}~(3) says that the space of quasi-periods of $E_\rho$ is spanned over $\ok$ by $\{\omega_\eK^+,\omega_\eK^-\}$, and $\omega_\eK^+,\omega_\eK^-$ are algebraically independent over $\ok$.
\end{example}

\begin{remark}
Here we use the notation in Example~\ref{ex: CSF-2}.
Assume further that $\bochi_\eK(1,i) = 1$ for any $i \in \ZZ/2\ZZ$.
Put $\chi_\eK \assign \bochi_\eK\big|_{(\eA/\dfk)^\times}$.
Then $\eK \subset \eK_{\dfk,1}$, which is a Carlitz cyclotomic function field.
In particular, the infinite place $\infty$ of $\ek$ is totally ramified in $\eK$.
These ensure that $q$ is odd, $\dfk$ is square-free with $\deg \dfk$ odd, $\eK = \ek(\sqrt{-\dfk})$, and $w_\eK = 1$.
In this case, we may write
\[
\varpi_{\eK}^{\pm}
=
\sqrt{\tilde{\pi}} \cdot
\prod_{a \in (\eA/\dfk)^\times}
\Gamma_{\geo}\left(\frac{a(\theta)}{\dfk(\theta)}\right)^{\pm \frac{\chi_\eK(a)}{2\#\Pic(O_\eK)}},
\]
and get a precise analogue of the classical Chowla--Selberg formula described by Gross in \cite[equation~(2) and (3)]{Gr78}.
\end{remark}

\subsection{The normalization of gamma values and a Lerch-type formula}

Given a $p$-adic integer $y = \sum_{i=0}^\infty y_i q^i \in \ZZ_p$, set (see \cite[3.6]{Thakur91})
\[
\partial(y) \assign \sum_{i=0}^\infty i y_i q^i \quad \in p\ZZ_p.
\]
In particular, let $\ell \in \NN$.
For each integer $c$ with $0\leq c<\ell$, one checks that
\[
\partial(\frac{q^c}{1-q^\ell}) = \frac{cq^c}{1-q^\ell} + \frac{\ell q^{\ell+c}}{(1-q^\ell)^2} \quad \in \QQ.
\]
Moreover, given a character $\bochi: \eG_{1,\ell} \cong \ZZ/\ell \ZZ \rightarrow \CC^\times$, we have the following ``$L'$-evaluator'' property:
writing $\bochi(1) =q^{-s_{\bochi}}$ where $s_{\bochi} = -\frac{2\pi \sqrt{-1}}{\ln q} \cdot \frac{d_{\bochi}}{\ell}$ for a unique integer $d_{\bochi}$ with $0\leq d_{\bochi}<\ell$,
\begin{align*}
 &\ \ln q \cdot \sum_{c=0}^{\ell-1}\bochi(c) \cdot \partial(\frac{q^c}{1-q^\ell}) \\
= & \ \ln q \cdot 
\left[ 
\frac{1}{1-q^\ell} \sum_{c=0}^{\ell-1}cq^{c(1-s_{\bochi})} + \frac{\ell q^\ell}{(1-q^\ell)^2} \sum_{c=0}^{\ell-1}q^{c(1-s_{\bochi})}
\right] \\
= & \
\ln q \cdot 
\left[ 
\frac{1}{1-q^\ell} \cdot\left( \frac{(1-q^{\ell(1-s_{\bochi})}) \cdot q^{1-s_{\bochi}}}{(1-q^{1-s_{\bochi}})^2} - \frac{\ell q^{\ell(1-s_{\bochi})}}{1-q^{1-s_{\bochi}}}\right)
+ \frac{\ell q^\ell}{(1-q^\ell)^2} \cdot \frac{1-q^{\ell(1-s_{\bochi})}}{1-q^{1-s_{\bochi}}}
\right] \\
= &\ \ln q \cdot \frac{q^{1-s_{\bochi}}}{(1-q^{1-s_{\bochi}})^2} \ = \ -L_\eA'(0,\bochi).
\end{align*}
For each $y \in \ZZ_{(p)}$, we follow \cite[3.6]{Thakur91} to normalize $\Gamma_{\ari}(y)$ by setting
\begin{equation}
    \Gamma^*_{\ari}(y) \assign \theta^{\partial(y-1)}\cdot \Gamma_{\ari}(y) = \theta^{\partial(y-1)} \cdot \Pi_{\ari}(y-1)
\end{equation}
to obtain
\begin{lemma}
For $\ell \in \NN$ and $\bochi \in \widehat{\eG}_{1,\ell}$ we have that
\[
L_\eA'(0,\bochi) = -\sum_{c=0}^{\ell-1}\bochi(c)\cdot \ln\left|\Gamma_{\ari}^*(1-\frac{q^c}{q^\ell-1})\right|_\infty.
\]
\end{lemma}

\begin{remark}
One consequence of the above lemma is that
\begin{align}
\ln \left|\tilde{\pi} \cdot \prod_{i=0}^{\ell-1}\Gamma_{\ari}^*(1-\frac{q^i}{q^\ell-1})^{-n_1(1,i)}\right|_\infty
&= \frac{q \ln q}{q-1} - \ln q \cdot \sum_{i=0}^{\ell-1} n_1(1,i)\cdot \partial(\frac{q^i}{1-q^\ell})  \nonumber \\
&= \frac{\zeta_\eA'(0)}{\zeta_\eA(0)} + \sum_{1\neq \bochi \in \widehat{\eG}_{1,\ell}}\frac{L_\eA'(0,\bochi)}{L_\eA(0,\bochi)} \nonumber \\
&= \frac{\zeta_{O_{1,\ell}}'(0)}{\zeta_{O_{1,\ell}}(0)}. \nonumber 
\end{align}
\end{remark}

In order to extend the above result to arbitrary cyclotomic function fields, we define
\[
\partial(x,y) \assign \langle x,-y\rangle, \quad \forall x \in k,\ y \in \ZZ_{(p)}.
\]
Then for each character $\bochi \in \widehat{\eG}_{\nfk,\ell}$ with $\cfk_{\bochi} \neq 1$, from the expression of $L_A^{\nfk}(s,\bochi)$ in \eqref{eqn: L-relation} the following ``$L'$-evaluator'' property holds:
\[
-\!\!\!\! \sum_{\substack{a \in \eA,\ \deg a<\deg \nfk \\ \text{gcd}(a,\nfk) = 1} }\  \sum_{i=0}^{\ell-1} \bochi(a,i+\deg \nfk)\cdot \partial(\frac{a(\theta)}{\nfk(\theta)},\frac{q^i}{1-q^\ell})\cdot \ln\left|\frac{a(\theta)}{\nfk(\theta)}\right|_\infty = (L_{\eA}^{\nfk})'(0,\bochi) + L_{\eA}^{\nfk}(0,\bochi)\cdot \ln|\nfk(\theta)|_\infty.
\]
We normalize the two-variable gamma values $\Gamma(x,y)$ as follows:
\begin{equation}\label{eqn: gamma-normalization}
    \Gamma^*(x,y) \assign \frac{x^{-\partial(x,y-1)}\cdot \Pi_{\geo}(x,y-1)}{\Gamma_{\ari}^*(x,y)} = x^{1-\partial(x,y-1)} \cdot \theta^{-\partial(y-1)} \cdot  \Gamma(x,y).
\end{equation}
In particular, $\Gamma^*(0,y) = \Gamma^*_{\ari}(y)^{-1}$ for all $y \in \ZZ_{(p)}$.
Then we have the following Lerch-type formula:

\begin{theorem}\label{thm: Lerch-type formula}
Let $\nfk \in \eA_+$ and $\ell \in \NN$.
Given a character $\bochi \in \eG_{\nfk,\ell}$, the following equality holds:
\[
L_\eA'(0,\bochi) = -L_\eA(0,\bochi)\ln|\cfk_{\bochi}|_\infty + \sum_{\substack{a \in \eA,\ \deg a<\deg \cfk_{\bochi} \\ \text{\rm gcd}(a,\cfk_{\bochi}) = 1}} \ \sum_{i=0}^{\ell-1}\bochi(a,i+\deg \cfk_{\bochi})\cdot \ln\left|\Gamma^*(\frac{a(\theta)}{\cfk_{\bochi}(\theta)},1-\frac{q^i}{q^\ell-1})\right|_\infty.
\]
\end{theorem}

Consequently, Theorem~\ref{thm: Lerch-type formula} leads us to:

\begin{corollary}\label{cor: log-der}
Let $\nfk \in \eA_+$ and $\ell \in \NN$.
Given an imaginary field $\eK \subset \eK_{\nfk,\ell}$, we have that
\[
\ln\left|\tilde{\pi}\cdot \prod_{\cfk \mid \nfk} \prod_{\substack{a \in \eA,\ \deg a < \deg \cfk \\ \text{\rm gcd}(a,\cfk) = 1}} \ 
\prod_{i=0}^{\ell-1}
\Gamma^*(\frac{a(\theta)}{\cfk(\theta)},1-\frac{q^i}{q^\ell-1})^{n_{\cfk}(a,i)}\right|_\infty 
= \frac{\zeta_{O_\eK}'(0)}{\zeta_{O_\eK}(0)} + \ln|\dfk(O_\eK/\eA)|_\infty,
\]
where $\dfk(O_\eK/\eA) \in \eA_+$ is the monic generator of the discriminant ideal of $O_\eK$ over $\eA$.
\end{corollary}

\begin{proof}
From the formula of $\tilde{\pi}$ in Remark~\ref{rem: Carlitz period} one gets
\[
\ln|\tilde{\pi}|_\infty=\frac{q\ln q}{q-1} = \frac{\zeta_{\eA}'(0)}{\zeta_{\eA}(0)}.
\]
As $\eK$ is imaginary, one has that $\eK^+=\ek$.
Thus for $\cfk \in \eA_+$ with $\cfk\mid \nfk$, $a \in \eA$ with $\deg a < \deg c$ and $\text{gcd}(a,c) = 1$, and $0\leq i < \ell$,
\[
n_\cfk(a,i) = \sum_{\substack{1\neq \bochi \in \widehat{\eG}_\eK\\ \cfk_{\bochi}=\cfk}} \frac{\bochi(a,i+\deg \cfk)}{L_{\eA}(0,\bochi)}.
\]
Therefore we obtain that
\begin{align*}
& \ln\left|\tilde{\pi}\cdot \prod_{\cfk \mid \nfk} \prod_{\substack{a \in \eA,\ \deg a < \deg \cfk \\ \text{gcd}(a,\cfk) = 1}} \ 
\prod_{i=0}^{\ell-1}
\Gamma^*(\frac{a(\theta)}{\cfk(\theta)},1-\frac{q^i}{q^\ell-1})^{n_{\cfk}(a,i)}\right|_\infty \\
=\ & 
\frac{\zeta_{\eA}'(0)}{\zeta_{\eA}(0)}+ \sum_{\cfk \mid \nfk} \sum_{\substack{a \in \eA,\ \deg a<\deg \cfk \\ \text{gcd}(a,\cfk)=1}} \sum_{i=0}^{\ell-1}n_\cfk(a,i) \cdot \ln \left|\Gamma^*(\frac{a(\theta)}{\cfk(\theta)},1-\frac{q^i}{q^\ell-1})\right|_\infty \\
=\ &
\frac{\zeta_{\eA}'(0)}{\zeta_{\eA}(0)} + \sum_{1\neq \bochi \in \widehat{\eG}_\eK}
\sum_{\substack{a \in \eA,\ \deg a<\deg \cfk_{\bochi} \\ \text{gcd}(a,\cfk_{\bochi}) = 1}} \ \sum_{i=0}^{\ell-1}\frac{\bochi(a,i+\deg \cfk_{\bochi})}{L_\eA(0,\bochi)}\cdot \ln\left|\Gamma^*(\frac{a(\theta)}{\cfk_{\bochi}(\theta)},1-\frac{q^i}{q^\ell-1})\right|_\infty \\
=\ &
\frac{\zeta_{\eA}'(0)}{\zeta_{\eA}(0)} + \sum_{1\neq \bochi \in \widehat{\eG}_\eK} \left(\frac{L_\eA'(0,\bochi)}{L_\eA(0,\bochi)} + \ln|\cfk_{\bochi}|_\infty\right) \hspace{2cm}  \text{(by Theorem~\ref{thm: Lerch-type formula})} \\
=\ &
\frac{\zeta_{O_\eK}'(0)}{\zeta_{O_\eK}(0)} + \ln|\dfk(O_\eK/\eA)|_\infty,
\end{align*}
where the last equality comes from the fact that $\zeta_{O_\eK}(s) = \prod_{\bochi \in \widehat{\eG}_\eK}L_\eA(s,\bochi)$ and the conductor-discriminant formula.
\end{proof}

\begin{remark}
Let $\eK$ be an imaginary field contained in $\eK_{\nfk,\ell}$ for $\nfk \in \eA_+$ and $\ell \in \NN$.
Take $\nu\assign \nu_1\big|_{\eK} :\eK \hookrightarrow \ok\subset \CC_\infty$, and let $K\assign \nu(\eK)$ (resp.\ $O_K\assign \nu(O_\eK)$).
Put $r = [\eK:\ek]$.
Every fractional ideal of $O_K$ becomes a discrete $A$-lattice of rank $r$ in $\CC_\infty$.
The Kronecker limit formula established in \cite{Wei} implies in particular that (see \cite[(5.3)]{Wei})
\[
\frac{\zeta_{O_\eK}'(0)}{\zeta_{O_\eK}(0)} = \frac{1}{\#\Pic(O_K)}\sum_{[\Afk] \in \Pic(O_K)} \ln\Big(N(\Afk)|\Delta(\Afk)|_\infty^{\frac{r}{q^r-1}}\Big),
\]
where
\begin{itemize}
    \item $N(\Afk) \assign \#(O_K/\Afk)$ for each nonzero ideal $\Afk$ of $O_K$; and
    \item $\Delta(\cdot)$ is the \emph{Drinfeld discriminant function} (on $A$-lattices of rank $r$ in $\CC_\infty$, see \cite[p.~880]{Wei}).
\end{itemize}
Therefore Corollary~\ref{cor: log-der} leads us to the following ``analytic'' Chowla--Selberg formula:
\begin{align}\label{eqn: A-CSF}
    & \frac{1}{\#\Pic(O_K)}\sum_{[\Afk] \in \Pic(O_K)} \ln\Big(N(\Afk)|\Delta(\Afk)|_\infty^{\frac{r}{q^r-1}}\Big) \nonumber \\
    &= \ln \left(\frac{|\tilde{\pi}|_\infty}{|\dfk(O_\eK/\eA)|_\infty}\right) + 
    \sum_{\cfk \mid \nfk} \sum_{\substack{a \in \eA,\ \deg a < \deg \cfk \\ \text{gcd}(a,\cfk) = 1}} \ 
    \sum_{i=0}^{\ell-1}
n_{\cfk}(a,i)\ln \Big|\Gamma^*(\frac{a(\theta)}{\cfk(\theta)},1-\frac{q^i}{q^\ell-1})\Big|_\infty.
\end{align}
\end{remark}

\begin{remark}
Let $E_\rho$ be a Drinfeld $\eA$-module over $\ok$ of rank $r$ which has CM by $O_\eK$, where $\eK$ is an imaginary field of degree $r$ over $\ek$. 
Suppose that $\eK$ is separable over $\ek$ and tamely ramified at $\infty$.  
Then the Colmez-type formula derived in \cite[Theorem 1.6]{Wei} (together with \cite[Remark~1.7~(3)]{Wei}) expresses the ``stable Taguchi height of $E_\rho$'' as follows:
\[
h_{\mathrm{Tag}}^{\mathrm{st}}(E_\rho) = -\frac{1}{2r}\cdot \ln |\dfk(O_\eK/\eA)|_\infty - \frac{1}{r}\cdot \frac{\zeta_{O_\eK}'(0)}{\zeta_{O_\eK}(0)}.
\]
When $\eK \subset \eK_{\nfk,\ell}$ for $\nfk \in \eA_+$ and $\ell \in \NN$,
\eqref{eqn: A-CSF} implies that
\begin{align}\label{eqn: height}
h_{\mathrm{Tag}}^{\mathrm{st}}(E_\rho) &= \frac{1}{2r}\cdot \ln |\dfk(O_\eK/\eA)|_\infty-\frac{1}{r}\ln|\tilde{\pi}|_\infty \nonumber \\
&-\sum_{\cfk \mid \nfk} \sum_{\substack{a \in \eA,\ \deg a < \deg \cfk \\ \text{gcd}(a,\cfk) = 1}} \ 
\sum_{i=0}^{\ell-1}\frac{n_{\cfk}(a,i)}{r}\cdot \ln \Big|\Gamma^*(\frac{a(\theta)}{\cfk(\theta)},1-\frac{q^i}{q^\ell-1})\Big|_\infty.
\end{align}
Suppose further that $O_\eK$ is a principal ideal domain and $E_\rho$ is ``sgn-normalized'' (see \cite[Chapter~7]{Goss98}).
Then the period lattice of $E_\rho$ is of the form:
\[
\Lambda_\rho = O_K \cdot \varpi_\rho,
\]
where $\varpi_\rho\in \CC_\infty$ is unique up to $O_K^\times$-multiples.
We can show that
\[
h_{\mathrm{Tag}}^{\mathrm{st}}(E_\rho)
= -\frac{1}{2r}\cdot \ln |\dfk(O_\eK/\eA)|_\infty
- \ln |\varpi_\rho|_\infty.
\]
Combining this with \eqref{eqn: height}, we arrive at the following ``absolute'' Chowla--Selberg formula:
\begin{equation}\label{eqn: A-CSF-2}
|\varpi_\rho|_\infty^r =  \frac{|\tilde{\pi}|_\infty}{|\dfk(O_\eK/\eA)|_\infty} \cdot \prod_{\cfk \mid \nfk}\prod_{\substack{a \in \eA,\ \deg a < \deg \cfk \\ \text{gcd}(a,\cfk) = 1}} \ 
\prod_{i=0}^{\ell-1}\left|\Gamma^*(\frac{a(\theta)}{\cfk(\theta)},1-\frac{q^i}{q^\ell-1})\right|_\infty^{n_{\cfk}(a,i)}.
\end{equation}
\end{remark}

\section{The Deligne--Gross period conjecture}

In this section, we shall apply the work of Pink and Hartl--Juschka (in \cite{Pink}, \cite{HP04}, \cite{Jus10}, and \cite{HJ20}) on the Hodge conjecture over function fields, to derive an analogue of the Deligne--Gross period conjecture.

\subsection{Hodge--Pink structures}

\begin{definition}
Let $H$ be a finite dimensional vector space over $\ek$. An \emph{exhaustive and separated increasing $\QQ$-filtration} $W_\bullet H$ is a collection of $\ek$-subspaces $W_\mu H \subset H$ for $\mu \in \QQ$ with $W_{\mu'} H \subset W_{\mu} H$ whenever $\mu'<\mu$, such that the associated graded space 
\[
\Gr^W H \assign \bigoplus_{\mu \in \QQ} \Gr^W_\mu H, 
\quad \text{ where } \quad \Gr^W_\mu H \assign W_\mu H/ (\cup_{\mu'<\mu} W_{\mu'}H),
\]
has the same dimension over $\ek$ as $H$.
\end{definition}

\begin{definition}
(See \cite[Definition~2.2.3]{HJ20}.) A \emph{(mixed) $\ek$-pre-Hodge--Pink structure} is a triple $\eH = (H,W_\bullet H,\qfk)$, where 
\begin{itemize}
    \item $H$ is a finite dimensional vector space over $\ek$;
    \item $W_\bullet H$ is an exhaustive and separated increasing $\QQ$-filtration; 
    \item $\qfk \subset \CC_\infty(\!(t-\theta)\!)\otimes_\ek H$ is a $\CC_\infty[\![t-\theta]\!]$-lattice of full rank.
\end{itemize}
Here the tensor product is with respect to the following inclusion
\[
\ek \subset \CC_\infty(t) = \CC_\infty(t-\theta) \subset \CC_\infty(\!(t-\theta)\!).
\]
In particular, $\eH$ is called \emph{pure (of weight $\mu$)} if $H \cong \Gr^W_{\mu} H$ for some $\mu \in \QQ$.

Given two (mixed) $\ek$-pre-Hodge--Pink structures $\eH = (H,W_\bullet H, \qfk)$ and $\eH' = (H', W_\bullet H', \qfk')$, a \emph{morphism} $f:\eH \rightarrow \eH'$ is a $\ek$-linear homomorphism $f: H \rightarrow H'$ such that $f(W_\mu H) \subset W_\mu H'$ for every $\mu \in \QQ$ and the induced homomorphism $\mathrm{id}\otimes f : \CC_\infty(\!(t-\theta)\!)\otimes_\ek H \rightarrow \CC_\infty(\!(t-\theta)\!)\otimes_\ek H'$ satisfies $(\mathrm{id}\otimes f)(\qfk) \subseteq \qfk'$.
\end{definition}

\begin{remark}\label{rem: completion embedding}
In fact, $\ek$ is contained in  $\CC_\infty[\![t-\theta]\!]$, as every nonzero $a \in \eA$ is contained in $\CC_\infty[\![t-\theta]\!]^\times$.
In particular,
\[
\frac{1}{t} =
\big(\theta + (t-\theta)\big)^{-1} = \sum_{n=0}^\infty (-1)^n\cdot (\frac{1}{\theta})^{n+1} \cdot (t-\theta)^n.
\]
It is known that (see \cite[p.~39 and Lemma~2.1.3]{HJ20}) the inclusion $\ek\subset \CC_\infty[\![t-\theta]\!]$ can be extended to a $\ek$-algebra embedding $\ek_\infty \hookrightarrow \CC_\infty[\![t-\theta]\!]$.
Thus it is natural to extend the above definition to \emph{(mixed) $\ek_\infty$-pre-Hodge--Pink structures} and the morphisms between them without difficulty.
Every $\ek$-pre-Hodge--Pink structure $\eH = (H,W_\bullet H,\qfk)$ can be extended to a $\ek_\infty$-pre-Hodge--Pink structure
$\eH_\infty = (H_\infty, W_\bullet H_\infty,\qfk)$ where $H_\infty \assign \ek_\infty \otimes_\ek H$ and $W_\mu H_\infty \assign \ek_\infty \otimes_\ek W_\mu H$ for every $\mu \in \QQ$.
\end{remark}

\begin{definition}
A (mixed) $\ek$-Hodge--Pink structure is a (mixed) $\ek$-pre-Hodge--Pink structure $\eH = (H,W_\bullet H,\qfk)$
satisfying the \emph{local semistable condition (at $\infty$)}: for every $\ek_\infty$-subspace $H_\infty'$ of $H_\infty$, put $W_\mu H_\infty' \assign H_\infty' \cap W_\mu H_\infty$ for every $\mu \in \QQ$, $\qfk' \assign \qfk \cap \CC_\infty(\!(t-\theta)\!)\otimes_{\ek_\infty} H'_\infty$, and $\pfk' \assign \CC_\infty[\![t-\theta]\!]\otimes_{\ek_\infty} H_\infty'$.
Then
\[
\dim_{\CC_\infty}\left(\frac{\qfk'}{\qfk'\cap \pfk'}\right) - \dim_{\CC_\infty}\left(\frac{\pfk'}{\qfk'\cap \pfk'}\right)\leq \sum_{\mu \in \QQ} \mu \cdot \dim_{\ek_\infty} \Gr^W_\mu H'_\infty,
\]
and equality holds when $H'_\infty = \ek_\infty \otimes_\ek W_\mu H$ for every $\mu \in \QQ$.
\end{definition}

\begin{remark}
The local semistable condition can be also defined on $\ek_\infty$-pre-Hodge--Pink structures, and we may say as well that a (mixed) $\ek_\infty$-pre-Hodge--Pink structure is a (mixed) $\ek_\infty$-Hodge--Pink structure if it satisfies the local semistable condition.
In particular, if $\eH$ is a $\ek$-Hodge--Pink structure, then $\eH_\infty$ is a $\ek_\infty$-Hodge--Pink structure.
\end{remark}

Given two (mixed) $\ek$-Hodge--Pink structures $\eH = (H,W_\bullet H, \qfk)$ and $\eH' = (H', W_\bullet H', \qfk')$,
let $f: \eH\rightarrow \eH'$ be a morphism.
The local semistable condition on $\eH$ and $\eH'$ ensures that $f$ must be \emph{strict} (see \cite[Remark~2.2.11]{HJ20}), i.e.
\[
f(W_\mu H) = f(H) \cap W_\mu H', \quad \forall \mu \in \QQ,
\quad \text{ and } \quad 
(\mathrm{id}\otimes f)(\qfk) = \CC_\infty(\!(t-\theta)\!)\otimes_\ek f(H) \cap \qfk'.
\]
We denote by $\Hom_{\mathrm{HP}}(\eH,\eH')$ the space of the morphisms from $\eH$ to $\eH'$, and $\End_{\mathrm{HP}}(\eH)$ the endomorphism ring of $\eH$.

\begin{remark}
Let $\eH = (H,W_\bullet H,\qfk)$ be a (mixed) $\ek$-Hodge--Pink structure.
A \emph{$\ek$-Hodge--Pink sub-structure of $\eH$} is a triple $\eH' =(H',W_\bullet H', \qfk')$, where
$H'$ is a $\ek$-vector suspace of $H$,
$W_\mu H' = H' \cap W_\mu H$ for every $\mu \in \QQ$, and
$\qfk' = \qfk \cap \CC_\infty(\!(t-\theta)\!) \otimes_\ek H'$.
In particular, if $\eH$ is pure of weight $\mu$, then so is $\eH'$.
\end{remark}

\begin{example}\label{ex: H(M)}
One natural example of a pure $\ek$-Hodge--Pink structure comes from pure uniformizable dual $t$-motives.
Let $\eM$ be a pure uniformizable dual $t$-motive of weight $\wt(\eM)$ over $\CC_\infty$. Recall in Section~\ref{sec: defn PS} that 
$H^\eM = \Hom_{\eA}(H_{\mathrm{Betti}}(\eM),\ek)$.
Set
\[
\qfk^\eM \assign \Hom_{\CC_\infty[t]}\big(\sigma \eM,\CC_\infty[\![t-\theta]\!]\big)
\quad \text{ and }\quad 
W_\mu H^\eM \assign
\begin{cases} 
0, & \text{ if $\mu<-\wt(\eM)$,}\\
H^\eM, &\text{ if $\nu\geq -\wt(\eM)$.}
\end{cases}
\]
By \cite[Theorem~2.4.32 and 2.3.34~(a)]{HJ20},
the triple $\eH(\eM)\assign (H^\eM,W_\bullet H^\eM,\qfk^\eM)$ is a pure Hodge--Pink structure of weight $\mu_\eM$, where
\[
\mu_\eM =-w(\eM)= \frac{\rank_{\CC_\infty[\sigma]}\eM}{\rank_{\CC_\infty[t]}\eM}.
\]
Moreover, given two pure uniformizable dual $t$-motives $\eM$ and $\eM'$ over $\CC_\infty$, every morphism $f:\eM\rightarrow \eM'$ (over $\CC_\infty$) induces a morphism
\[
f_{\mathrm{HP}}: \eH(\eM')\rightarrow \eH(\eM),
\]
and $(f\mapsto f_{\mathrm{HP}})$ gives a $\ek$-linear (resp.\ $\ek$-algebra) isomorphism (resp.\ when $\eM = \eM'$) (see \cite[Theorem~2.4.32 and 2.3.34~(b)]{HJ20})
\[
\ek \otimes_\eA \Hom_{\CC_\infty[t,\sigma]}(\eM,\eM') \cong \Hom_{\mathrm{HP}}(\eH(\eM'),\eH(\eM)).
\]
\end{example}

The following fact comes from the function field analogue of the Hodge conjecture for dual $t$-motives (see \cite[Theorem~2.4.32 and 2.3.34~(c)]{HJ20}):

\begin{theorem}\label{thm: HP-conj}
Let $\eM$ be a pure uniformizable dual $t$-motive over $\CC_\infty$.
Given a $\ek$-Hodge--Pink sub-structure $\eH'$ of $\eH(\eM)$, there exists a pure uniformizable dual $t$-motive $\eM'$ over $\CC_\infty$ together with an essentially surjective morphism $f: \eM \rightarrow \eM'$ such that the corresponding morphism $f_{\mathrm{HP}}: \eH(\eM')\hookrightarrow \eH(\eM)$ gives an isomorphism $\eH(\eM')\cong \eH'$.
\end{theorem}

\begin{remark}
It is worth pointing out that
the ``effectiveness'' of $\eM'$ (i.e.\ $\eM'$ is indeed a $\CC_\infty[\sigma]$-module) comes from the ``duality'' in \cite[Proposition~2.4.3]{HJ20}, and the essential surjectivity of the morphism $f: \eM\rightarrow \eM'$ implies that $\eM'$ is finitely generated over $\CC_\infty[\sigma]$.
Thus $\eM'$ is indeed a dual $t$-motive in our setting.
\end{remark}

\subsection{\texorpdfstring{CM Hodge--Pink structures and CM dual $t$-motives}{CM Hodge--Pink structures and CM dual t-motives}}

Let $\eH = (H,W_\bullet H, \qfk)$ be a pure $\ek$-Hodge--Pink structure of weight $w(\eH)$.
Put $\pfk \assign \CC_\infty[\![t-\theta]\!]\otimes_{\ek} H$, and identify
\[
\frac{\pfk}{(t-\theta) \pfk} \cong \CC_\infty \underset{v_\theta,\ek}{\otimes} H \rassign H_{\CC_\infty}.
\]
The \emph{Hodge--Pink filtration} of $H_{\CC_\infty}$ is $F^\bullet H_{\CC_\infty} \assign (F^i H_{\CC_\infty})_{i\in\ZZ}$, where $F^i H_{\CC_\infty}$ is the subspace of $H_{\CC_\infty}$ corresponding to 
\[
\frac{\pfk\cap (t-\theta)^i \qfk}{(t-\theta)\pfk\cap (t-\theta)^i \qfk} \quad \cong \quad \frac{\pfk \cap (t-\theta)^i \qfk+(t-\theta) \pfk}{(t-\theta)\pfk} \quad \subset \quad \frac{\pfk}{(t-\theta)\pfk}.
\]
Let $r = \dim_{\ek} H$.
Take an integer $n$ sufficiently large so that $(t-\theta)^n\pfk \subset \qfk$.
There exist $r$ integers $w_1,...,w_r$ satisfying that
\[
\frac{\qfk}{(t-\theta)^n \pfk} \cong \bigoplus_{i=1}^r \frac{\CC_\infty[\![t-\theta]\!]}{(t-\theta)^{n+w_i}},
\]
which is independent of the choice of $n$.
We call $w_1,...,w_r$ the \emph{Hodge--Pink weights of $\eH$}.
In particular, the local semistability of $\eH$ implies that 
\[
r \cdot w(\eH) = w_1+\cdots + w_r.
\]

Suppose further that $\eH$ has \emph{full-CM} by a CM field $\eK$, i.e.\ $\dim_\ek H = [\eK:\ek]$ and there exists a $\ek$-algebra homomorphism
\begin{equation}\label{eqn: HCM}
\eK \hookrightarrow \End_{\mathrm{HP}}(\eH).
\end{equation}
Then $\pfk$ becomes a free module of rank one over $\CC_\infty[\![t-\theta]\!]\otimes_{\ek} \eK$.
Recall that $O_\eK$ denotes the integral closure of $\eA$ in $\eK$, and here we put $O_{\bK} = \CC_\infty \otimes_{\FF_q}O_\eK$.
Observe that
\[
\CC_\infty[\![t-\theta]\!] \otimes_{\ek} \eK \cong \CC_\infty[\![t-\theta]\!] \otimes_{\eA} O_\eK \cong \prod_{\xi \in J_\eK} \hat{O}_{\bK,\xi}
\quad \text{and}\quad
\CC_\infty(\!(t-\theta)\!) \otimes_{\ek} \eK \cong
\prod_{\xi \in J_\eK} \hat{\bK}_{\xi},
\]
where for each $\xi \in J_\eK$, $\hat{O}_{\bK,\xi}$ is the completion of $O_{\bK}$ with respect to the maximal ideal $\Pfk_\xi$ and $\hat{\bK}_{\xi}$ is the field of fraction of $\hat{O}_{\bK,\xi}$.
As $\qfk$ is also a $\CC_\infty[\![t-\theta]\!] \otimes_{\ek} \eK$-module and there exists a sufficiently large integer $n$ so that 
\[(t-\theta)^n\cdot \pfk \subset \qfk \subset (t-\theta)^{-n} \cdot \pfk \quad \subset \CC_\infty(\!(t-\theta)\!) \otimes_{\ek} H,
\]
we obtain that $\qfk$ is actually a free $\CC_\infty[\![t-\theta]\!] \otimes_{\ek} \eK$-module of rank one, and there exist unique integers $w_\xi$ for $\xi \in J_\eK$ satisfying that
\[
\frac{\qfk}{(t-\theta)^n \pfk} \cong \prod_{\xi \in J_\eK}\frac{\hat{O}_{\bK,\xi}}{\Pfk_\xi^{n+w_\xi}\hat{O}_{\bK,\xi}} \cong \prod_{\xi \in J_\eK}\frac{\CC_\infty[\![t-\theta]\!]}{(t-\theta)^{n+w_\xi}},
\]
where the last isomorphism is due to that 
$(t-\theta)O_\bK = \prod_{\xi \in J_\eK}\Pfk_\xi$ (i.e.\ $(t-\theta)$ splits completely in $O_\bK$).
Thus the numbers $w_\xi$ for $\xi \in J_\eK$ coincide with the Hodge--Pink weights of $\eH$.

\begin{definition}\label{defn: HP-type}
Let $\eH$ be a $\ek$-Hodge--Pink structure with full-CM by a CM field $\eK$ over $\ek$.
The \emph{Hodge--Pink type of $\eH$} is 
\[
\Phi_{\eH} \assign \sum_{\xi \in J_\eK} w_\xi \xi \quad \in I_\eK.
\]
\end{definition}

In fact, we have the following:

\begin{lemma}\label{lem: HP-type in IK0}
Let $\eH$ be a pure $\ek$-Hodge--Pink structure of weight $w(\eH)$ which has full-CM by a CM field $\eK$ over $\ek$.
Then $\Phi_\eH$ lies in $I_\eK^0$ with $\wt(\Phi_\eH) = [\eK:\eK^+]\cdot w(\eH)$,
where $\eK^+$ is the maximal totally real subfield of $\eK$.
\end{lemma}

\begin{proof}
Write $\eH = (H,W_\bullet H,\qfk)$, and denote by $w(\eH)$ be the weight of $\eH$.
Recall that we let $\eH_\infty\assign (H_\infty, W_\bullet H_\infty, \qfk)$ where $H_\infty = \ek_\infty \otimes_\ek H$ and $W_\mu H_\infty = \ek_\infty \otimes_\ek W_\mu H$, which is a $\ek_\infty$-Hodge--Pink structure.
Then we have an embedding
\[
\ek_\infty  \otimes_\ek \eK \hookrightarrow \End_{\mathrm{HP}}(\eH_\infty),
\]
which makes $H_\infty$, $W_\mu H_\infty$ for every $\mu \in \QQ$, and $\qfk$ become $\ek_\infty \otimes_\ek \eK$-modules.
On the other hand, as $\eK$ is a CM field over $\ek$, we have that
\[
\ek_\infty \otimes_{\ek}\eK \cong \prod_{\xi^+ \in J_{\eK^+}} \eK_{\infty,\xi^+},
\]
where for each $\xi^+ \in J_{\eK^+}$, $\eK_{\infty,\xi^+} \assign \ek_\infty \underset{\tilde{\nu}_{\xi^+},\eK^+}{\otimes} \eK$ is the compositum of $\eK$ and $\ek_\infty$ over $\eK^+$ with respect to the embedding
\[
\tilde{\nu}_{\xi^+}: \eK^+\hookrightarrow \ek_\infty \quad \text{ so that } \quad 
v_\theta \circ \tilde{\nu}_{\xi^+} = \nu_{\xi^+} \in 
\Emb(\eK^+,\CC_\infty).
\]
In particular, one has that 
$[\eK_{\infty,\xi^+}:\ek_\infty] = [\eK:\eK^+]$ for every $\xi^+ \in J_{\eK^+}$.

For each $\xi^+ \in J_{\eK^+}$, we take $e_{\xi^+}$ to be the  idempotent of $\ek_\infty \otimes_\ek \eK$ corresponding to $\eK_{\infty,\xi^+}$.
Then $e_{\xi^+} \eH_\infty \assign (e_{\xi^+} H_\infty, W_\bullet (e_{\xi^+} H_\infty), e_{\xi^+} \qfk)$ is a pure $\ek_\infty$-Hodge--Pink sub-structure of $\eH_\infty$, and 
\[
\eH_\infty = \bigoplus_{\xi \in J_{\eK^+}} e_{\xi^+}\eH_{\infty}.
\]
Moreover, comparing with the decomposition
\begin{align*}
\prod_{\xi \in J_\eK}\hat{O}_{\bK,\xi}
& \cong \CC_\infty[\![t-\theta]\!] \otimes_{\ek} \eK 
 \cong
\CC_\infty[\![t-\theta]\!] \otimes_{\ek_\infty} (\ek_\infty \otimes_{\ek} \eK) \\
& \cong 
\prod_{\xi^+ \in J_{\eK^+}} \CC_\infty[\![t-\theta]\!] \otimes_{\ek_\infty} \eK_{\infty,\xi^+} 
\cong 
\prod_{\xi^+ \in J_{\eK^+}} \CC_\infty[\![t-\theta]\!] \underset{\tilde{\nu}_{\xi^+},\eK^+}{\otimes} \eK \\
& \cong 
\prod_{\xi^+ \in J_{\eK^+}} \CC_\infty[\![t-\theta]\!] \underset{\tilde{\nu}_{\xi^+},O_{\eK^+}}{\otimes} O_{\eK}
\cong 
\prod_{\xi^+ \in J_{\eK^+}}
\left(\prod_{\substack{\xi \in J_{\eK} \\ \pi_{\bX/\bX^+}(\xi) = \xi^+}} \hat{O}_{\bK,\xi}\right),
\end{align*}
we get
\[
e_{\xi^+} = \sum_{\substack{\xi \in J_{\eK} \\ \pi_{\bX/\bX^+}(\xi) = \xi^+}} \hat{e}_{\xi},
\quad \text{ where $\hat{e}_{\xi}$ is the
idempotent of $\prod_{\xi \in J_\eK}\hat{O}_{\bK,\xi}$ corresponding to $\hat{O}_{\bK,\xi}$.}
\]
Consequently, taking sufficiently large integer $n$ so that $(t-\theta)^n \pfk \subset \qfk$, we obtain that
\[
\frac{e_{\xi^+}\qfk}{(t-\theta)^n e_{\xi^+}\pfk}\cong
\prod_{\substack{\xi \in J_\eK \\ \pi_{\bX/\bX^+}(\xi) = \xi^+}} \frac{ \hat{O}_{\eK,\xi}}{\Pfk_{\xi}^{n+w_\xi} \hat{O}_{\eK,\xi}}.
\]
The local semistable condition of $\eH$ (at $\infty$) ensures that
\begin{align}\label{eqn: wt-equality}
\sum_{\substack{\xi \in J_\eK \\ \pi_{\bX/\bX^+}(\xi) = \xi^+}} w_\xi 
& =  \dim_{\CC_\infty}\left(\frac{e_{\xi^+}\qfk}{e_{\xi^+}\qfk \cap e_{\xi^+}\pfk}\right) - \dim_{\CC_\infty}\left(\frac{e_{\xi^+}\pfk}{e_{\xi^+}\qfk \cap e_{\xi^+}\pfk}\right) \nonumber \\
& \leq \dim_{\ek_\infty} (e_{\xi^+} H_\infty) \cdot w(\eH)  = [\eK:\eK^+] \cdot w(\eH) .
\end{align}
Since
\[
\sum_{\xi^{+} \in J_{\eK}^{+}} \left(\sum_{\substack{\xi \in J_\eK \\ \pi_{\bX/\bX^+}(\xi) = \xi^+}} w_\xi\right)= \sum_{\xi \in J_\eK} w_\xi = \dim_{\ek} (H) \cdot w(\eH)  = \sum_{\xi^{+} \in J_{\eK}^{+}} \bigg([\eK:\eK^+] \cdot w(\eH) \bigg),
\]
the inequality~\eqref{eqn: wt-equality} is actually an equality, whence $\Phi_{\eH} \in I_\eK^0$ with $\wt(\Phi_{\eH}) = [\eK:\eK^+] \cdot w(\eH) $ as desired.
\end{proof}

\begin{remark}\label{rem: effective}
Let $\eH=(H,W_\bullet H,\qfk)$ be a pure $\ek$-Hodge--Pink structure of weight $w(\eH)$ which has full-CM by a CM field $\eK$ over $\ek$. 
Suppose $\eH$ is \emph{effective}, i.e.\ 
\[
\pfk = \CC_\infty[\![t-\theta]\!] \otimes_{\ek} H \subset \qfk.
\]
Then $w_\xi \geq 0$ for every $\xi \in J_\eK$, for which $\Phi_\eH$ becomes a generalized CM type of $\eK$ with
\[
\wt(\Phi_\eH) = [\eK:\eK^+]\cdot w(\eH).
\]
\end{remark}

\begin{proposition}\label{prop: CM-HP}
Let $\eM$ be a pure uniformizable dual $t$-motive of weight $w(\eM)$ over $\CC_\infty$.
Given a $\ek$-Hodge--Pink sub-structure $\eH'$ of $\eH(\eM)$, suppose that $\eH'$ has full-CM by a CM field $\eK$ over $\ek$.
Then $\eH'$ is effective and we may take $\eM'$ in Theorem~\ref{thm: HP-conj} to be a CM dual $t$-motive with generalized CM type $(\eK,\Phi_{\eH'})$ over $\CC_\infty$, where $\wt(\Phi_{\eH'}) = -[\eK:\eK^+]\cdot w(\eM)$.
\end{proposition}

\begin{proof}
Recall that $\eH(\eM) = (H^{\eM},W_\bullet H^{\eM},\qfk^\eM)$ is a pure $\ek$-Hodge--Pink structure of weight $-w(\eM)$, and
\begin{align*}
\qfk^\eM = \Hom_{\CC_\infty[t]}(\sigma \eM, \CC_\infty[\![t-\theta]\!]) &\supset \Hom_{\CC_\infty[t]}(\eM,\CC_\infty[\![t-\theta]\!])
= \Hom_{\TT^\dagger}(\MM^\dagger,\CC_\infty[\![t-\theta]\!])\\
&
= \Hom_{\eA}(H_{\mathrm{Betti}}(\eM), \CC_\infty[\![t-\theta]\!]) 
= \CC_\infty[\![t-\theta]\!]\otimes_{\ek} H^{\eM} \rassign \pfk^{\eM}.
\end{align*}
Here $\TT^\dagger$ and $\MM^\dagger$ are introduced in Section~\ref{sec: uniformizibility}.
Then $\eH' = (H',W_\bullet H', \qfk')$, where 
$H'$ is a $\ek$-subspace of $H^\eM$, $W_\mu H' = W_\mu H^{\eM}\cap  H'$, and 
\[\qfk' = \qfk \cap \CC_\infty(\!(t-\theta)\!) \otimes_\ek H' \supset \pfk^\eM \cap \CC_\infty(\!(t-\theta)\!) \otimes_\ek H' = \CC_\infty[\![t-\theta]\!] \otimes_\ek H' = \pfk'.
\]
Thus $\eH'$ is pure of weight $-w(\eM)$ and effective.
By Theorem~\ref{thm: HP-conj}, there exists a pure uniformizable dual $t$-motive $\eM'$ over $\CC_\infty$ together with an essentially surjective morphism $f: \eM \rightarrow \eM'$ over $\CC_\infty$ such that $f_{\mathrm{HP}}: \eH(\eM') \cong \eH'$.
As $\eH'$ has full-CM by $\eK$, one has that $\rank_{\CC_\infty[t]}(\eM') = \dim_{\ek}(\eH') = [\eK:\ek]$, and
\[
\eK \hookrightarrow \End_{\mathrm{HP}}(\eH') \cong 
\ek \otimes_{\eA}\End_{\CC_\infty[t,\sigma]}(\eM').
\]
In particular, the intersection of $\eK$ and $\End_{\CC_\infty[t,\sigma]}(\eM')$ via the above embedding is an $\eA$-order $\Ocal$ in $\eK$.
Let $O_\eK$ be the integral closure of $\eA$ in $\eK$.
Replacing $\eM'$ by $O_\eK \otimes_{\Ocal} \eM'$ if necessary, we may assume that $\Ocal = O_\eK$.

From the isomorphism $f_{\mathrm{HP}}: \eH(\eM') \cong \eH'$, we have that
\[
\CC_\infty[\![t-\theta]\!]\otimes_{\CC_\infty[t]} \left(\frac{\eM'}{\sigma \eM'}\right) \cong \frac{\qfk^{\eM'}}{\pfk^{\eM'}} \cong \frac{\qfk'}{\pfk'} \cong \prod_{\xi \in J_\eK}\frac{\hat{O}_{\eK,\xi}}{\Pfk_\xi^{w_\xi'}\hat{O}_{\eK,\xi}},
\]
where $\Phi_{\eH'} = \sum_{\xi \in J_\eK} w_{\xi}' \xi$ is the Hodge--Pink type of $\eH'$.
Let $\Ifk_{\Phi_{\eH'}}\assign \prod_{\xi \in J_\eK} \Pfk^{w_\xi'} \subset O_{\bK}$ to get $\sigma \eM' = \Ifk_{\Phi_{\eH'}} \eM'$.
By Lemma~\ref{lem: HP-type in IK0} and Remark~\ref{rem: effective}, we know that $\Phi_{\eH'}$ is a generalized CM type of $\eK$.
Hence by Definition~\ref{defn: CM dual-t-motive}, $\eM'$ is a CM dual $t$-motive with CM type $(\eK,\Phi_{\eH'})$ over $\CC_\infty$,
and
\[
\wt(\Phi_{\eH'}) = [\eK:\eK^+]\cdot w(\eH') = -[\eK:\eK^+]\cdot w(\eM).
\]
\end{proof}

\begin{remark}\label{rem: field-of-defn}
Keep the notations in Proposition~\ref{prop: CM-HP}.
By Theorem~\ref{thm: CM prop}~(3) we may choose $\eM'$ to be defined over $\ok$.
Suppose further that $\eM$ is also defined over $\ok$, then by Lemma~\ref{lem: morphism-defining-field} the essentially surjective morphism $\eM \rightarrow \eM'$ is over $\ok$ as well.
\end{remark}

\begin{definition}\label{defn: CM-HP}
Let $\eM$ be a pure uniformizable dual $t$-motive over $\ok$, and $\eH'$ a $\ek$-Hodge--Pink substructure of $\eH(\eM)$ with full-CM by a CM field $\eK$ over $\ek$.
A dual $t$-motive $\eM'$ over $\ok$ given in Proposition~\ref{prop: CM-HP} is called a \emph{CM dual $t$-motive associated to $\eH'$ over $\ok$}.
By Theorem~\ref{thm: CM prop}~(3), such $\eM'$ is unique up to isogeny.
\end{definition}

\subsection{The Deligne--Gross conjecture}

Let $\eM$ be a pure uniformizable dual $t$-motive defined over $\ok$.
Given a $\ek$-Hodge--Pink substructure $\eH' = (H', W_\bullet H', \qfk')$ of $\eH(\eM)$, via the period isomorphism $H^{\eM}_{\CC_\infty} \cong H_{\mathrm{dR}}(\eM,\CC_\infty)$ (see \eqref{eqn: period-iso}), one may identify $H'_{\CC_\infty} \assign \CC_\infty \otimes_{\ek} H' \subset H^{\eM}_{\CC_\infty}$ as a subspace of $H_{\mathrm{dR}}(\eM,\CC_\infty)$.
Suppose $\eH'$ has full-CM by a CM field $\eK \subset \ek^{\sep}$, i.e.\ $\dim_\ek H'= [\eK:\ek]$ and $\eK$ embeds into the endomorphism algebra $\End_{\mathrm{HP}}(\eH')$ over $\ek$.
Note that $H'_{\CC_\infty}$ becomes a $\CC_\infty\underset{v_\theta,\ek}{\otimes} \eK$-module, and
\[
\CC_\infty\underset{v_\theta,\ek}{\otimes} \eK \cong \prod_{\nu \in \Emb(\eK,\CC_\infty)} \CC_\infty.
\]
For each embedding $\nu \in \Emb(\eK,\CC_\infty)$,
there exists a nonzero $\omega_{\nu}^{\eH'} \in H_{\mathrm{dR}}(\eM,\CC_\infty) \cap H'_{\CC_\infty}$, which is unique up to $\CC_\infty^\times$-multiples, satisfying that
\[
\alpha (\omega_{\nu}^{\eH'}) = \nu(\alpha) \cdot \omega_{\nu}^{\eH'}, \quad \forall \alpha \in \eK,
\]
The main result in this section is to derive an explicit description of the periods
\[
\int_{\gamma} \omega_{\nu}^{\eH'}, \quad \forall 0\neq \gamma \in H_{\mathrm{Betti}}(\eM),
\]
when $\omega_{\nu}^{\eH'}$ is algebraic (i.e.\  $\omega_{\nu}^{\eH'} \in H_{\mathrm{dR}}(\eM,\ok)\cap H'_{\CC_\infty}$), in terms of two-variable gamma values when the CM field $\eK$ is contained in a cyclotomic function field.\\

Suppose $\eK \subset \eK_{\nfk,\ell}$ where $\nfk \in \eA_+$ and $\ell \in \NN$.
Let $\eM'$ be a CM dual $t$-motive associated to $\eH'$ over $\ok$.
Then $\eM'$ has generalized CM type $(\eK, \Phi_{\eH'})$, where $\Phi_{\eH'} \in I_{\eK}^0$ is the Hodge--Pink type of $\eH'$.
By Proposition~\ref{prop: St-span}~(3), there exists a function $\varepsilon^{\eH'}:(\frac{1}{\nfk(\theta)}A/A) \times (\frac{1}{q^{\ell}-1} \ZZ/\ZZ) \rightarrow \QQ$ such that
\[
\varphi_{\eK,\Phi_{\eH'}} = \sum_{\substack{x \in \frac{1}{\nfk(\theta)}A/A \\ y \in \frac{1}{q^{\ell}-1} \ZZ/\ZZ}}\varepsilon^{\eH'}(x,y)\cdot \St(x,y).
\]
By Theorem~\ref{thm: Gamma dist}, we prove the following analogue of Deligne--Gross period conjecture:

\begin{theorem}\label{thm: DG-conj}
Let $\eM$ be a pure uniformizable dual $t$-motive defined over $\ok$.
Given a $\ek$-Hodge--Pink sub-structure $\eH' = (H',W_\bullet H', \qfk')$ of $\eH(\eM)$, suppose $\eH'$ has full-CM by a CM field $\eK$, where $\eK \subset \eK_{\nfk,\ell}$ for $\nfk \in \eA_+$ and $\ell \in \NN$.
For each $\varrho \in \eG_{\nfk,\ell}$, put $\nu_{\varrho}\assign \nu_1 \circ \varrho\big|_\eK : \eK \hookrightarrow \CC_\infty$.
Identifying $H'_{\CC_\infty} \assign \CC_\infty \otimes_{\ek} H' \subset H^{\eM}_{\CC_\infty}$ with a subspace of $H_{\mathrm{dR}}(\eM,\CC_\infty)$ via the period isomorphism in \eqref{eqn: period-iso}, let  $\omega_{\varrho}^{\eH'} \in H_{\mathrm{dR}}(\eM,\CC_\infty) \cap \eH'_{\CC_\infty}$ be a nonzero differential satisfying
\[
\alpha (\omega_{\varrho}^{\eH'}) = \nu_{\varrho}(\alpha) \cdot \omega_{\varrho}^{\eH'}, \quad \forall \alpha \in \eK.
\]
Take 
$\varepsilon^{\eH'}:(\frac{1}{\nfk(\theta)}A/A) \times (\frac{1}{q^{\ell}-1} \ZZ/\ZZ) \rightarrow \QQ$
such that
\[
\varphi_{\eK,\Phi_{\eH'}}
= \sum_{\substack{x \in \frac{1}{\nfk(\theta)}A/A \\ y \in \frac{1}{q^{\ell}-1} \ZZ/\ZZ}}\varepsilon^{\eH'}(x,y)\cdot \St(x,y) \quad \in \Sscr(\eG_{\nfk,\ell}),
\]
where $\Phi_{\eH'}$ is the Hodge--Pink type of $\eH'$, see Definition~\ref{defn: HP-type}. 
Suppose that $\omega_\varrho^{\eH'}$ is algebraic, i.e.\ $\omega_\varrho^{\eH'} \in H_{\mathrm{dR}}(\eM,\ok)$.
For $\gamma \in \eH_{\mathrm{Betti}}(\eM)$ with $\int_{\gamma} \omega_{\varrho}^{\eH'} \neq 0$, 
we then have that
\[
\int_{\gamma} \omega_{\varrho}^{\eH'}
\sim \prod_{\substack{x \in \frac{1}{\nfk(\theta)}A/A \\ y \in \frac{1}{q^{\ell}-1}\ZZ/\ZZ}}\tilde{\Gamma}(x, y)^{\varepsilon^{\eH'}(\varrho^{-1}\star x,\varrho^{-1}\star y))}.
\]
\end{theorem}

\begin{proof}
Let $\eM'$ be a CM dual $t$-motive associated to $\eH'$ over $\ok$, and $f: \eM \rightarrow \eM'$ the essentially surjective homomorphism over $\ok$ such that $f_{\mathrm{HP}}: \eH(\eM') \cong \eH' \subset \eH(\eM)$ as in Proposition~\ref{prop: CM-HP}.
Let $\xi_{\varrho} \in J_\eK$ be the point so that $\nu_{\xi_{\varrho}} = \nu_\varrho$.
Let $\omega_{\eM',\xi_{\varrho}} \in H_{\mathrm{dR}}(\eM',\ok)$ be the differential associated with $\xi_\varrho$.
Up to a $\ok^\times$-multiple, we then have $\omega_{\varrho}^{\eH'} = f^*\omega_{\eM',\xi_{\varrho}}$.
Note that $\eM'$ has generalized CM type $(\eK,\Phi_{\eH'})$.
Hence for each $\gamma \in H_{\mathrm{Betti}}(\eM)$ such that $\int_{\gamma} \omega_{\varrho}^{\eH'}$ is nonzero, we obtain that 
\begin{align*}
    \int_{\gamma} \omega_{\varrho}^{\eH'} &\sim \int_{\gamma} f^* \omega_{\eM',\xi_{\varrho}} \sim \int_{f_* \gamma } \omega_{\eM',\xi_{\varrho}}\\
    &\sim p_{\eK}(\xi_{\varrho}, \Phi_{\eH'}) \sim p_{\eK}(\xi_1, \Phi_{\eH'}^{\varrho^{-1}}) \sim \wp^{\cyc}_{\nu_1}\big(\varrho \cdot  \varphi_{\eK,\Phi_{\eH'}}\big) 
    \sim \prod_{\substack{x \in \frac{1}{\nfk(\theta)}A/A \\ y \in \frac{1}{q^{\ell}-1}\ZZ/\ZZ}}\wp_{\nu_1}^{\cyc}\big(\St(\varrho\star x,\varrho\star y)\big)^{\varepsilon^{\eH'}(x,y)} \\
    & \sim \prod_{\substack{x \in \frac{1}{\nfk(\theta)}A/A \\ y \in \frac{1}{q^{\ell}-1}\ZZ/\ZZ}}\wp_{\nu_1}^{\cyc}\big(\St(x,y)\big)^{\varepsilon^{\eH'}(\varrho^{-1}\star x,\varrho^{-1} \star y)}
    \sim 
    \prod_{\substack{x \in \frac{1}{\nfk(\theta)}A/A \\ y \in \frac{1}{q^{\ell}-1}\ZZ/\ZZ}}\tilde{\Gamma}(x,y)^{\varepsilon^{\eH'}(\varrho^{-1}\star x,\varrho^{-1} \star y)},
\end{align*}
where the last equivalence is from Theorem~\ref{thm: Gamma dist} (and Remark~\ref{rem: Gamma-dist}).
\end{proof}

\begin{remark}\label{rem: alg-omega}
As $\eM'$ is defined over $\ok$ and $\eH'$ is isomorphic to $\eH(\eM')$ via the essentially surjective morphism $f:\eM\rightarrow \eM'$ over $\ok$, we are always able to choose the differential $\omega_{\varrho}^{\eH'}$ to be algebraic for every $\varrho \in \eG_{\nfk,\ell}$ by taking $\omega_{\varrho}^{\eH'} = f^* \omega_{\eM',\xi_\varrho}$ as in the proof above.
\end{remark}

\begin{remark}\label{rem: ex-epsilon}
Keeping the notation in Theorem~\ref{thm: DG-conj},
write the Hodge--Pink type of $\eH'$ as $\Phi_{\eH'} = \sum_{\varrho \in \eG_\eK}m_{\varrho} \cdot \xi_{\varrho} \in I_\ek^0$.
Although the choice of the function $\varepsilon^{\eH'}$ might not be unique, we can follow Theorem~\ref{thm: ex-HB-period} and take $\varepsilon^{\eH'}$ to be the following explicit one: for $x \in \frac{1}{\nfk(\theta)}A/A$ and $y \in \frac{1}{q^\ell-1}\ZZ/\ZZ$, write $x = \frac{a(\theta)}{\cfk(\theta)} \bmod A$ where $a \in \eA$ and $\cfk \in \eA_+$ with $\cfk \mid \nfk$ and $\text{gcd}(a,\cfk) = 1$, put
\[
\varepsilon^{\eH'}(x,y) \assign \begin{cases}
\displaystyle \frac{(1-q)\wt(\Phi_{\eH'})}{[\eK:\eK^+]} + \frac{m_\varrho n_{\cfk}(\varrho,a,i)}{[\eK:\ek]}, & \text{ if $\cfk = 1$ and $y = \displaystyle \frac{q^i}{1-q^\ell} \bmod \ZZ$ for $0\leq i <\ell$,}\\
\displaystyle\frac{m_\varrho n_{\cfk}(\varrho,a,i)}{[\eK:\ek]}, & \text{ if $\cfk \neq 1$ and $y = \displaystyle \frac{q^i}{1-q^\ell} \bmod \ZZ$ for $0\leq i <\ell$,}\\
0, & \text{ otherwise.}
\end{cases}
\]
Here $n_\cfk(\varrho,a,i)$ is defined in \eqref{eqn: defn-nc}.
As in \eqref{eqn: phi-ST}, we get
\begin{align*}
& \sum_{\substack{x \in \frac{1}{\nfk(\theta)}A/A \\ y \in \frac{1}{q^\ell-1}\ZZ/\ZZ}}\varepsilon^{\eH'}(x,y)\St(x,y) \\
&= \frac{\wt(\Phi_{\eH'})}{[\eK:\eK^+]} \cdot \mathbf{1}_{\eG_\eK}+\sum_{\cfk \mid \nfk}\sum_{\substack{a \in (\eA/\cfk)^\times \\ i \in \ZZ/\ell \ZZ}}\left(\sum_{\varrho \in \eG_\eK} m_\varrho \cdot n_\cfk(\varrho,a,i)\right)\cdot \St(\frac{a(\theta)}{\cfk(\theta)},\frac{q^i}{1-q^\ell}) \\
&= \varphi_{\eK,\Phi_{\eH'}}.
\end{align*}
Hence given $\varrho_0 \in \eG_\eK$,  we obtain that
\[
\int_{\gamma}\omega_{\varrho_0}^{\eH'} \sim 
\tilde{\pi}^{\frac{\wt(\Phi_{\eH'})}{[\eK:\eK^+]}}
\cdot  \prod_{\varrho \in \eG_\eK} \left(\prod_{\cfk\mid \nfk}\prod_{\substack{a \in (\eA/\cfk)^\times \\ i \in \ZZ/\ell \ZZ}} \tilde{\Gamma}(\frac{a(\theta)}{\cfk(\theta)},\frac{q^i}{1-q^\ell})^{n_{\cfk}(\varrho\varrho_0^{-1},a,i)}\right)^{\frac{m_{\varrho}}{[\eK:\ek]}}
\]
whenever $\gamma \in H_{\mathrm{Betti}}(\eM)$ such that $\int_{\gamma}\omega_{\varrho_0}^{\eH'}$ is nonzero.
\end{remark}

\begin{example}\label{ex: Gross-formula}
Let $E_\rho$ be an abelian $t$-module over $\ok$ which is pure and uniformizable (i.e.\ the Hartl--Juschka dual $t$-motive $\eM(\rho)$ is pure and uniformizable).
Put $\eH(E_\rho) \assign \eH(\eM(\rho))$, and for each $N \in \NN$ we define the \emph{$N$-th Hodge--Pink structure of $E_\rho$} by
\[
\eH^N(E_\rho) \assign \bigwedge^N_{\ek} \eH(E_\rho) = \eH\Big(\bigwedge^N_{\ok[t]} \eM(\rho)\Big).
\]

Let $d$ be the dimension of $E_\rho$ and $r$ be the rank of $E_\rho$.
Suppose further that $E_\rho$ has ``complex multiplication'' by $O_\eK$, where $\eK$ is a CM field over $\ek$ with $[\eK:\ek] = n$ and $O_\eK$ is the integral closure of $\eA$ in $\ek$.
This means that there exists an $\eA$-algebra homomorphism from $O_\eK \hookrightarrow \End_{\text{ab.~$t$-mod.}}(E_\rho)$,
which corresponds to a $\ek$-algebra embedding $\eK \hookrightarrow \End_{\mathrm{HP}}(\eH(E_\rho))$.
In particular, one has that $n\mid r$.
Put $d_0\assign r/n$.
We may identify
\[
\eH^{d_0}_{\eK}(\rho)\assign \bigwedge^{d_0}_{\eK} \eH(E_\rho) = \eH\Big(\bigwedge^{d_0}_{O_\bK}\eM(\rho)\Big)
\hookrightarrow
\eH\Big(\bigwedge^{d_0}_{\ok[t]} \eM(\rho)\Big) = \eH^{d_0}(E_\rho).
\]
In particular, the Hodge--Pink structure $\eH^{d_0}_{\eK}(\rho)$ is pure of weight $d_0 \cdot \frac{d}{r} = \frac{d}{n}$ and has full-CM by $\eK$.
This is equivalent to saying that $\bigwedge^{d_0}_{O_\bK}\eM(\rho)$ is a CM dual $t$-motive with generalized CM type $(\eK,\Phi_{\eH^{d_0}_{\eK}(\rho)})$, where $\Phi_{\eH^{d_0}_{\eK}(\rho)}$ is the Hodge--Pink type of $\eH^{d_0}_{\eK}(\rho)$ with
\[
\wt(\Phi_{\eH^{d_0}_{\eK}(\rho)}) = \frac{d}{n} \cdot [\eK:\eK^+] = \frac{d}{[\eK^+:\ek]}.
\]

Now, suppose $\eK \subset \eK_{\nfk,\ell}$ for $\nfk \in \eA_+$ and $\ell \in \NN$.
Write
\[
\Phi_{\eH^{d_0}_{\eK}(\rho)} = \sum_{\varrho \in \eG_{\eK}}m_\varrho \xi_\varrho \quad \in I_\eK^0. 
\]
Then by Remark~\ref{rem: ex-epsilon} we have that for $\varrho_0 \in \eG_\eK$,
\begin{align*}
\int_{\gamma_1\wedge\cdots \wedge\gamma_{d_0}}\omega_{\varrho_0}^{\eH^{d_0}_{\eK}(\rho)} & \sim 
\tilde{\pi}^{\frac{d}{[\eK:\ek]}}
\cdot  \prod_{\varrho \in \eG_\eK} \left(\prod_{\cfk\mid \nfk}\prod_{\substack{a \in (\eA/\cfk)^\times \\ i \in \ZZ/\ell \ZZ}} \tilde{\Gamma}(\frac{a(\theta)}{\cfk(\theta)},\frac{q^i}{1-q^\ell})^{n_{\cfk}(\varrho\varrho_0^{-1},a,i)}\right)^{\frac{m_{\varrho}}{[\eK:\ek]}}
\end{align*}
for 
$\gamma_1\wedge \cdots \wedge \gamma_{d_0} \in H_{\mathrm{Betti}}\big(\bigwedge^{d_0}_{\ok[t]}\eM(\rho)\big)$ so that the period integral on the left hand side is nonzero.
In particular, when $\eK$ is imaginary and quadratic, we may write $\eG_\eK = \{\varrho_1 = \mathrm{id}_\eK,\varrho_2\}$, and
\[
\Phi_{\eH^{d_0}_{\eK}(\rho)} = m_1 \xi_{\varrho_1} + m_2 \xi_{\varrho_2} \quad \text{ with } \quad 
m_1+m_2 = d \quad (\text{as $\eK^+ =\ek$}).
\]
Then we get the following analogue of the Gross formula in \cite[(4),(5) in p.~194]{Gr78}:
\[
\int_{\gamma_1\wedge\cdots \wedge\gamma_{d_0}}\omega_{\varrho_1}^{\eH^{d_0}_{\eK}(\rho)} \sim (\varpi_\eK^+)^{m_1}\cdot (\varpi_{\eK}^-)^{m_2} \sim (\varpi_\eK^+)^{m_1} \cdot \left(\frac{\tilde{\pi}}{\varpi_\eK^+}\right)^{m_2},
\]
and 
\[
\int_{\gamma_1\wedge\cdots \wedge\gamma_{d_0}}\omega_{\varrho_2}^{\eH^{d_0}_{\eK}(\rho)}  \sim (\varpi_\eK^+)^{m_2}\cdot (\varpi_{\eK}^-)^{m_1} \sim (\varpi_\eK^+)^{m_2} \cdot \left(\frac{\tilde{\pi}}{\varpi_\eK^+}\right)^{m_1},
\]
where $\varpi_\eK^{\pm}$ are given in \eqref{eqn: pi-K-pm}.
\end{example}

\bibliographystyle{plain}

\end{document}